\documentclass[a4paper,11pt,twoside]{amsart}
\usepackage[utf8]{inputenc}
\usepackage[english]{babel}
\usepackage{amsmath,amsfonts,amssymb}
\usepackage{graphicx}
\usepackage[top=2.5cm, bottom=2.2cm, left=2.5cm, right=2.5cm]{geometry}
\usepackage{cancel}
\usepackage{float}
\usepackage{tensor}
\usepackage{mathtools}
\usepackage{mathrsfs}
\usepackage{amscd}
\usepackage{graphicx}
\usepackage{keyval}
\usepackage{xcolor}
\usepackage{relsize}
\usepackage{extpfeil}
\usepackage{tikz}
\usepackage{tikz-cd}
\usepackage{verbatim}
\usepackage{nameref}
\usepackage{enumitem}
\usepackage[new]{old-arrows}
\usepackage{url}
\usepackage[hyphenbreaks]{breakurl}
\usetikzlibrary{matrix}
\usepackage{amsthm,lipsum}
\usepackage{dsfont}
\usepackage[backend=biber, bibencoding=utf8]{biblatex}
\usepackage[autostyle,italian=guillemets]{csquotes}
\usepackage[pdfusetitle]{hyperref}
\hypersetup{
    colorlinks=true,
    citecolor=blue,
    linkcolor=blue,
    filecolor=magenta,      
    urlcolor=cyan,
}

\usepackage[mode=multiuser,status=draft,lang=english]{fixme}
\fxsetup{theme=colorsig}
\FXRegisterAuthor{MP}{aP}{MP}
\usepackage{biblatex}
\makeatletter
\renewcommand*\env@matrix[1][*\c@MaxMatrixCols c]{
  \hskip -\arraycolsep
  \let\@ifnextchar\new@ifnextchar
  \array{#1}}
\makeatother

\makeindex 

\theoremstyle{plain}
\newtheorem{theorem}{Theorem}[section]
\newtheorem*{theorem*}{Theorem}
\newcounter{fakecnt}[subsection]

\newtheorem{thm}{Theorem}[fakecnt]

\newtheorem{corollary}[theorem]{Corollary}
\newtheorem{lemma}[theorem]{Lemma}
\newtheorem{proposition}[theorem]{Proposition}
\newtheorem{hypothesis}[theorem]{Hypothesis}
\newtheorem{fact}[theorem]{Fact}

\theoremstyle{definition}
\newtheorem{definition}[theorem]{Definition}
\newtheorem{assumptions}[theorem]{Assumptions}
\newtheorem*{claim*}{Claim}
\newtheorem*{fact*}{Fact}
\newtheorem*{framework*}{Framework}
\newtheorem*{proposition*}{Proposition}

\theoremstyle{remark}
\newtheorem{remark}[theorem]{Remark}

\makeatletter
\def\thm@space@setup{
	\thm@preskip=0.4cm 
	\thm@postskip=\thm@preskip 
}
\makeatother

\newcommand{\Inc}{\mathrm{Inc}}
\newcommand{\Z}{\mathbb{Z}}
\newcommand{\C}{\mathbb{C}}
\newcommand{\Hilb}{\mathrm{Hilb}}
\newcommand{\Spec}{\mathrm{Spec}}
\newcommand{\id}{\mathrm{id}}
\newcommand{\Sym}{\mathrm{Sym}}
\newcommand{\Pic}{\mathrm{Pic}}
\newcommand{\CDiv}{\mathrm{CDiv}}

\usepackage{pdfpages}
\usetikzlibrary{decorations.markings}

\makeatletter
\tikzcdset{
  open/.code     = {\tikzcdset{hook, circled};},
  closed/.code   = {\tikzcdset{hook, slashed};},
  open'/.code    = {\tikzcdset{hook', circled};},
  closed'/.code  = {\tikzcdset{hook', slashed};},
  circled/.code  = {\tikzcdset{markwith = {\draw (0,0) circle (.375ex);}};},
  slashed/.code  = {\tikzcdset{markwith = {\draw[-] (-.4ex,-.4ex) -- (.4ex,.4ex);}};},
  markwith/.code ={
    \pgfutil@ifundefined
    {tikz@library@decorations.markings@loaded}
    {\pgfutil@packageerror{tikz-cd}{You need to say 
      \string\usetikzlibrary{decorations.markings} to use arrows with markings}{}}{}
    \pgfkeysalso{/tikz/postaction = {
      /tikz/decorate,
      /tikz/decoration={markings, mark = at position 0.5 with {#1}}}
    }
  },
}
\makeatother

    \title{\, On the Stable Birationality of Hilbert schemes of points on surfaces\, }
    \author[]{ Morena Porzio$^1$}\thanks{$^1$ Columbia University, 2990 Broadway New York, NY 10027. \texttt{mp3947-at-columbia.edu}. }
    
\begin{document}

\maketitle

\vspace{-5mm}

\begin{abstract}
    The aim of this paper is to study the stable birational type of $\Hilb^n_X$, the Hilbert scheme  
    of degree $n$ points on a surface $X$. 
    More precisely, it addresses the question for which pairs of positive integers $(n,n')$ the variety $\Hilb^n_X$ is stably birational to $\Hilb^{n'}_X$, when $X$ is a surface with irregularity $q(X)=0$.
    After general results for such surfaces, we restrict our attention to geometrically rational surfaces, proving that there are only finitely many stable birational classes among the $\Hilb^n_X$'s.
    As a corollary, we deduce the rationality of the motivic zeta function $\zeta_{\textup{mot}}(X,t)$ in $K_0(Var/k)/([\mathbb{A}^1_k])[[t]]$ over fields of characteristic zero.
\end{abstract}

\begin{center}
    \rule{150mm}{0.2mm}
\end{center}

\setcounter{tocdepth}{1}
\tableofcontents

\section{Introduction and Outline}\label{1}

Let $X$ be a smooth, geometrically connected, projective surface over an arbitrary field $k$. 
For $n \geq 1$, consider the Hilbert scheme $\Hilb^n_X$ of $n$ points on $X$. 
This is a smooth geometrically connected projective scheme of dimension $2n$ over $k$. 
It is birational to the $n$-th symmetric product $\Sym^n_X$ of $X$.

The aim of the present paper is to study for which pairs of integers $n, n'$ the Hilbert schemes $\Hilb^n_X$ and $\Hilb^{n'}_X$ are stably birational.
Notice that the question depends only on the stable birationality type of $X$. 
Indeed:

\begin{fact}[Koll\'ar \cite{Ko}, Cor. 8]\label{factonstablebira}
    Let $X$ and $Y$ be two (positive dimensional) geometrically integral separated schemes of finite type over $k$. 
    If $X$ is stably birational to $Y$, then $\Sym^n_X$ is stably birational to $\Sym^n_Y$ for any $n$.
\end{fact}

Koll\'ar, Litt and Wood studied the cases when $P$ is a Brauer Severi surface (that is $P_{\overline{k}} \cong \mathbb{P}^2_{\overline{k}}$) and when $X$ is a surface with non-negative Kodaira dimension (over a field of characteristic zero).

\begin{theorem*}[\cite{Ko}, Koll\'ar]
    Let $P$ be a Brauer Severi surface. 
    Let $\textup{ind}(P)$ denote the \emph{index} of $P$ $($that is the $\gcd\{$degrees of closed points on $P\})$.
    Then for any $n\geq 0$ the variety $\Hilb^n_P$ is stably birational to $\Hilb^{\gcd(n,\textup{ind}(P))}_P$.
\end{theorem*}

\begin{theorem*}[\cite{Li} Thm. 19, Litt]
    Let $X$ be a smooth, geometrically connected, projective surface over a field of characteristic zero.
    If $H^0(X,\omega_X)\neq 0$ then for any $n\geq 0$ there exists a sufficiently large integer $m(n)$ such that $\Hilb^n_X$ is \emph{never} stably birational to $\Hilb^{n'}_X$ for any $n'\geq m(n)$.
\end{theorem*}

\begin{theorem*}[\cite{Li} Prop. 35, Wood]
    Let $X$ be a smooth, geometrically connected, projective surface over a field of characteristic zero. 
    If the Kodaira dimension of $X$ is non-negative, then for any $n$ such that either $H^0(X,\omega_X^n)$ or $H^0(X,\omega_X^{2n})$ is non-zero, then $\Hilb^n_X$ is stably birational to $\Hilb^{n'}_X$ if and only if $n=n'$.
\end{theorem*}

Given the results, it makes sense to focus on the case when the Kodaira dimension of $X$ is negative, or equivalently $\kappa(X_{\overline{k}})=-\infty$. 
Since Kodaira dimension is a birational invariant for $X_{\underline{k}}$, by the classification of surfaces over algebraically closed fields of arbitrary characteristic (see \cite{Lu}), we have two cases left:
\begin{enumerate}
    \item the surface $X$ is geometrically rational, namely $X_{\overline{k}}$ is birational to $\mathbb{P}^2_{\overline{k}}$, or
    \item the surface $X_{\overline{k}}$ is birational to $\mathbb{P}^1_{\overline{k}}\times_{\overline{k}} C$, where $C$ a smooth curve of positive genus.
\end{enumerate}

In this paper ultimately we want to focus on the first family: geometrically rational surfaces over arbitrary fields.
Since the cohomology group $H^1(X_{\underline{k}},\mathcal{O}_{X_{\underline{k}}})$ is a birational invariant for $X_{\underline{k}}$ (see \cite[Prop. V.3.4]{Harty}), this is equivalent to say that we focus on surfaces of negative Kodaira dimension whose $H^1(X,\mathcal{O}_X)$ is trivial.

Actually, in the first part of this paper, having trivial $H^1(X,\mathcal{O}_X)$ is the only assumption that we put on our surface $X$: only later we specialize to geometrically rational ones.

\subsection{Results}
The input for our study is a triple $(X,\mathcal{L},Q)$ where
\begin{itemize}
    \item $X$ is a smooth, geometrically connected, projective surface over a field $k$, whose $H^1(X,\mathcal{O}_X)$ is trivial, namely a surface of irregularity $q(X)=\dim_k H^1(X,\mathcal{O}_X)=0$,
    \item $\mathcal{L}$ is an ample line bundle on $X$, and 
    \item $Q$ is a geometrically reduced, effective $0$-cycle of $X$ of degree $d$.
\end{itemize}  
We call these $(X,\mathcal{L},Q)$'s \emph{suitable triples}: by extension, if $X$, $\mathcal{L}$ or $Q$ have the above properties we call them \emph{suitable}.

\begin{remark}
    As said above, we do not assume the vanishing of $H^2(X,\mathcal{O}_X)$ from the start.
    We will see that $H^2(X,\mathcal{O}_X)=0$ becomes a necessary condition only when we look at result for $\Hilb^n_X$ with $n>>0$ (cf. theorem \ref{theoremgap1} with proposition \ref{propositionc1(L)KXnegative}).
\end{remark}

\begin{remark}
    Notice that smooth, geometrically connected, projective surfaces $X$ have many suitable $\mathcal{L}$ and $Q$ on them. 
    
    Indeed, if the field $k$ is infinite, we can embed $X$ into a projective space and use smooth hyperplane sections to define reduced effective $0$-cycles on $X$: this is classical Bertini's theorem \cite[Thm. II.8.18]{Harty}.
    If the field $k$ is finite, we can use Poonen's version of Bertini's theorem \cite[Thm. 1.1]{Poonen} and intersect $X$ with hypersurfaces of sufficiently big degrees.

    Therefore, there are many suitable triples $(X,\mathcal{L},Q)$ - the ample line bundles come from $X$ being projective.
\end{remark}

It turns out that if a suitable triple $(X,\mathcal{L},Q)$ satisfies certain assumptions with respect to an integer $n$ - the assumptions are listed at the beginning of Section \ref{5}; note that they depend on the integer $n$ - then we have the following theorem.

\begin{thm}\label{theorem1}
Let $X$ be a suitable surface. 
Then $\Hilb^n_X$ is stably birational to $\Hilb^{n+d}_X$ anytime there exist $(\mathcal{L},Q)$ suitable, $Q$ of degree $d$, such that the following \hyperlink{assumptions}{assumptions} hold true:  
\begin{enumerate}
    \item the restriction map $H^0(X,\mathcal{L}) \longrightarrow H^0(X, \mathcal{L}|_{Q})$ is surjective,
    
    \item  $\dim_k H^0(X, \mathcal{L})\geq n+1+2d$,
    
    \item the sheaf $\mathcal{L} \otimes \mathcal{I}_Q$ is generated by global sections,
    
    \item for any $\overline{k}$-point $x\in X_{\overline{k}}\setminus Q_{\overline{k}}$, we have that
    \[
    H^0(X,\mathcal{L}\otimes \mathcal{I}_Q)_{\overline{k}} \longrightarrow \mathcal{L}_x/m^2_x\mathcal{L}_x
    \]
    is surjective,
    \item for any curve $C$ in the linear system $|\mathcal{L}|$, the arithmetic genus $p_a(C)\leq n$.
\end{enumerate}
\end{thm}

After establishing this general theorem for surfaces of irregularity zero, we focus on finding sufficient and necessary conditions on $n$ so that we have triples satisfying the \hyperlink{assumptions}{assumptions} with respect to $n$. 

\begin{thm}\label{theorem2}
    Let $(X,\mathcal{L},Q)$ be a suitable triple. 
    There exists an integer $e_0=e_0(X,\mathcal{L},Q)$ such that for any $e\geq e_0$ the triple $(X,\mathcal{L}^{\otimes e},Q)$ satisfies the \hyperlink{assumptions}{assumptions} \ref{thefiveassumptions} with respect to $n$ if and only if $n$ is in the interval
    \[
        I_e\coloneqq \left[\frac{e^2 c_1(\mathcal{L})^2 +  e\, c_1(\mathcal{L})\cdot K_X}{2}+1, \frac{ e^2 c_1(\mathcal{L})^2-e\, c_1(\mathcal{L}) \cdot K_X }{2} + \dim_k H^2(X, \mathcal{O}_X) -2\deg(Q) \right] ,
    \]
    in which case the variety $\Hilb^n_X$ is stably birational to $\Hilb^{n+\deg(Q)}_X$.
\end{thm}

Given the nature of theorem \ref{theorem2}, it is natural to wonder about the non-emptiness and the growth of $I_e$ (as $e$ grows).

We already expect that for a surface $X$ of non-negative Kodaira dimension the intervals $I_e$ must be empty for $e>>0$ (Wood's theorem above).
This is confirmed by proposition \ref{propositionc1(L)KXnegative} where we will see that a necessary conditions for having non-empty intervals $I_e$ for $e>>0$ is for $X$ to be geometrically rational.

However, being geometrically rational turns out to be a \emph{sufficient} condition as well.
Actually, it is a sufficient condition for having enough suitable triples $(X,\mathcal{L},Q)$ so that the intervals $I_e$ cover an interval of $+\infty$.

\begin{thm}\label{theoremrationalsurfintro}
    Let $X$ be a geometrically rational suitable surface.
    Let $\textup{ind}(X)$ be the index of $X$.
    Then there exists an integer $n_0$ such that for any $n \geq n_0$ the variety $\Hilb^n_X$ is stably birational to $\Hilb^{n+\textup{ind}(X)}_X$.
\end{thm}

\subsection{Application to the rationality of the motivic zeta function}
We conclude the paper working out the implications of theorem \ref{theoremrationalsurfintro} on the rationality of the \emph{motivic zeta function}.
First introduced by Kapranov \cite{Ka} as a generalization of the Hasse-Weil zeta function $Z(X,t)$, the motivic zeta function is defined as the power series
\[
    \zeta_{\textup{mot}}(X,t) = 1 + \sum_{ n=1}^{\infty} [ \Sym^n_X ] t^n \in K_0(Var/k)[[t]]
\] 
over the Grothendieck ring of $k$-varieties. 

The question about the rationality of the motivic zeta function rose naturally from the theorem about the rationality of the \emph{Hasse-Weil} zeta function (conjectured by Weil in \cite{We2} and proved in general by Dwork in \cite{Dwo}).

However, in general the answer is negative: indeed, as soon as $H^2(X,\mathcal{O}_X)\neq 0$ then Larsen and Lunts in \cite{LL} proved that $\zeta_{\textup{mot}}(X,t)$ is \emph{not} rational.
On the other hand, Kapranov proved that the answer is positive for curves (\cite[Thm. 1.1.9]{Ka}): the same is true for rational surfaces over algebraically closed fields $k=\overline{k}$ of characteristic zero (see for instance Musta\c t\u a \cite[Prop. 4.3]{Mu}).

Modding out by the ideal generated by $[\mathbb{A}^1_k]$, we extend the latter result to geometrically rational surfaces over any field of characteristic zero.

\begin{thm}\label{rationalityofstablmotiviczeta}
    If $k$ is a field of characteristic zero, 
    then the motivic zeta function $\zeta_{\textup{mot}}(X,t)$ of a geometrically rational surface $X$ is rational in the quotient $K_0(Var/k)/([\mathbb{A}^1_k])[[t]]$. 
\end{thm}

\subsection{Outline}
    In Section \ref{2} we recall the definitions and properties of objects that will be later used: results about the Abel-Jacobi map and about Brill-Noether theory.
    In Section \ref{3}  we introduce a geometrical framework in order to compare $\Hilb^n_X$ with $\Hilb^{n'}_X$.
    The idea is to define an appropriate \emph{incidence correspondence} $\Inc^Q_{X,\mathcal{L},n,n'}$ (definition \ref{definitionincidencevariety}) together with 
    two morphisms
    \[
    \begin{tikzcd}
        &
        \Inc^Q_{X,\mathcal{L},n,n'}
        \arrow[dl, "{\pi_n}"']
        \arrow[dr, "{\pi_{n'}}"]
        &
        \\
        \Hilb^n_X 
        &
        &
        \Hilb^{n'}_X.
    \end{tikzcd}
    \]
    The rest of the section is devoted to check that the functor $\Inc^Q_{X,\mathcal{L},n,n'}$ is representable by a quasi-projective $k$-scheme.
    
    In Section \ref{4} we describe the general fiber of $\pi_n$ and $\pi_{n'}$ respectively.
    We will see (in proposition \ref{definitionofUn} and proposition \ref{existenceopenforPE}) that after restricting to appropriate opens of $\Inc^Q_{X,\mathcal{L},n,n'}$ the two maps are projective bundles.

    At this point, it remains to verify that the intersection of those opens in $\Inc^Q_{X,\mathcal{L},n,n'}$ is non-empty.
    Therefore, in Section \ref{5} we introduce the \hyperlink{assumptions}{assumptions} \ref{thefiveassumptions} and in Sections \ref{5} and \ref{6} we prove that under those assumptions the existence of a suitable triple $(X,\mathcal{L},Q)$ implies that $\Hilb^n_X$ is stably birational to $\Hilb^{n+\deg(Q)}_X$ (theorem \ref{theorem1}).
    While doing so, we prove a Bertini's-like theorem for surfaces with a fixed closed (see proposition \ref{bertiniwithclosedfixed}).
    
    In Section \ref{7} we first find sufficient and necessary conditions on $n$ (theorem \ref{theorem2}) for which we can apply theorem \ref{theorem1}.
    Then we study the stable birational class of $\Hilb^n_X$ for $n>>0$ and find a sufficient criterion on $X$ in order to have only finitely many different classes (corollary \ref{corollarydefinitestable}).

    Using Iskovshih classification of geometrically rational surfaces minimal over $k$ (see \cite{Isk}), in Section \ref{8} we prove that for geometrically rational surfaces there are always only finitely many different stable birational classes of $\Hilb^n_X$ (theorem \ref{theoremrationalsurfintro}).
    
    Finally, in Section \ref{9} we exploit the result of Section \ref{8} in order to study the rationality of the motivic Hilbert zeta function (proposition \ref{proprationalityformotivicHilb}) and the motivic zeta function of geometrically rational surfaces (theorem \ref{rationalityofstablmotiviczeta}).

\subsection{Notation}

In what follows, unadorned products are meant to be over $k$. 
By a variety over $k$ we mean a reduced separated scheme over $k$ of finite type.

Let $T$ be a scheme. 
Let $\mathcal{E}$ be a quasi-coherent sheaf of $\mathcal{O}_T$-modules.
By $\mathbb{P}(\mathcal{E})$ we mean the $T$-scheme parametrizing the $\mathcal{O}_T$-sheaf quotients of $g^*\mathcal{E}$ which are
locally free of rank 1:
\[
    \mathbb{P}(\mathcal{E})(T'\overset{g}{\rightarrow} T) = \{ (\mathcal{M},\alpha): \mathcal{M} \textup{ is invertible on } T' \textup{ and a quotient of } g^*\mathcal{E} \textup{ via } \alpha:g^*\mathcal{E} \xtwoheadrightarrow{} \mathcal{M} \}.
\]
Given a scheme $Y$, the unadorned tensor product $\mathcal{F}\otimes \mathcal{G}$ of $\mathcal{O}_Y$-sheaves is meant over $\mathcal{O}_Y$.

\subsection{Acknowledgments} 
I want to deeply thank my advisor, Professor Aise Johan de Jong, for his invaluable support, guidance throughout all the stages of this paper and for giving me feedbacks on the draft.

I also wish to thank all my friends at the Columbia Math Department, for the countless hours spent typing side by side, and to my friend Guglielmo Nocera for supports and for double-checking the legibility of the introduction.

\section{Preliminaries}\label{2}

\noindent As stated in the introduction, we start with a triple $(X,\mathcal{L},Q)$ where
\begin{itemize}
    \item $X$ is a smooth, geometrically connected, projective surface over a field $k$, whose $H^1(X,\mathcal{O}_X)$ is trivial,
    \item $\mathcal{L}$ is an ample line bundle on $X$, and 
    \item $Q$ is a $0$-dimensional, geometrically reduced, closed subscheme of $X$ of degree $d$.
\end{itemize}  
\begin{remark}\label{remarkgeomconnected}
    Notice that any curve $C$ in the linear system defined by $\mathcal{L}$ is geometrically connected.
    
    Indeed, since $\mathcal{L}$ is ample, there is an integer $m$ for which $H^1(X,\mathcal{L}^{\otimes m}\otimes \omega_X)=0$. 
    By Serre duality, $H^1(X,(\mathcal{L}^{\vee})^{\otimes m})=0$, which implies $H^0(mC,\mathcal{O}_{mC})=k$. 
    Hence, $C$ is a geometrically connected curve (see \cite[Tag 0FD1]{deJ}). 
\end{remark}

In the next section, we will define the \emph{incidence correspondences} introduced for comparing $\Hilb^n_X$ and $\Hilb^{n'}_X$.
So let us first recall the objects used in their definitions.

\subsection{Relative effective Cartier divisors}

An \emph{effective Cartier divisor} $D\subset Y$ of a scheme $Y$ is a closed subscheme whose ideal sheaf $\mathcal{I}_D$ (also denoted by $\mathcal{O}_Y(-D)$) is an invertible $\mathcal{O}_Y$-module. 

Given a morphism of schemes $Y\longrightarrow S$, a \emph{relative effective Cartier divisor} on $Y/S$ is an effective Cartier divisor $D\subset Y$ such that $D$ is flat over $S$. 
This notion is stable under base-change: in particular, the restrictions to the fibers $\{D_s\subset Y_s\}_{s\in S}$ are all effective Cartier divisor of $Y_s$ (see \cite[Lemma 056Q]{deJ}).

We can define the functor 
\[
    \textup{CDiv}_{X}: \textup{Sch}_k^{\textup{op}} \longrightarrow \textup{Set},\quad \quad T\mapsto \{ \textup{relative effective Cartier divisors } D \textup{ on } X_T/T \}.
\] 
Since $X$ is projective and smooth, $\CDiv_{X}$ is representable by the disjoint union of connected projective schemes over $k$ (\cite[\S 13, Prop.\ 2]{Bej}).

We have a natural transformation of functors into the relative Picard functor 
\[
    AJ_X : \CDiv_X \longrightarrow \Pic_X, \quad D \subset X_T \mapsto [\mathcal{O}_{X_T}(D)]
\]
called the \emph{Abel-Jacobi morphism} (see \cite[\S 14, Prop.\ 1]{Bej}).

We refer to \cite[\S 14]{Bej} for a discussion of the relative Picard functor $\Pic_{X/k}$ and the Abel-Jacobi map (see also \cite[Thm.\ 3.1]{Gro}). 
However, in the next subsection we recall some properties about its fibers.

\subsection{Fibers of Abel-Jacobi map}

Let $f:Y\longrightarrow S$ be a flat projective scheme over a quasi-compact scheme $S$. 
Assume that $f$ has geometrically integral fibers.

Let $T$ be an $S$-scheme and let $\mathcal{L}$ be a line bundle on the fiber product $Y_T \overset{f_T}{\longrightarrow} T$.
The class $[\mathcal{L}]$ in $\Pic_{Y/S}$ corresponds to a morphism $T\overset{[\mathcal{L}]}{\longrightarrow} \Pic_{Y/S}$. 

If $\mathcal{L}$ is \emph{cohomologically flat of dimension zero} (namely, the formation of its direct image commutes with base change), then ${f_T}_*\mathcal{L}$ is locally free and the bundle projection $ \mathbb{P}( ({f_T}_*\mathcal{L})^{\vee} )  \longrightarrow T$ fits in the cartesian diagram
\[
\begin{tikzcd}
    \mathbb{P}( ({f_T}_*\mathcal{L})^{\vee} )\arrow[r] \arrow[d] 
    & 
    \CDiv_{Y/S} 
    \arrow[d, "{AJ_{Y/S}}"] 
    \\
    T 
    \arrow[r, "{[\mathcal{L}]}"]
    &
    \Pic_{Y/S}
\end{tikzcd}
\] (see \cite[Prop.\ 8.2.7]{BLR90}).
\begin{definition}
    In general, let us denote by $\CDiv_{Y/S}^\mathcal{L}$ the fiber product $\CDiv_{Y/S}\times_{\Pic_{Y/S}} T$ induced by $[\mathcal{L}]$.
\end{definition}

In Section \ref{4} we will need the description of the above mentioned isomorphism between the fiber product $\CDiv_{Y/S}^{\mathcal{L}}$ and $\mathbb{P}( ({f_T}_*\mathcal{L})^{\vee} )$. 
Therefore let us recall how it is defined (for a complete proof see \cite[\S 8]{BLR90}).

\begin{proposition}\label{propisoquotientlinebundlesandcartierdiv}
    Let $f:Y\longrightarrow S$ and $\mathcal{L}$ as above. 
    The isomorphism between $\CDiv_{Y/S}^{\mathcal{L}}$ and $\mathbb{P}( ({f_T}_*\mathcal{L})^{\vee} )$ is given by
    \[
        (D\subset Y_T) \mapsto (\mathcal{M}, (f_{T\, *} i_D )^{\vee} )
    \]
    where $i_D:\mathcal{O}_{Y_T}\xhookrightarrow{} \mathcal{O}_{Y_T}(D)$ and $\mathcal{M}$ is such that $\mathcal{O}_{X_T}(D)\cong \mathcal{L}\otimes f_T^*(\mathcal{M})$.
\end{proposition}

\begin{proof}[Sketch of the proof]
Since $f$ is proper and has geometrically integral fibers, then the formation of the direct image $f_*\mathcal{O}_Y$ commutes with base change. 
Then we have the exact sequence
\[
    0\longrightarrow \Pic(T) \overset{f_T^*}{\longrightarrow} \Pic(Y_T) \longrightarrow \Pic_{Y/S}(T)
\]
(see \cite[Prop.\ 8.1.4]{BLR90}).

Consider an element $D\subset Y_T$ in $\CDiv_{Y/S}^{\mathcal{L}}$: by definition of fiber product, the class $[\mathcal{O}_{Y_T}(D)]$ coincides with  $[\mathcal{L}]$ in $\Pic_{Y/S}(T)$.
Therefore, by the previous exact sequence this happens if and only if there exists a (unique) line bundle $\mathcal{M}$ on $T$ such that 
\[
    \mathcal{O}_{Y_T}(D)\cong \mathcal{L}\otimes f_T^*(\mathcal{M}).
\]
Taking the pushforward of the inclusion $\mathcal{O}_{Y_T} \xhookrightarrow{i_D} \mathcal{O}_{X_T}(D)$ we still get an inclusion
\[
    f_{T\, *}\mathcal{O}_{Y_T} \cong \mathcal{O}_T \xhookrightarrow{f_{T\, *} i_D} f_{T\, *} \mathcal{O}_{Y_T}(D) \cong f_{T\, *}\mathcal{L}\otimes \mathcal{M},
\]
where the last isomorphism is due to the projection formula.

Tensoring both sides by $\mathcal{M}^{\vee}$ and then taking the dual, we get a surjection
\[
    (f_{T\, *} i_D )^{\vee}:({f_T}_*\mathcal{L})^{\vee}\xtwoheadrightarrow{ } \mathcal{M},
\] 
that is an element of $\mathbb{P}( ({f_T}_*\mathcal{L})^{\vee} )$.
Conversely, given a surjective map 
\[
    \alpha:({f_T}_*\mathcal{L})^{\vee} \xtwoheadrightarrow{} \mathcal{M}
\]
onto a line bundle $\mathcal{M}$ on $T$, we can work backwards till we get
\[
    f_{T\, *} \mathcal{O}_{Y_T} \cong \mathcal{O}_{T} \xhookrightarrow{} f_{T\, *} ( \mathcal{L}\otimes f_T^*\mathcal{M} ).
\]
Then thanks to the pullback-pushforward adjunction we get a unique map
\[
    \mathcal{O}_{Y_T}\longrightarrow \mathcal{L}\otimes f_T^*\mathcal{M}.
\] 
This map is injective on each fiber $Y_t$ (with $t\in T$) since its dual fits inside the factorization
\[
    \begin{tikzcd}
        f_t^*f_{t\, *} ( \mathcal{L}  \otimes  f_t^*\mathcal{M} )^{\vee} 
        \arrow[rr, two heads]
        \arrow[dr, dashed]
        &
        &
        \mathcal{O}_{Y_t}.
        \\
        &
        (\mathcal{L}_t \otimes f_t^*\mathcal{M})^{\vee}
        \arrow[ru, two heads]
        &
    \end{tikzcd}
\]
This realises $(\mathcal{L} \otimes f_T^*\mathcal{M})^{\vee}$ as the ideal sheaf of a relative Cartier divisor.
\end{proof}

Let us apply the previous construction to the case $[\mathcal{L}]\in \Pic(X)$.
The surface $X$ is geometrically connected, smooth (hence flat and geometrically integral) and projective over $k$.
By flat base change (\cite[Tag 02KH]{deJ}), any line bundle on a quasi-compact and quasi-separated scheme over a field is cohomologically flat of dimension zero.
Therefore we have the following cartesian diagram
\[
    \begin{tikzcd}
    \mathbb{P}(H^0(X,\mathcal{L})^{\vee})
    \arrow[d] 
    \arrow[r] 
    & 
    \CDiv_{X/k} 
    \arrow[d, "AJ_X"] 
    \\
    \textup{Spec}(k) 
    \arrow[r, "{[\mathcal{L}]}"]                  
    & 
    \Pic_{X/k}.              
    \end{tikzcd}
\]

\subsection{Brill-Noether theory}

In this subsection the field $k$ is assumed to be algebraically closed, but starting from the next one $k$ will denote again any field.\\

For any positive $m$, consider the Abel-Jacobi map for a smooth, geometrically connected, projective $k$-curve $C$:
\begin{equation*}
    \Hilb^m_{C/k}\longrightarrow \Pic^m_{C/k}.
\end{equation*}
Since $C$ is geometrically connected, it is cohomologically flat in dimension zero (\cite[Cor. 1.3]{Con}).
Moreover since $k$ is algebraically closed the map $C(k)\neq \varnothing$. 
Therefore for any $k$-schemes $T$ we have the short exact sequence
\[
    0\longrightarrow 
    \Pic(T)
    \longrightarrow
    \Pic(C\times T)
    \longrightarrow
    \Pic_{C/k}(T)
    \longrightarrow 0
\]
(see \cite[Prop. 8.1.4]{BLR90}).
In particular, all elements in $\Pic_{C/k}$ are represented by the class of some line bundle.

Given an integer $r\geq 0$ there is a natural subfunctor of $\Pic^m_{\mathcal{C}_y/k(y)}$ whose support is the set of linear systems $|D|$ of degree $m$ and dimension at least $r$.
Precisely, we denote by $W^r_m( C/k )$ the functor 
\[
    W^r_m( C/k ): \textup{Sch}^{\textup{op}}_k \longrightarrow \textup{Set}, 
\]
\[
    T \mapsto \{ [\mathcal{M}]\in \Pic^m_{C/k}(T): \dim_{k(t)} H^0(C_t,\mathcal{M}_t)\geq r+1\ \forall \, t\in T\}.
\] 
By upper semicontinuity \cite[Thm III.12.8]{Harty}, the subfunctor $W^r_m( C/k )$ is a closed subscheme of $\Pic^m_{C/k}$
(see \cite[\S IV.3]{ACGH} for a detailed description). 
Furthermore, we have the following fundamental property. 

\begin{fact}[Lemma IV.3.5 of \cite{ACGH}]\label{factBrillNoether}
    If $p_a(C) + r - m \geq 0$, then no irreducible component of $W^r_m( C/k )$ is entirely contained in $W^{r+1}_m( C/k )$. 
    In particular, $W^r_m(C/k)\setminus W^{r+1}_m( C/k )$ is non-empty.
\end{fact}

\section{The Incidence Correspondences}\label{3}

\noindent Fix an integer $n \geq 1$ and set $n' = n + d = n + \deg(Q)$. 

\begin{definition}\label{definitionincidencevariety}
Let us define the functor
\[
    \Inc^Q_{X, \mathcal{L}, n, n'} : \textup{Sch}_k^{\textup{op}} \longrightarrow \textup{Set},
\]
sending $T$ to the set of triples $(Z, Z', C) \in \Hilb^n_X(T)\times \Hilb^{n+d}_X(T)\times \CDiv_{X}^\mathcal{L}(T)$, such that:
\begin{enumerate}
\item[\hypertarget{condition1}{(1)}] the subschemes $Q_T$, $Z$, and $Z'$ are all contained in $C \subset X_T$,
\item[\hypertarget{condition3}{(2)}] the scheme $C$ is smooth over $T$,
\item[\hypertarget{condition4}{(3)}] $Q_T$, $Z$, and $Z'$ are relative effective Cartier divisors on $C$ over $T$,
\item[\hypertarget{condition5}{(4)}] the classes $[\mathcal{O}_C(Z + Q_T)]$ and  $[\mathcal{O}_C(Z')]$ agree in $\Pic_{C/T}(T)$. 
\end{enumerate}
\end{definition}

\begin{definition}
Analogously, one can define the functor
\[
    \Inc^Q_{X, \mathcal{L}, n} : \textup{Sch}_k^{\textup{op}} \longrightarrow \textup{Set}, \quad
    T \longmapsto \{ (Z,C) \in  \Hilb^{n}_X(T)\times \CDiv_{X}^\mathcal{L}(T) 
    : Z, Q_T\subset C \subset X_T \}.
\]
\end{definition}

\noindent The functors $\Inc^Q_{X, \mathcal{L}, n}$ and $\Inc^{Q}_{X, \mathcal{L}, n, n'}$ are the \emph{incidence correspondences} we are interested in.\\

These functors fit in the following diagram
\[
\begin{tikzcd}[column sep={14mm}, row sep={10mm}]
    \Inc^Q_{X,\mathcal{L},n,n'}(T)
    \arrow[dr, hook] 
    \arrow[drr, bend left=5mm, dashed, "{(Z,Z',C)\mapsto (Z',C)}" description]
    \arrow[ddr, dashed, bend right=11mm, "{(Z,Z',C)\mapsto (Z,C)}" description]
    &
    &
    &
    \\
    &
    \Inc^Q_{X,\mathcal{L},n}\times_{\CDiv_X^{\mathcal{L}}}\Inc^Q_{X,\mathcal{L},n'} (T)
    \arrow[r] \arrow[d]
    & 
    \Inc^Q_{X,\mathcal{L},n'} (T)
    \arrow[r, "{(Z',C)\mapsto Z'}"] \arrow[d, "{(Z',C)    \mapsto C}" description]
    &
    \Hilb_X^{n+d}(T)
    &
    \\
    &
    \Inc^Q_{X,\mathcal{L},n}(T)
    \arrow[r, "{(Z,C)\mapsto C}" description] \arrow[d, "{(Z,C)\mapsto Z}" description]
    &
    \CDiv_X^{\mathcal{L}}(T)
    &
    &
    \\
    &
    \Hilb_X^n(T),
    &
    &
    &
\end{tikzcd}
\]
where all the maps are the forgetful functors.
In particular we get two morphisms
\[
    \begin{tikzcd}
        &
        \Inc^Q_{X,\mathcal{L},n,n'}
        \arrow[dl, "{\pi_n}"']
        \arrow[dr, "{\pi_{n'}}"]
        &
        \\
        \Hilb^n_X 
        &
        &
        \Hilb^{n'}_X.
    \end{tikzcd}
    \]

\begin{remark}
    Consider $C\in \CDiv^{\mathcal{L}}_X(T)$.
    Since the pullback of a relative effective Cartier divisor is still a relative effective Cartier divisor, then for any point $t$ of $T$ the pullback $C_t$ is a relative effective Cartier divisor as well.
    In particular the map $C\longrightarrow T$ is surjective.
\end{remark}

\begin{remark}
    In general, if $Y\longrightarrow S$ is flat and $D\subset Y$ is a closed subscheme flat over $S$, then $D$ is a relative effective Cartier divisor if and only if for any point $s$ of $S$ the subscheme $D_s\subset Y_s$ is (see \cite[\S 13, Cor.\ 2]{Bej}).

    Therefore, since $X_T$ is flat over $T$, if $C\subset X_T$ is a closed subscheme flat over $T$, then we can check whether $C\in \CDiv_X(T)$ on fibers.\\

    Analogously, given a triple $(Z,Z',C)\in \Hilb^n(T)\times \Hilb^{n+d}(T)\times \CDiv_X(T)$, then since $Q_T$, $Z$, $Z'$ and $C$ are all flat over $T$, we can check whether $Q_T$, $Z$, $Z'$ are relative effective Cartier divisor on $C/T$ on fibers.
\end{remark}

\begin{remark}\label{remarkaboutbeingrelativecartiersincepoints}
    Since any $0$-dimensional closed subscheme of a smooth curve over a field is an effective Cartier divisor, any triple $(Z,Z',C)\in \Hilb^n(T)\times \Hilb^{n+d}(T)\times \CDiv_X(T)$ satisfying \hyperlink{condition1}{(1)} and \hyperlink{condition1}{(2)} it satisfies \hyperlink{condition1}{(3)} as well.
\end{remark}

\begin{remark}\label{remarkaboutgeomintegral}
    By remark \ref{remarkgeomconnected}, for any $C\in \CDiv^{\mathcal{L}}_X(T)$ we know that $C_t$ is geometrically connected.
    If $C$ satisfies \hyperlink{condition1}{(2)} then $C_t$ is smooth and therefore it is geometrically integral.
    
    This means that for any smooth element $C\in \CDiv^{\mathcal{L}}_X(T)$, the fibers of $C\longrightarrow T$ are geometrically integral.\\

    Therefore since $C\longrightarrow T$ is projective, flat, surjective and has integral geometric fibers, the functor $\Pic_{C/T}$ is representable by a separated $T$-scheme (see \cite[Thm 8.2.1]{BLR90}).
\end{remark}

\subsection{Representability of the Incidence Correspondences} 
In the remaining part of this section, we want to check the representability of the functors just described. 

We first focus on the representability of $\Inc^Q_{X,\mathcal{L},n}$.  
In order to do so, we need first a general lemma.

\begin{lemma}\label{Lemmaclosedsubschemesfunctor}
Let $Y$ be a proper and flat scheme over a noetherian scheme $S$ and let $Z_1,Z_2$ two closed subschemes of $Y$, both flat over $S$. 
Then the functor
\[
    \delta_{Z}(T\overset{f}{\longrightarrow} S)=
    \left\{ 
    \begin{matrix}
    \{ f \} 
    & 
    \textup{if } Z_1\times_f T\subseteq Z_2\times_f T, 
    \\
    \varnothing 
    & 
    \quad \quad \textup{otherwise.} \quad \quad \quad \quad \quad \quad \quad
    \end{matrix}
    \right.
\]
is representable by a closed subscheme of $S$.
\end{lemma}

\begin{proof}
Let $f_T$ be the morphism $Y_T\longrightarrow T$ induced by $f$.
Let $\mathcal{I}_i$ denote the ideal sheaf of $Z_i\xhookrightarrow{j_i} Y$: since $\mathcal{O}_Y$ and $j_{i\, *}\mathcal{O}_{Z_i}$ are all flat over $S$ then $\mathcal{I}_i$ are flat over $S$ as well. 
Therefore the kernels are stable under base-change:
    \[
        \ker( 
        \mathcal{O}_{Y_T}\longrightarrow j_{i\,*}\mathcal{O}_{Z_{i\, T}} 
        )=
        \mathcal{I}_{i,T} 
        =
        f_T^*\mathcal{I}_{i}.
    \] 
Therefore we have a map $\sigma_T: \mathcal{I}_{2,T}\longrightarrow {j_{1\,*}\mathcal{O}_{Z_{1\,T}}}$ induced by
    \[
        \begin{tikzcd}
        {\mathcal{I}_{2,T}} 
        \arrow[rd, hook] 
        \arrow[dd, "\sigma_T"', dashed] 
        &
        & 
        {\mathcal{I}_{1,T}} 
        \arrow[ld, hook] 
        \\
        & 
        \mathcal{O}_{Y_T} 
        \arrow[ld, two heads] 
        \arrow[rd, two heads] 
        &                                      
        \\
        {j_{1\,*}\mathcal{O}_{Z_{1\,T}}} 
        &
        & 
        {j_{2\,*}\mathcal{O}_{Z_{2\,T}}}.             
        \end{tikzcd}
    \]
Since the subscheme $Z_{1\,T}$ is contained in $Z_{2\,T}$ if and only if $\mathcal{I}_{2,T}\subseteq \mathcal{I}_{1,T}$, then $Z_{1\,T}\subseteq Z_{2\,T}$ if and only if $\sigma_T=0_T$.
Let $\sigma$ be the one associated to $f=\id_S$.\\

\noindent Consider the $Hom$-functor
    \[
        Hom_{\mathcal{O}_Y}(\mathcal{I}_2, j_{1\,*}\mathcal{O}_{Z_1}):\textup{Sch}_S^{\textup{op}} \longrightarrow \textup{Set},\ \ \  
        T\mapsto \textup{Hom}_{\mathcal{O}_{Y_T}}(\mathcal{I}_{2,T}, j_{1\,*}\mathcal{O}_{Z_{1\,T}}).
    \] 
Since $Y$ is proper over the noetherian scheme $S$ and $\mathcal{O}_{Z_{1}}$ is flat over $S$, then by \cite[Tag 08JY]{deJ} $Hom_{\mathcal{O}_Y}(\mathcal{I}_2, j_{1,*}\mathcal{O}_{Z_1})$ is an algebraic space affine and of finite presentation over $S$. 
Since affine morphisms of algebraic spaces are representable (see \cite[Tag. 03WG]{deJ}) then $Hom_{\mathcal{O}_Y}(\mathcal{I}_2, j_{1,*}\mathcal{O}_{Z_1})$ is a scheme affine and of finite presentation over $S$.
In particular it is separated over $S$. 

Therefore the global sections $\sigma$ and $0$ are both closed embedding of $S$ into the functor.
    The fiber product of $\sigma$ and $0$ represents the $S$-functor 
    \[
    \begin{tikzcd}
    & 
    S 
    \arrow[rd, "\sigma", hook] 
    &
    &
    \\
    \delta_{Z} 
    \arrow[ru, hook] 
    \arrow[rd, hook] 
    &                      
    &
    {Hom_{\mathcal{O}_Y}(\mathcal{I}_2, j_{1,*}\mathcal{O}_{Z_1})} 
    \arrow[r] 
    &
    S
    \\
    &
    S
    \arrow[ru, "0"', hook]
    &                     
    &  
    \end{tikzcd}
    \] 
    which is then a closed subscheme of $S$.

\end{proof}

\begin{proposition}\label{propositionincQclosedsubscheme}
    $\Inc^Q_{X, \mathcal{L}, n}$ is a closed subscheme of $\Hilb^n_X\times_k \CDiv_X^{\mathcal{L}}$, so it is projective over $k$. 
\end{proposition}

\begin{proof}
Let $\mathcal{Z}$, $\mathcal{C}$ be the universal family associated to $\Hilb^n_X$ and $\CDiv_X^{\mathcal{L}}$ respectively.
Therefore we have three closed subschemes
\[
    \begin{tikzcd}[column sep={7mm}]
         \mathcal{Z}_{ \CDiv_X^{\mathcal{L}} },\,
         Q_{ \Hilb^n_X \times \CDiv_X^{\mathcal{L}} },\,
         \mathcal{C}_{\Hilb^n_X }
         \arrow[r, "\subset" description, phantom]
         &
         X\times \Hilb^n_X \times \CDiv^{\mathcal{L}}_X
         \arrow[d]
         \\
         &
         \Hilb^n_X \times \CDiv^{\mathcal{L}}_X
    \end{tikzcd}
\]
all flat over $\Hilb^n_X \times \CDiv^{\mathcal{L}}_X$. 
Since $X$ is proper and flat over $k$, lemma \ref{Lemmaclosedsubschemesfunctor} applies to the pairs
\[
    \mathcal{Z}_{ \CDiv_X^{\mathcal{L}} }\subset \mathcal{C}_{\Hilb^n_X }
    \quad \textup{ and } \quad
    Q_{ \Hilb^n_X \times \CDiv_X^{\mathcal{L}} } \subset \mathcal{C}_{\Hilb^n_X }
\]
producing two closed representable subfunctors of $\Hilb^n_X\times \CDiv_X^{\mathcal{L}}$, whose intersection is $\Inc^Q_{X, \mathcal{L}, n}$. 
Hence $\Inc^Q_{X, \mathcal{L}, n}$ is a closed subscheme of $\Hilb^n_X\times \CDiv_X^{\mathcal{L}}$.  \\ 

Finally both $\Hilb^n_X$ and $\CDiv_X^{\mathcal{L}}$ (which is isomorphic to $ \mathbb{P}(H^0(X,\mathcal{L})^{\vee})$) are projective over $k$.
So $\Inc^Q_{X, \mathcal{L}, n}$ is projective as well.
\end{proof} 

\begin{corollary}
    The functor $\Inc^Q_{X, \mathcal{L}, n}\times_{\CDiv_X^{\mathcal{L}}} \Inc^Q_{X, \mathcal{L}, n'}$ is representable by a projective scheme over $k$.
\end{corollary}

\begin{proof}
Since the composite
\[
    \Inc^Q_{X, \mathcal{L}, n} \longhookrightarrow \Hilb^n_X\times \CDiv_X^{\mathcal{L}} \longrightarrow \CDiv_X^{\mathcal{L}}, 
\] 
is projective, the fiber product $\Inc^Q_{X, \mathcal{L}, n}\times_{\CDiv_X^{\mathcal{L}}} \Inc^Q_{X, \mathcal{L}, n'}$ is projective over $k$ as well ($\CDiv_X^{\mathcal{L}}$ is projective over $k$).
\end{proof}

Notice that the projective scheme $\Inc^Q_{X, \mathcal{L}, n}\times_{\CDiv_X^{\mathcal{L}}} \Inc^Q_{X, \mathcal{L}, n'}$ already parametrizes triples $(Z,Z',C)$ which satisfy condition \hyperlink{condition1}{(1)} of definition \ref{definitionincidencevariety}.

\begin{proposition}
    The functor $\Inc^Q_{X,\mathcal{L},n,n'}$ is representable by a locally closed subscheme of $\Inc^Q_{X,\mathcal{L},n}\times_{\CDiv_X^{\mathcal{L}}} \Inc^Q_{X, \mathcal{L}, n'}$.
    In particular it is a quasi-projective scheme over $k$.
\end{proposition}

\begin{proof}
    Let $\mathcal{S}$ denote the scheme $\Inc^Q_{X,\mathcal{L},n}\times_{\CDiv_X^{\mathcal{L}}}\Inc^Q_{X, \mathcal{L}, n'}$. 
    Let $\mathcal{Z}, \mathcal{Z}', Q_{\mathcal{S}}, \mathcal{C}$ be the closed subschemes of $X\times \mathcal{S}$ corresponding to the universal families of $Z,Z',Q$ and $C$ respectively. 
    By openness of the smooth locus 
    the locus where the composite
    \[
        p:\mathcal{C}\subset X\times \mathcal{S}\longrightarrow \mathcal{S}
    \] 
    is smooth (and hence cohomologically flat of dimension zero) 
    is open.
    Let $\mathcal{S}^{\circ}$ denote such open and let $p^{\circ}:\mathcal{C}^{\circ}\longrightarrow \mathcal{S}^{\circ}$ the restriction. \\
    
    By remark \ref{remarkaboutbeingrelativecartiersincepoints}, on the open $\mathcal{S}^{\circ}$ the closed subschemes $\mathcal{Z}^{\circ}, \mathcal{Z}^{' \, \circ}, Q_{\mathcal{S}^{\circ}}$ are all relative effective divisors. 
    Therefore we can consider the line bundles $\mathcal{O}_{\mathcal{C}^{\circ}}({\mathcal{Z}'}^{\circ})$ and $\mathcal{O}_{\mathcal{C}^{\circ}}(\mathcal{Z}^{\circ}+\mathcal{Q}^{\circ})$ and their induced maps  
    \[
        \begin{tikzcd}
        & 
        {\mathcal{S}^{\circ}} \arrow[rd, "{[\mathcal{O}_{\mathcal{C}^{\circ}}({\mathcal{Z}'}^{\circ})]}", hook]                   
        &  
        &                                 
        \\
        \mathcal{P} \arrow[rd, hook'] \arrow[ru, hook] 
        &
        & 
        {\Pic_{\mathcal{C}^{\circ}/ \mathcal{S}^{\circ} }} \arrow[r] 
        & 
        \mathcal{S}^{\circ}.
        \\
        & 
        {\mathcal{S}^{\circ}} \arrow[ru, "{[\mathcal{O}_{\mathcal{C}^{\circ}}(\mathcal{Z}^{\circ}+\mathcal{Q}^{\circ})]}"', hook'] 
        &   
        &                                 
        \end{tikzcd}
   \] 
   Notice that the fiber product $\mathcal{P}$ is exactly the functor $\Inc^Q_{X,\mathcal{L},n,n'}$.
   Due to cohomological flatness, as mentioned before the functor $\Pic_{\mathcal{C}^{\circ}/\mathcal{S}^{\circ}}$ is represented by an algebraic space.
   Therefore the fiber product $\mathcal{P}$ is represented by a scheme. 
   Since the global sections induced by the line bundles are locally closed embeddings, $\mathcal{P}$ is locally closed in $\mathcal{S}^{\circ}$. 
   Therefore $\Inc^Q_{X,\mathcal{L},n,n'}$ is a locally closed subscheme of $\Inc^Q_{X,\mathcal{L},n}\times_{\CDiv_X^{\mathcal{L}}} \Inc_{X, \mathcal{L}, n'}$.
\end{proof}

\section{Description of the General fibers in the Incidence Correspondences}\label{4}

    Our next step is to study the general fibers of the projection maps $\Inc^Q_{X,\mathcal{L},n} \longrightarrow \Hilb^n_X$ and $\Inc^Q_{X,\mathcal{L},n,n'} \longrightarrow \Inc^Q_{X,\mathcal{L},n}$. 

\subsection{The general fiber on \(\mathbf{\Hilb^n_X}\)}
    We focus first on $\Inc^Q_{X,\mathcal{L},n} \longrightarrow \Hilb^n_X$.
    We want to prove the following.

    \begin{proposition}[Definition of $U_n$ and $U_{n+d}$]\label{definitionofUn}
    Define $U_n$ to be the subfunctor of $\Hilb^n_{X\setminus Q}$ describing the points $y$ for which 
    \[
        H^0(X_{y},(pr^*_X\mathcal{L})_{y}) \longrightarrow H^0(X_{y}, (pr^*_X\mathcal{L})_{y}|_{\mathcal{Z}_{y} \sqcup \mathcal{Q}_y}  )
    \] 
    is surjective and
    \[
        pr_{H *}(pr^*_X\mathcal{L})\otimes_{\mathcal{O}_{\Hilb^n_X, y}} k(y)
        \longrightarrow 
        H^0(X_y, (pr^*_X\mathcal{L})_y)
    \]
    is an isomorphism. 

    For notation sake, let us denote by $H^n_{\setminus Q}$ the open subscheme $\Hilb^n_{X\setminus Q}$ of $\Hilb^n_X$.
    Then consider the universal families $\mathcal{Z}$ and $Q_{\Hilb^n_X}$ restricted to this open: 
    \[
    \begin{tikzcd}[column sep={17mm}]
        \mathcal{Z}_{H^n_{\setminus Q}} \sqcup Q_{H^n_{\setminus Q}} \subset  X\times H^n_{\setminus Q} 
        \arrow[r, "{pr_{\Hilb^n_{X\setminus Q}}}"]
        & 
        H^n_{\setminus Q}.
    \end{tikzcd}
    \]
    Because of they are disjoint, we can consider the ideal sheaf $\mathcal{I}_{\mathcal{Z}_{H^n_{\setminus Q}}\sqcup Q_{H^n_{\setminus Q}}}$.
    \\

    The following hold:
    \begin{enumerate}
        \item the subfunctor $U_n$ is an open subscheme of $H^n_{\setminus Q}$,
        \item the sheaves $pr_{H^n_{\setminus Q} *}( pr^*_X\mathcal{L})$ and $pr_{H^n_{\setminus Q} *}( pr^*_X\mathcal{L} \otimes \mathcal{I}_{\mathcal{Z}_{H^n_{\setminus Q}} \sqcup Q_{H^n_{\setminus Q}} })$ are locally free and their formation commutes with base change,
        \item denote by $\mathcal{E}_n$ the dual sheaf
        \[
            (pr_{H^n *}( pr^*_X\mathcal{L} \otimes \mathcal{I}_{\mathcal{Z}_{H^n_{\setminus Q}} \sqcup Q_{H^n_{\setminus Q}} })\, |_{U_n})^{\vee}.
        \]
        Then the projection map $\mathbb{P}(\mathcal{E}_n)\longrightarrow U_n$ is the fiber product in the cartesian diagram
        \[
            \begin{tikzcd}
                \mathbb{P}(\mathcal{E}_n)
                \arrow[r, hook, circled]
                \arrow[d]
                &
                \Inc^Q_{X,\mathcal{L},n}
                \arrow[d]
                \\
                U_n
                \arrow[r, hook, circled]
                &
                \Hilb^n_X.
            \end{tikzcd}
        \]
    \end{enumerate}
    \noindent Same definition and properties hold for $U_{n+d}\subseteq \Hilb^{n+d}_{X\setminus Q}$ and $\mathcal{E}_{n+d}$.
    \end{proposition}

    The idea is to use the result about the fibers of Abel-Jacobi map with respect to the diagram
    \[
        \begin{tikzcd}
            &
            \CDiv_X
            \arrow[d, "{AJ_X}"]
            \\
            \Hilb^n_X
            \arrow[r, "{[pr_X^*\mathcal{L}]}"]
            &
            \Pic_X
        \end{tikzcd}
    \] 
    where $pr_X$ is the projection $X\times \Hilb^n_X\longrightarrow X$.
    The issue is that at priori we do not know that formation of direct image of $pr_{H *} pr_X^*\mathcal{L}$ commutes with base change. 

    Furthermore $\Inc^Q_{X,\mathcal{L},n}$ is a closed subscheme of the above fiber product: therefore we need to find an appropriate subbundle of $pr_{H *} pr_X^*\mathcal{L}$ which is still cohomologically flat of dimension zero.

    Proposition \ref{definitionofUn} tells us that both issues can be overcome after restricting to a suitable open subscheme of $\Hilb^n_X$.

    \begin{remark}\label{remarkgeneralopenUbasechange}
        Let $f:Y\longrightarrow S$ be a proper, flat morphism with $S$ locally noetherian. 
        Let $W\subseteq Y$ be a closed subscheme, flat over $S$. 
        Let $\mathcal{L}$ be an invertible sheaf on $Y$. 
        Since the sheaf $\mathcal{L}\otimes \mathcal{I}_W$ is flat over $S$ as well, by Cohomology and base change \cite[Thm.\ III.12.11]{Harty}, the locus where the formation of direct image $f_*(\mathcal{I}_W\otimes \mathcal{L})$ commutes with base change is open. 
    \end{remark}

\begin{remark}
    The cited \emph{Cohomology and Base Change} result from Hartshorne's \emph{Algebraic Geometry} is proven under projectivity assumptions on $f$. 
    However, properness of the map is enough (see \cite{Con}).
\end{remark}

We first want to prove the first two items of proposition \ref{definitionofUn}.

\begin{lemma}\label{lemmaopenwherethesheafislocallyfree}
    Let $Y,S,W,\mathcal{L}$ be as in remark \ref{remarkgeneralopenUbasechange}. 
    Let $U$ be the set of points $s\in S$ where
    \[
    H^0(Y_{s},\mathcal{L}_{s}) 
    \longrightarrow 
    H^0(Y_s , \mathcal{L}_{s}|_{W_{s}}  )
    \] 
    is surjective and
    \[
        f_*\mathcal{L}\, \otimes_{\mathcal{O}_{S, s}} k(s)
        \longrightarrow
        H^0(Y_s, \mathcal{L}_{s})
    \] 
    is an isomorphism (we denote by $\mathcal{L}_{s}$ the pullback of $\mathcal{L}$ to $Y_s$).
    
    Then the set $U$ is open and formation of direct images of $f_*\mathcal{L}$ and $f_*(\mathcal{I}_{W}\otimes  \mathcal{L})$ commutes with base change on $U$.
\end{lemma}

    \begin{proof}
        Again by Cohomology and base change \cite[Thm.\ III.12.11]{Harty} we know that the set $S^{\circ}$ where the map
        \[
            f_*\mathcal{L} \otimes_{\mathcal{O}_{S, s}} k(s)
            \longrightarrow H^0(Y_s, \mathcal{L}_{s})
        \] 
        is an isomorphism is open. 
        In particular, over $S^{\circ}$ the sheaf $f_*\mathcal{L}$ is locally free and its formation commutes with base change.  
        
        In particular on each connected component of $\mathcal{S}^{\circ}$ the rank of $f_*\mathcal{L}$ is constant: hence we can assume that $H^0(Y_s, \mathcal{L}_{s})$ is constant.\\
        
        Surjectivity of $H^0(Y_{s},\mathcal{L}_{s}) \longrightarrow H^0(Y_s , \mathcal{L}_{s}|_{W_{s}}  )$ is equivalent to have dimension of the kernel $\dim_{k(s)} H^0(Y_{s},\mathcal{L}_{s}\otimes \mathcal{I}_{W_s})$ equal to 
        \[
            \dim_{k(s)} H^0(Y_{s},\mathcal{L}_{s}) 
            - 
            \dim_{k(s)} H^0(Y_s , \mathcal{L}_{s}|_{W_{s}}  ).
        \]
        By upper-semicontinuity \cite[Thm.\ III.12.8]{Harty}, for any integer $m$ the locus where 
        \[
            \dim_{k(s)} H^0(Y_s , \mathcal{L}_{s}|_{W_{s}}  ) \geq m
        \]
        is open.
        Since (again by upper-semicontinuity) the locus where 
        \[
            \dim_{k(s)} H^0(Y_{s},\mathcal{L}_{s}\otimes \mathcal{I}_{W_s}) \geq 
            \dim_{k(s)} H^0(Y_{s},\mathcal{L}_{s}) 
            - m
        \]
        is open then surjectivity is achieved on an open. 
        Therefore the set $U$ is open.\\

        Last, we need to check the formation of direct image of $\mathcal{L}\otimes \mathcal{I}_W$ on $U$.
        Since $\mathcal{I}_{W}\otimes \mathcal{L}$, $\mathcal{L}$ and  $\mathcal{O}_W\otimes \mathcal{L}$ are all flat over $S$, the exact sequence 
        \[
            0\longrightarrow
            \mathcal{L}\otimes \mathcal{I}_W
            \longrightarrow
            \mathcal{L}
            \longrightarrow
            \mathcal{L}|_W 
            \longrightarrow 0
        \]
        remains exact after we base change along $s: \Spec(k(s)) \longrightarrow S$.
        By definition of $U$, for any point $s\in U$ we get the exact sequence
        \[
        0
        \longrightarrow              
        H^0( Y_{s}, \mathcal{L}_{s} \otimes \mathcal{I}_{W_s}  )
        \longrightarrow 
        H^0( Y_{s},\mathcal{L}_{s})
        \longrightarrow  
        H^0( Y_{s}, \mathcal{L}_{s} |_{W_s}  ) 
        \longrightarrow 0
        \]
        If we now restrict to $U$ the pushforwarded sequence,
        \begin{equation}\label{equationrestrictedtoU}
        0\longrightarrow
        f_*( \mathcal{L}\otimes \mathcal{I}_W )|_U
        \longrightarrow
        f_*\mathcal{L}|_U
        \longrightarrow 
        f_*(\mathcal{L} |_{W} )|_U,
        \end{equation}
        and we tensor it by $-\otimes_{\mathcal{O}_{U, s}} k(s)$,
        \[
            f_*( \mathcal{L}\otimes \mathcal{I}_W )|_U \otimes_{\mathcal{O}_{U, s} } k(s)
            \longrightarrow f_*\mathcal{L} |_U \otimes_{\mathcal{O}_{U, s} } k(s)
            \longrightarrow f_*( \mathcal{L} |_{W} ) |_U \otimes_{\mathcal{O}_{U, s} } k(s),
        \]
        we get to the following commutative diagram:
        \begin{equation}\label{equationcohomologyandbasechange}
        \begin{tikzcd}[column sep={4mm}]
        &
        f_*( \mathcal{L}\otimes \mathcal{I}_W ) |_U \otimes_{\mathcal{O}_{U, s} } k(s) 
        \arrow[r] \arrow[d] 
        &
        f_*\mathcal{L} |_U \otimes_{\mathcal{O}_{U, s}}  k(s)
        \arrow[r] \arrow[d, "\wr"] 
        &
        f_*(\mathcal{L} |_{W}) |_U \otimes_{\mathcal{O}_{U, s} } k(s) 
        \arrow[d] 
        &
        \\
        0
        \arrow[r]
        &
        H^0(Y_s,\mathcal{L}_{s} \otimes  \mathcal{I}_{W_s} ) 
        \arrow[r]
        &
        H^0(Y_s,\mathcal{L}_{s}) 
        \arrow[r]       
        &
        H^0(Y_s, \mathcal{L}_{s}|_{W_s}  )
        \arrow[r]
        &
        0
        \end{tikzcd}    
        \end{equation}
        where the vertical maps are the natural maps given by the base change. 
        Since $s\in U$ the vertical map in the middle of the diagram \ref{equationcohomologyandbasechange} is an isomorphism which implies that the one on the right is surjective, and hence an isomorphism as well.
        Therefore by Cohomology and base change \cite[Thm. III.12.11]{Harty}, the sheaf $f_* (\mathcal{L} |_W) |_U$ is locally free and its formation commutes with base change.

        Furthermore, by composition, the morphism $f_*\mathcal{L}|_U \otimes_{\mathcal{O}_{U, s}} k(s)\longrightarrow f_*(\mathcal{L}|_{W} ) |_U \otimes_{\mathcal{O}_{U,s}} k(s)$ is surjective as well. 
        Since both $f_* \mathcal{L} |_U$ and $f_* (\mathcal{L} |_W) |_U$ are locally free, the surjectivity of the map lifts to
        \[
            f_*\mathcal{L}|_U\longrightarrow f_* (\mathcal{L}|_{W}) |_U.
        \] 
        Therefore the pushforwarded sequence \ref{equationrestrictedtoU} is exact on the right as well:
        \[
            0\longrightarrow
            f_*( \mathcal{L}\otimes \mathcal{I}_W )|_U
            \longrightarrow
            f_*\mathcal{L}|_U
            \longrightarrow 
            f_*(\mathcal{L} |_{W} )|_U
            \longrightarrow 0.
        \]
        Therefore for any $s\in U$ we get the exact sequence
        \[
            f_*( \mathcal{L}\otimes \mathcal{I}_W ) |_U \otimes_{\mathcal{O}_{U, s} } k(s)
            \longrightarrow f_*\mathcal{L} |_U \otimes_{\mathcal{O}_{U, s} } k(s)
            \longrightarrow f_*( \mathcal{L} |_{W} ) |_U \otimes_{\mathcal{O}_{U, s} } k(s)
            \longrightarrow 0.
        \]
        By Snake lemma, the vertical map on the left in \ref{equationcohomologyandbasechange} is then surjective, hence an isomorphism, which implies that $f_*( \mathcal{L}\otimes \mathcal{I}_W)|_U$ is locally free and its formation commutes with base change.
    \end{proof}

    Finally, we focus on the third item of proposition \ref{definitionofUn}.

    \begin{proof}[Proof of Proposition \ref{definitionofUn}]
    Consider the diagram
    \[
        \begin{tikzcd}
                \mathbb{P}(\mathcal{E}_n)
                \arrow[r, hook] 
                &  
                \mathbb{P}( (pr_{U_n *}(pr^*_X\mathcal{L})|_{U_n} )^{\vee}) \arrow[d] \arrow[r]
                & 
                \CDiv_X \arrow[d] 
                \\
                & 
                U_n \arrow[r, " {[pr^*_X\mathcal{L}]}" ]   
                &   
                \Pic_X.      
        \end{tikzcd}
    \] 
    The fiber product $\mathbb{P}( (pr_{U_n *}(pr^*_X\mathcal{L})|_{U_n} )^{\vee})$ coincides with $U_n\times \CDiv^{pr^*_X\mathcal{L}}_X$ because $[pr^*_X\mathcal{L}]$ factors as
    \[
        U_n\longrightarrow \Spec(k)\overset{[\mathcal{L}]}{\longrightarrow} \Pic_X.
    \]
    So both $\mathbb{P}( \mathcal{E}_n )$ and  $\Inc^Q_{X,\mathcal{L},n}|_{U_n}$ are closed subschemes of $\mathbb{P}( (pr_{U_n *}(pr^*_X\mathcal{L})|_{U_n} )^{\vee})$.
    We want to prove that they coincide.\\
    
    In order to do so, we check that they correspond via the isomorphism between
    \[
        \CDiv^{pr^*_X\mathcal{L}}_X \cong \mathbb{P}( (pr_{U_n *}(pr^*_X\mathcal{L})|_{U_n} )^{\vee})
    \]
    described in proposition \ref{propisoquotientlinebundlesandcartierdiv}.
    Fix $t:T\longrightarrow U_n$ so that we have
    \[
        \begin{tikzcd}
            X \times T
            \arrow[d, "{pr_T}"]
            \arrow[r]
            &
            X\times U_n
            \arrow[d, "{pr_{U_n}}"]
            \arrow[r, "{pr_X}"]
            &
            X
            \arrow[d]
            \\
            T
            \arrow[r, "t"]
            &
            U_n
            \arrow[r]
            &
            \Spec(k).
        \end{tikzcd}
    \]
    On one hand, the subbundle $\mathbb{P}( \mathcal{E}_n )$ parametrizes line bundles $(\mathcal{M}, \alpha)$ on $T\overset{t}{\longrightarrow} U_n$  whose quotient map $\alpha: t^*(pr_{U_n *}(pr^*_X\mathcal{L})|_{U_n})^{\vee}\longrightarrow \mathcal{M}$ factors via the inclusion $\mathcal{I}_{\mathcal{Z}_{U_n} \sqcup Q_{U_n} }\longhookrightarrow \mathcal{O}_{X\times U_n}$ 
    \begin{equation}
        \begin{tikzcd}\label{diagramfactorization}
        t^*(pr_{U_n *}(pr^*_X\mathcal{L})|_{U_n})^{\vee} 
        \arrow[rr,"\alpha"] 
        \arrow[dr]
        &
        &
        \mathcal{M}
        \\
        &
        t^*( \mathcal{E}_n )
        \arrow[ur]
        &
        \end{tikzcd}
    \end{equation}
    (see \cite[Rmk.\ 4.2.9]{AGroII}).

    On the other hand, a Cartier divisor $D \subset X_T$ whose ideal sheaf is isomorphic to
    \[
        \mathcal{O}_{X_T}(-D)\cong (id\times t)^*pr^*_X\mathcal{L}^{\vee}\otimes pr^*_T\mathcal{M}^{\vee}
    \]
    (where $\mathcal{M}$ a line bundle on $T$) contains $\mathcal{Z}_T \sqcup Q_T$ if and only if the inclusion $\mathcal{O}_{X_T} \xhookrightarrow{} \mathcal{O}_X(D)$ factors as
    \[
        \begin{tikzcd}
        \mathcal{O}_{X_T} 
        \arrow[rr, hook] 
        \arrow[dr, hook]
        &
        &
        (id\times t)^*pr^*_X\mathcal{L}\otimes pr^*_T\mathcal{M}
        \\
        &
        \mathcal{I}_{ \mathcal{Z}_T \sqcup Q_T }\otimes (id\times t)^*pr^*_X\mathcal{L}\otimes  pr^*_T\mathcal{M}.
        \arrow[ur,  hook]
        &
        \end{tikzcd}
    \]
    Taking the pushforward with respect to $pr_{T *}$ and then tensoring by $\mathcal{M}^{\vee}$, by proposition \ref{propisoquotientlinebundlesandcartierdiv} we have this factorization if and only if
    \begin{equation}\label{diagramfactorization2}
        \begin{tikzcd}
        \mathcal{M}^{\vee}
        \arrow[rr, hook] 
        \arrow[dr, hook]
        &
        &
        pr_{T *}(id\times t)^*pr^*_X\mathcal{L} 
        \\
        &
        pr_{T *}( \mathcal{I}_{ \mathcal{Z}_T \sqcup Q_T }\otimes (id\times t)^*pr^*_X\mathcal{L} ).
        \arrow[ur,  hook]
        &
        \end{tikzcd}
    \end{equation}
    By diagram chasing the sheaf $pr_{T *}(id\times t)^*pr^*_X\mathcal{L}$ coincides with $t^*(pr_{U_n *}(pr^*_X\mathcal{L}|_{U_n}))$.
    Since $\mathcal{I}_{\mathcal{Z}_T\sqcup Q_T}$ is the pullback via $(id\times t)$ of the ideal sheaf $\mathcal{I}_{\mathcal{Z}_{U_n}\sqcup Q_{U_n}}$ then 
    \[
        pr_{T *}(id\times t)^*( \mathcal{I}_{ \mathcal{Z}_{U_n} \sqcup Q_{U_n}}\otimes pr^*_X\mathcal{L} ) \cong t^* (pr_{U_n *}( pr^*_X\mathcal{L} \otimes \mathcal{I}_{ \mathcal{Z}_{U_n} \sqcup Q_{U_n}}  )|_{U_n}) \cong t^* \mathcal{E}_n.
    \]
    Taking the dual of diagram \ref{diagramfactorization2} we get a factorization like in diagram \ref{diagramfactorization} and viceversa. 
    Therefore the subbundle $\mathbb{P}( \mathcal{E}_n )$ coincides with $\Inc^Q_{X,\mathcal{L},n}|_{U_n}$.
\end{proof}

\subsection{The general fiber on $\mathbf{\Inc^Q_{X,\mathcal{L},n}}$} It is now time to turn our attention to the projection maps
    \[
    \Inc^Q_{X,\mathcal{L},n,n'}\longrightarrow \Inc^Q_{X,\mathcal{L},n}
    \quad \ \textup{and} \ \quad
    \Inc^Q_{X,\mathcal{L},n,n'}\longrightarrow \Inc_{X,\mathcal{L},n'}
    \] 
    (where $n'$ is always $n+d$).
    Since we are interested in the general fiber on a open of $\Hilb^n_X$ (respectively of $\Hilb^{n+d}_X$), we study their restriction to $U_n$ and $U_{n+d}$
    \[
    \Inc^Q_{X,\mathcal{L},n,n'}|_{U_n} \longrightarrow \mathbb{P}( \mathcal{E}_n )
    \quad \ \textup{and} \ \quad
    \Inc^Q_{X,\mathcal{L},n,n'}|_{U_{n+d}} \longrightarrow \mathbb{P}( \mathcal{E}_{n+d} ).
    \] 
    We aim for a statement analogous to proposition \ref{definitionofUn}.

    The strategy is similar to the one adopted in the previous section. 
    However now we want to use the Abel-Jacobi map for 
    \[
    AJ_{\mathcal{C}/\mathbb{P}(\mathcal{E}_n)} : \CDiv_{\mathcal{C}/\mathbb{P}(\mathcal{E}_n)}\longrightarrow \Pic_{\mathcal{C}/\mathbb{P}(\mathcal{E}_n)}
    \]
    where $\mathcal{C}$ is the universal curve
    \begin{equation}\label{definitionofg}
    \begin{tikzcd}[column sep={6mm}]
        \mathcal{C} 
        \arrow[r, hook]
        \arrow[rd, "g"']
        &
        X \times \mathbb{P}( \mathcal{E}_n )
        \arrow[d]
        \\
        &
        \mathbb{P}( \mathcal{E}_n ).
    \end{tikzcd}
    \end{equation}
    Let $\mathcal{Z}$ be the universal family (contained in $\mathcal{C}$) with respect to the subscheme $Z$.

\begin{proposition}[Definition of $V_n$ and $V'_{n+d}$]\label{existenceopenforPE}
Let $V$ be the open of $\mathbb{P}(\mathcal{E}_n)$ where $g$ is smooth (and hence cohomologically flat of dimension zero). 

In particular, all the fibers of $g$ on $V$ are geometrically integral (see remark \ref{remarkaboutgeomintegral}). 
Moreover restricted to $V$ the subschemes $\mathcal{Z}_V$ and $Q_V$ are relative effective Cartier divisor on $\mathcal{C}_{V}$ (see remark \ref{remarkaboutbeingrelativecartiersincepoints}).

Consider now the subfunctor $V_n$ of the subscheme $V\subseteq \mathbb{P}(\mathcal{E}_n)$ made of the points $y\in V$ such that 
\[ 
    g_*\mathcal{O}_{\mathcal{C}_V}(\mathcal{Z}_V+ Q_V )  \otimes \mathcal{O}_{V, y}\, k(y) \longrightarrow H^0(\mathcal{C}_y, \mathcal{O}_{\mathcal{C}_y}(\mathcal{Z}_y+Q_y))
\]
is surjective.\\

Then the following hold:
\begin{enumerate}
    \item the subfunctor $V_n$ is an open subscheme of $V$,
    \item the sheaf $g_*\mathcal{O}_{\mathcal{C} |_{V_n} }(\mathcal{Z}_{V_n}+ Q_{V_n} ) $ is locally free and its formation commutes with base change, 
    \item denote by $\mathcal{F}_n$  the dual of $g_*\mathcal{O}_{\mathcal{C}|_{V_n } }(\mathcal{Z}_{V_n}+ Q_{V_n} ) $: then the projection map $\mathbb{P}(\mathcal{F}_n)\longrightarrow V_n$ fits in the cartesian diagram
    \[
        \begin{tikzcd}
                \mathbb{P}(\mathcal{F}_n)
                \arrow[r, hook, circled]
                \arrow[d]
                &
                \Inc^Q_{X,\mathcal{L},n,n'}|_{U_n}
                \arrow[d]
                \\
                V_n
                \arrow[r, hook, circled]
                &
                \mathbb{P}(\mathcal{E}_n).
            \end{tikzcd}
    \]
\end{enumerate}
Analogously, let $V'$ be the open where $g'$ is smooth: then same definition and properties go for $\Inc^Q_{X,\mathcal{L},n,n'}|_{U_{n+d}} \longrightarrow \mathbb{P}( \mathcal{E}_{n+d} )$: but in this case the sheaf considered is $g'_*\mathcal{O}_{\mathcal{C}_{V'}}(\mathcal{Z}'_{V'}-Q_{V'})$. 
We denote the open and the dual sheaf by $V'_{n+d}$ and $\mathcal{F}'_{n+d}$ respectively.
\end{proposition}

\begin{proof}
    Since $g$ is flat and projective over a noetherian scheme, by \cite[\S 13 Prop.\ 2]{Bej} the functor $\CDiv_{\mathcal{C} / \mathbb{P}(\mathcal{E}_n) }$ is representable by a scheme.
    As remarked in \ref{remarkaboutgeomintegral}, by \cite[\S Thm. 8.2.1]{BLR90} the functor $\Pic_{\mathcal{C}_{V}/V }$ is representable by a scheme.
    
    Consider the diagram induced by the line bundle $\mathcal{O}_{\mathcal{C}_V}(\mathcal{Z}_V + Q_V)$:
    \[
        \begin{tikzcd}[column sep={17mm}]
            \CDiv^{\mathcal{O}_{\mathcal{C}_{V}}(\mathcal{Z}+Q_V)}_{\mathcal{C}_{V } / V }
            \arrow[r] 
            \arrow[d]
            &
            \CDiv_{\mathcal{C}_{V } / V }
            \arrow[d, "{ AJ_{ \mathcal{C}_{ V}  / V } }"]
            \\
            V
            \arrow[r, "{ [\mathcal{O}_{\mathcal{C}_{V}}(\mathcal{Z}_V+Q_V)] }"]
            &
            \Pic_{\mathcal{C}_{V} / V }.
            \end{tikzcd}
    \] 
    By Cohomology and base change \cite[Thm.\ III.12.11]{Harty}, the set where
    \[ 
    g_*\mathcal{O}_{\mathcal{C}_V}(\mathcal{Z}_V+ Q_V )  \otimes \mathcal{O}_{V, y}\, k(y) \longrightarrow H^0(\mathcal{C}_y, \mathcal{O}_{\mathcal{C}_y}(\mathcal{Z}_y+Q_y))
    \]
    is surjective is an open $V_n$ of $V$ over which the direct image $g_*(\mathcal{O}_{\mathcal{C}_{V}}(\mathcal{Z}_V+Q_V))$ is locally free and its formation commutes with base-change.

    Then the projection $\mathbb{P}(\mathcal{F}_n)\longrightarrow V_n$ is the fiber product in
    \[
        \begin{tikzcd}[column sep={30mm}]
            \mathbb{P}( \mathcal{F}_n )
            \arrow[r] 
            \arrow[d]
            &
            \CDiv_{\mathcal{C}_{V_n}/V_n}
            \arrow[d, "{AJ_{\mathcal{C}_{V_n}/V_n}}"]
            \\ 
            V_n
            \arrow[r, "{ [\mathcal{O}_{\mathcal{C}_{V_n}}(\mathcal{Z}_{V_n}+Q_{V_n})] }"]
            &
            \Pic_{\mathcal{C}_{V_n}/V_n}.
        \end{tikzcd}
    \] 
    Since the fiber product $\CDiv^{\mathcal{O}_{\mathcal{C}_{V_n}}(\mathcal{Z}_{V_n}+Q_{V_n})}_{\mathcal{C}_{V_n}/V_n}$ is the same as $\Inc^Q_{X,\mathcal{L},n,n'}$ we get the wanted cartesian diagram:
    \[
        \begin{tikzcd}
            \mathbb{P}( \mathcal{F}_n)
            \arrow[r, hook, circled] 
            \arrow[d]
            &
            \Inc^Q_{X,\mathcal{L},n,n'}|_{ U_n }
            \arrow[d]
            \\
            V_n
            \arrow[r, hook, circled]
            &
            \mathbb{P}( \mathcal{E}_n).
        \end{tikzcd}
    \]
    The proof for $V'_{n+d}$ goes in the same way.
\end{proof}

\section{The General fiber is non-empty}\label{5}

The aim of this section is to give sufficient conditions on the triple $(X,\mathcal{L},Q)$ so that we can ensure that the opens defined in the previous section are all non-empty.

\begin{assumptions}\label{thefiveassumptions}
The five \hypertarget{assumptions}{assumptions} \emph{with respect to} $n$ are the following:
\begin{enumerate}\label{assumptions}
    \item[\hypertarget{assumptions}{(1)}] the restriction map $H^0(X,\mathcal{L}) \longrightarrow H^0(X, \mathcal{L}|_{Q})$ is surjective,
    
    \item[\hypertarget{assumptions}{(2)}]  $\dim_k H^0(X, \mathcal{L})\geq n+1+2d$,
    
    \item[\hypertarget{assumptions}{(3)}] the sheaf $\mathcal{L} \otimes \mathcal{I}_Q$ is generated by global sections,
    
    \item[\hypertarget{assumptions}{(4)}] for any $\overline{k}$-point $x\in X_{\overline{k}}\setminus Q_{\overline{k}}$, we have that
    \[
    H^0(X,\mathcal{L}\otimes \mathcal{I}_Q)_{\overline{k}} \longrightarrow \mathcal{L}_x/m^2_x\mathcal{L}_x
    \]
    is surjective,
    
    \item[\hypertarget{assumptions}{(5)}] for any curve $C$ in the linear system $|\mathcal{L}|$, the arithmetic genus $p_a(C)\leq n$.
\end{enumerate}
\end{assumptions}

\begin{remark}
    Notice that assumption \hyperlink{assumptions}{(5)} makes sense since $C$ is a geometrically connected curve (see \ref{remarkgeomconnected}).
\end{remark}

\begin{lemma}
    Assumptions \hyperlink{assumptions}{(1)} and \hyperlink{assumptions}{(3)} imply that $\mathcal{L}$ is generated by global sections. 
\end{lemma}

\begin{proof}
    A line bundle $\mathcal{L}$ on $X$ is generated by global sections if and only if for any point $x\in X$ there is a section of $\mathcal{L}$ not vanishing at $x$. 
    Therefore, it is enough to check this property on closed points. 

    Let $x$ be a closed point.
    If $x\notin Q$ then $\mathcal{L}_x\cong \mathcal{L}_x\otimes \mathcal{I}_{Q,x}$ and it is enough to pick a right section in $H^0(X,\mathcal{L}\otimes \mathcal{I}_Q)$, which exists by \hyperlink{assumptions}{(3)}. 
    If $x\in Q$ then by \hyperlink{assumptions}{(1)} it is enough to pick a lift of a section of $H^0(X,\mathcal{L}|_{Q})$ not vanishing at $x$.  
\end{proof}

\begin{remark}\label{remarksurjectivity}
    By assumption \hyperlink{assumptions}{(3)}, \hyperlink{assumptions}{(4)} is equivalent to 
    the fact that for each $\overline{k}$-point $x\in X_{\overline{k}}\setminus Q_{\overline{k}}$ the sections $H^0(X,\mathcal{L}\otimes \mathcal{I}_Q)_{\overline{k}}$ separate tangent vectors in $X_{\overline{k}}\setminus Q_{\overline{k}}$. 
    Namely, for each $\overline{k}$-point $x\in X_{\overline{k}}\setminus Q_{\overline{k}}$, the set 
    \[
        \{ s\in H^0(X,\mathcal{L}\otimes \mathcal{I}_Q)_{\overline{k}}: s_x\in m_x\mathcal{L}_x\}
    \] spans the $\overline{k}$-vector space $m_x\mathcal{L}_x/m^2_x\mathcal{L}_x$,
\end{remark}

\begin{remark}\label{rmkaboutitem5}
    By Riemann-Roch for surfaces, for any element $C$ of the linear system $|\mathcal{L}|$ we have
    \[
        p_a(C) = \frac{1}{2} C \cdot (C+K_X) +1
               = \frac{1}{2} c_1(\mathcal{L}) \cdot (c_1(\mathcal{L})+K_X) +1.
    \] 
    Therefore assumption \hyperlink{assumptions}{(5)} is equivalent to
    \[
       \frac{ c_1(\mathcal{L})^2 + c_1(\mathcal{L}) \cdot K_X }{2}+1 \leq n.
    \]
\end{remark}

\subsection{Non-emptiness of $\mathbf{U_n}$ and $\mathbf{U_{n+d}}$}

\begin{proposition}\label{propositionaboutWnnonempty}
    The opens $U_n\subseteq \Hilb^n_X$ and $U_{n+d}\subseteq \Hilb^{n+d}_X$ are non-empty.
\end{proposition}

    Let us focus on $U_n$.
    In order to show that it is non-empty, we can assume that the field $k$ is algebraically closed. 
    Hence, by definition of $U_n$ \ref{definitionofUn}, it is enough to find a closed $0$-dimensional subscheme $Z$ of length $n$ (defined over $\overline{k}$ such that 
    \[
    H^0(X, \mathcal{L}) \longrightarrow H^0(X, \mathcal{L}|_{Z \sqcup Q}  )
    \] 
    is surjective and
    \[
    pr_{H *}(pr^*_X\mathcal{L})\otimes_{\mathcal{O}_{\Hilb^n_X, Z}} k
    \longrightarrow 
    H^0(X, \mathcal{L})
    \]
    is an isomorphism.\\ 
    The latter condition is true for any point $Z$, due to flat base change (see \cite[Tag 02KH]{deJ}).
    So we need to check only the former.
    In order to do so, we prove a general lemma.

\begin{lemma}\label{Lemmasections}
Let $k$ be an algebraically closed field.
Assume that a triple $(X,\mathcal{L}, Q)$ satisfies assumption \hyperlink{assumptions}{(1)}.
Let $N$ denote $\dim_k H^0(X, \mathcal{L})$.
Then for any $0\leq m\leq N-d$ there exists a degree $m$ effective $0$-cycle $Z_{m}$ such that 
\begin{enumerate}
    \item it is reduced,
    \item it is disjoint from $Q$, and
    \item the map
    \[
    H^0(X, \mathcal{L})\longrightarrow H^0(X, \mathcal{L}|_{Z_{m}\sqcup Q})
    \]
    is surjective.
\end{enumerate}
\end{lemma}
    
\begin{proof}
    By assumption \hyperlink{assumptions}{(1)} we have $d$-global sections $s_1,\dots , s_d$ of $H^0(X, \mathcal{L})$ such that they form a basis of 
    $H^0(Q, \mathcal{L}|_Q)$.
    So the statement is true for $m=0$.
    
    Consider now the greatest non-negative integer $m$ such that for any $m'\leq m$ we have
    \begin{itemize}
        \item a degree $m'$ effective reduced $0$-cycle $Z_{m'}\subset X$, disjoint from $Q$, and
        \item $m'$-sections $s_{d+1},\dots , s_{d+m'}$ of $H^0(X, \mathcal{L})$ such that $\{s_d,\dots,s_d,s_{d+1},\dots , s_{d+m'} \}$ is a basis of $H^0(X, \mathcal{L}|_{Z_{m'}\sqcup Q} )$.
    \end{itemize}
    Assume by contradiction that $ m < N-d$. 
    Then there exists a section $t\in H^0(X, \mathcal{L})$ linearly independent of $s_1,\dots , s_{d+m}$.
    Since $\{ s_1,\dots , s_{d+m} \}$ is a basis of $H^0(Z_m\sqcup Q, \mathcal{L}|_{Z_m\sqcup Q }  )$, we can write 
    \[
        t|_{Z_m\sqcup Q }=\sum_{i=1}^{d+m} a_is_i|_{Z_m\sqcup Q} \quad \textup{ with } a_i\in k.
    \] 
    Hence, changing $t$ with $t-\sum_{i=1}^{d+m} a_is_i$ we can assume that $t|_{Z_m\sqcup Q}=0$.
    
    Consider the open  $X \setminus \{ x\in X : t(x)=0 \in k(x) \}$: it is non-empty, and therefore it has a $k$-point $z$. 
    Then set $Z_{m+1}=Z_m\sqcup z$: since $Z_m\sqcup Q\subseteq \{ x\in X : t(x)=0 \in k(x) \}$, $Z_{m+1}$ is a well-defined effective $0$-cycle of degree $m+1$, disjoint from $Q$.

    Consider the restriction map
    \begin{equation}\label{equationtoZmandz}
        H^0(X, \mathcal{L}) 
        \longrightarrow 
        H^0(Z_{m+1}\sqcup Q, \mathcal{L}|_{Z_{m+1}\sqcup Q}  ).
    \end{equation}
    Then the elements $s_1|_{Z_{m+1}\sqcup Q},\dots , s_{d+m}|_{Z_{m+1}\sqcup Q}, t|_{Z_{m+1}\sqcup Q}$ form a basis of $H^0(Z_{m+1}\sqcup Q, \mathcal{L}|_{Z_{m+1}\sqcup Q}  )$.
    Indeed the matrix corresponding to \ref{equationtoZmandz} looks like
    \[
        \begin{bmatrix}[ccc|cccc]
        &   &   & a_1 & * & \dots & * \\
        \quad & Id_{d+m} &   & \vdots & \vdots & \ddots & \vdots \\
        &   &   & a_{d+m} & * & \dots & *\\
        \hline
        0 & \dots & 0 & b & * & \dots & *
        \end{bmatrix}.
    \]
    However since the element
    \[
        t|_{Z_m\sqcup Q}=\sum_{i=1}^{d+m} a_is_i|_{Z_m \sqcup Q}
    \] 
    is trivial and the elements $s_i|_{Z_m \sqcup Q}$'s are linearly independent, then $a_i=0$ for any $i=1,\dots, d+m$.
    Since $t|_{z}=t(z)\neq 0$ then $b\neq 0$ and the above matrix has full rank. 
    Therefore we can increase $m$ by $1$ and hence $m\geq N-d$.
\end{proof}

\begin{proof}[Proof of Proposition \ref{propositionaboutWnnonempty}]
    Let us recall that we are still working under the assumption that $k$ is algebraically closed.
    By assumption \hyperlink{assumptions}{(2)} the dimension $N$ is greater than $n+2d$. 
    Then by lemma \ref{Lemmasections} we can find a closed subscheme $Z$ for which the restriction map is surjective.
    Same proof goes for $U_{n+d}$: so $U_n$ and $U_{n+d}$ are both non-empty.
\end{proof}

\subsection{Non-emptiness of $\mathbf{P(\mathcal{E}_n)}$ and $\mathbf{P(\mathcal{E}_{n+d})}$}

\begin{lemma}\label{lemmaPEnintegralanddimoffibers}
    The scheme $\mathbb{P}(\mathcal{E}_n)$ (resp. $\mathbb{P}(\mathcal{E}_{n+d})$) is integral of dimension at least 
    \[
        n + \dim_{k} H^0(X,\mathcal{L}) - (d+1).
    \]
    (resp. $(n+d) + \dim_{k} H^0(X, \mathcal{L})-(d+1)$).
    In particular, it is non-empty (same holds for $\mathbb{P}(\mathcal{E}_{n+d})$).
\end{lemma}

\begin{proof}
Consider the projective-bundle $\mathbb{P}(\mathcal{E}_n) \longrightarrow U_n$ (for $\mathbb{P}(\mathcal{E}_{n+d})$ it is analogous).
For any point $y\in U_n$ we have
\[
    \dim_{k(y)} H^0(X_y, \mathcal{L}_y \otimes \mathcal{I}_{\mathcal{Z}_y\sqcup Q_y} ) \geq \dim_{k(y)} H^0(X_y, \mathcal{L}_y)-(n+d) 
\] 
Therefore, the dimension of $\mathbb{P}(\mathcal{E}_n)$ is at least 
\[
    \dim U_n + \min_{y\in U_n} 
    \dim_{k(y)} H^0(X_y, \mathcal{L}_y \otimes \mathcal{I}_{\mathcal{Z}_y\sqcup Q_y} ) -1 \geq 2n + \dim_{k} H^0(X, \mathcal{L})-(n+d+1).
\] 
By assumption \hyperlink{assumptions}{(2)}, $\dim_{k} H^0(X, \mathcal{L})\geq  n+1+2d$: so we get 
\[
    \dim \mathbb{P}(\mathcal{E}_n) \geq 2n +d 
\]
In particular it is non-empty.\\

For what concerns being integral, let us notice first that $U_n$ is integral, because it is an open of the integral scheme $\Hilb^n_X$.
Therefore, since $\mathcal{E}_n$ is locally free, $\Sym^*(\mathcal{E}_n)$ is a sheaf of integral domains on $U_n$: so $\mathbb{P}(\mathcal{E}_n)=\underline{\textup{Proj}}(\Sym^*(\mathcal{E}_n))$ is integral as well.
\end{proof}

\subsection{Non-emptiness of $\mathbf{V_n}$ and $\mathbf{V'_{n+d}}$}\, \\

It is now time to turn our attention to the opens $V_n\subseteq \mathbb{P}(\mathcal{E}_n)$ and $V'_{n+d}\subseteq \mathbb{P}(\mathcal{E}_{n+d})$. 
Recall that by definition, $V_n$ is the open made of points $y\in \mathbb{P}(\mathcal{E}_n)$ for which 
\begin{enumerate}
    \item the map $g$ in \ref{definitionofg} is smooth 
    (call the open where this happens $V$), and
    \item Cohomology and base change holds for the map
    \begin{equation}\label{eqcohombasechange1}
        g_*\mathcal{O}_{\mathcal{C}_V}(\mathcal{Z}_V+Q_{V}) \otimes_{\mathcal{O}_{V, y}} k(y) \longrightarrow H^0(\mathcal{C}_y, \mathcal{O}_{\mathcal{C}_y}(\mathcal{Z}_y+Q_y)).
    \end{equation}
\end{enumerate}

Similarly, $V'_{n+d}$ is the open made of points $y\in \mathbb{P}(\mathcal{E}_{n+d})$ for which 
\begin{enumerate}
    \item the map $g'$ is smooth (call the open where this happens $V'$), and
    \item Cohomology and base change holds for the map
    \begin{equation}\label{eqcohombasechange2}
        g'_*\mathcal{O}_{\mathcal{C}_{V'}}(\mathcal{Z}'_{V'}-Q_{V'}) \otimes_{\mathcal{O}_{V', y}} k(y) \longrightarrow H^0(\mathcal{C}_y, \mathcal{O}_{\mathcal{C}_y}(\mathcal{Z}'_y - Q_y)).
    \end{equation}
\end{enumerate}

We proceed to first check that the opens $V$ and $V'$ are non-empty, and only subsequently we focus on $V_n$ and $V'_{n+d}$.

\begin{proposition}[Bertini's theorem with a fixed closed]\label{bertiniwithclosedfixed}
    The open locus $V$ (resp. $V'$) where $g$ (resp. $g'$) is smooth are non-empty (and integral).
\end{proposition}

\begin{proof}
    The proof is a modification of the proof of classical Bertini's theorem (\cite[Thm. II.8.18]{Harty}) and de Jong's proof (\cite[Tag. 0FD6]{deJ}).
    Let $\Sigma$ be the singular locus in $\mathcal{C}$, namely the closed subscheme of points in $\mathcal{C}$ where $g$ is not smooth. 
    We need to check that $\Sigma$ is not the whole $\mathcal{C}$.
    
    The image $g(\Sigma)$ is a closed subset of the integral scheme $\mathbb{P}( \mathcal{E}_n )$.
    Since $g$ is surjective, in order to show that $\Sigma\neq \mathcal{C}$ it is enough to prove that
    \[
        \dim g(\Sigma) < \dim \mathbb{P}( \mathcal{E}_n ).
    \] 
    Dimensions are unaltered by passing to the algebraic closure of $k$: so we can assume that $k=\overline{k}$.

    \begin{claim*} 
    Let $x$ be a $k$-point of $X$. 
    Consider the fiber $\Sigma_x$ via the map
    \[
        \Sigma \subseteq \mathcal{C} \subset X \times \mathbb{P}(\mathcal{E}_n)\overset{pr_X}{\longrightarrow} X.
    \]
    Then
    \begin{enumerate}
        \item for all but finitely many $k$-points $x\in X$ we have
        $\dim \Sigma_x \leq \dim  \mathbb{P}( \mathcal{E}_n)  -3$, and
        \item for the remaining $k$-points we have $\dim \Sigma_x \leq \dim  \mathbb{P}( \mathcal{E}_n ) - 1$:
    \end{enumerate}
    \end{claim*}

    \begin{proof}[Proof of the Claim]
    By definition of $\Sigma$ the set $\Sigma_x(k)$ coincides with the pairs 
    \[
        \{ (Z,C) \in U_n(k) \times \mathbb{P}(H^0(X,\mathcal{L})^{\vee})(k): Z\sqcup Q \subset C, \textup{ and } x\in C \textup{ is singular}  \}.
    \] 
    In terms of sections $f\in H^0(X,\mathcal{L}\otimes \mathcal{I}_{Q})^*$, $\Sigma_x(k)$ can be described as 
    \[
        \{ (Z,f):  Z \sqcup Q \subset V(f)\subset X,\textup{ and } x \in V(f) \textup{  is singular} \}/ f\sim \lambda f'.
    \] 
    Now $x\in V(f)$ if and only if $f_x\in m_x\mathcal{L}_x$, and $x\in V(f)$ is singular if and only if $f_x\in m^2_x\mathcal{L}_x$. 
    Indeed being singular at $x$ means 
    \[
        \dim_k m_x/(m^2_x,f_x)  > 1 = \dim \mathcal{O}_{V(f), x}.
    \] 
    If $f_x\in m_x^2$, then the dimension on the right is 2 and if $f_x\in m_x\setminus m^2_x$ the dimension drops to $1$. 
    So the sections $f$ for which $x$ is singular (together with the zero-section) correspond to the kernel
    \[
        \phi_x: H^0(X,\mathcal{L}\otimes \mathcal{I}_Q) 
        \longrightarrow \mathcal{L}_x/m^2_x\mathcal{L}_x.
    \] 
    Therefore the set $\Sigma_x(k)$ coincides with
    \[
        \{(Z,f)\in U_n(k) 
        \times 
        \mathbb{P}(\ker(\phi_x)^{\vee})(k): Z\in \Hilb^n_{V(f)} \}.
    \] 
    This implies that the dimension of $\Sigma_x$ is smaller or equal to (we use assumption \hyperlink{assumptions}{(1)})
    \begin{align*}
       \dim \mathbb{P}( \ker(\phi_x)^{\vee} )  + \max_{f\neq 0} \dim \Hilb^n_{V(f)} = \dim H^0(X,\mathcal{L}\otimes \mathcal{I}_Q) - 1 - \dim \textup{im}(\phi_x) + n
        = 
        \\
        = 
        n+\dim_k H^0(X,\mathcal{L}) - (d + 1) - \dim \textup{im}(\phi_x) 
        \leq
        \dim \mathbb{P}( \mathcal{E}_n ) - \dim \textup{im}(\phi_x).
    \end{align*}
    If $x$ does not belong to $Q$ then by assumptions \hyperlink{assumptions}{(3)} and \hyperlink{assumptions}{(4)} we have that $\phi_x$ is surjective (see remark \ref{remarksurjectivity}). 
    So in this case 
    \[
        \dim \Sigma_x \leq \dim \mathbb{P}( \mathcal{E}_n) - 3.
    \] 
    If $x\in Q$ then $\mathcal{L}_x\otimes \mathcal{I}_{Q,x} = \mathcal{L}_x\otimes m_x$ ($Q$ is reduced at $x$). 
    Since $\mathcal{L}\otimes \mathcal{I}_Q$ is generated by global sections, the map $\phi_x$ is not the zero-map, and so the image has dimension at least 1.
    Therefore
    \[
        \dim \Sigma_x \leq \dim \mathbb{P}( \mathcal{E}_n )-1.
    \]   
    \end{proof}
    
    From the bounds of the claim it follows directly that $\dim g(\Sigma)$  is smaller that $\dim \mathbb{P}(\mathcal{E}_n)$ (see for instance \cite[Tag 0B2L]{deJ}). 
    The proof for $V'$ is analogous.
\end{proof}

We can finally focus on the opens $V_n$ and $V'_{n+d}$.
In order to prove that they are non-empty it remains to check that there is a  point $y$ (resp. $y'$) for which equation \ref{eqcohombasechange1} (resp. \ref{eqcohombasechange2}) holds.
In order to do so, we prove a sufficient condition first.

\begin{lemma}\label{lemmasufficientconditiontostayinVn}
Let $y\in V$ : if $H^1(\mathcal{C}_y,\mathcal{O}_{\mathcal{C}_y}(\mathcal{Z}_y+Q_y))$ is trivial, then equation \ref{eqcohombasechange1} holds.

Analogously, let $y'\in V'$: if $H^1(\mathcal{C}_{y'}, \mathcal{O}_{\mathcal{C}_{y'}}(\mathcal{Z}'_{y'}-Q_{y'}))=0$ then equation \ref{eqcohombasechange2} holds.
\end{lemma}

\begin{proof}
    This is a standard Cohomology and base change argument.

    By upper semi-continuity in $y$ of $\dim_{k(y)} H^1(\mathcal{C}_y, \mathcal{O}_{\mathcal{C}_y}(\mathcal{Z}_y+Q_y))$, the set where the $H^1$ is trivial is open (\cite[Thm. III.12.8]{Harty}). 
    Then on a (connected) open neighbourhood $U_y$ of $y$ (which is integral because open of $V$, integral scheme) the sheaf $R^1g_*\mathcal{O}_{\mathcal{C}_y}(\mathcal{Z}_y+Q_y)$ is identically zero (see \cite[Cor. III.12.9]{Harty}).

    By Cohomology and base change \cite[Thm. III.12.11(b)]{Harty}, on $U_y$ the base change map at the level of $H^0$ is an isomorphism.
\end{proof}

\begin{proposition}\label{propositionVnnonempty}
    The opens $V_n$ and $V'_{n+d}$ are non-empty.
\end{proposition}

\begin{proof}
    Denote by $\CDiv_X^{Q\subset \mathcal{L}}$ the subfunctor of $\CDiv_X^{\mathcal{L}}$ of the divisors in $\CDiv_X^{\mathcal{L}}$ containing $Q$. 
    By lemma \ref{Lemmaclosedsubschemesfunctor} it is closed subscheme of $\CDiv_X^{\mathcal{L}}$.

    Denote by $\textup{Sm}\CDiv_X^{Q\subset \mathcal{L}}$ the open subscheme of $\CDiv_X^{Q\subset \mathcal{L}}$ consisting of divisors containing $Q$ and smooth over the base. 
    It is non empty because $V$ is non-empty.
    Then we have the following diagram (all the squares are cartesian):
    \[
    \begin{tikzcd}
    V
    \arrow[rr]
    \arrow[d, open]
    &
    &
    \textup{Sm}\CDiv^{Q\subset \mathcal{L}}_X
    \arrow[d, open]
    &
    \\
    \mathbb{P}(\mathcal{E}_n)
    \arrow[r, closed]
    &
    \mathcal{P}
    \arrow[r]
    \arrow[d, closed]
    &
    \CDiv^{Q\subset \mathcal{L}}_X
    \arrow[d, closed]
    &
    \\
    &
    \mathbb{P}((pr_{H *}(pr^*_X\mathcal{L})|_{U_n})^{\vee})
    \arrow[d]
    \arrow[r]
    &
    \CDiv^{\mathcal{L}}_X
    \arrow[d]
    \arrow[r]
    &
    \CDiv_X
    \arrow[d, "{AJ_X}"]
    \\
    &
    U_n
    \arrow[r]
    &
    \Spec(k)
    \arrow[r, "{[\mathcal{L}]}"]
    &
    \Pic_X
    \end{tikzcd}
    \]
    Let $y\in \textup{Sm}\CDiv_X^{Q\subset \mathcal{L}}$ be a geometric point such that the fiber $V_y$ is non-empty.
    Namely, $y$ corresponds to a smooth curve $C\subset X_{k(y)}$ containing $Q_{k(y)}$ (and $k(y)=\overline{k(y)}$).
    
    The $k(y)$-points of $V_y$ are exactly those closed subscheme $Z\subset C$ of degree $n$ such that $Z\in U_n(k(y))$. 
    More generally, the $T$-points of $V_y$ are the closed subschemes $Z\subset C_T$, flat and finite of degree $n$ over $T$, such that $Z\in U_n(T)$.
    In other words, the fiber $V_y$ coincides with $U_n\times_{\Hilb^n_X} \Hilb^n_{C/k(y)}$, an open subscheme of $\Hilb^n_{C/k(y)}$.\\

    Consider now the Abel-Jacobi map for the curve $C$
    \begin{equation}\label{AbelJAcobimapforcurves}
        AJ_C:\Hilb^n_{C/k(y)}\longrightarrow \Pic^n_{C/k(y)}.
    \end{equation}
    By assumption \hyperlink{assumptions}{(5)} the genus of $C$ is less or equal than the degree $n$: by Riemann-Roch, this means that the Abel-Jacobi map \ref{AbelJAcobimapforcurves} is surjective.
    Therefore, since $\Hilb^n_{C/k(y)}$ is integral,
    the composite
    \[
        V_y \xhookrightarrow{} \Hilb^n_{C/k(y)} \xrightarrow[]{AJ_{C}}  \Pic^n_{C/k(y)}
    \]
    is dominant. 
    Further composing with the isomorphism
    \[
        \Pic^{n}_{C/k(y)} \overset{\sim}{\longrightarrow} \Pic^{n+d}_{C/k(y)}, \quad  \mathcal{M}\mapsto \mathcal{M}\otimes \mathcal{O}_{C}(Q_{k(y)}),
    \]
    we get another dominant map
    \[
        V_y \xhookrightarrow{} \Hilb^n_{C/k(y)} \longrightarrow \Pic^n_{C/k(y)} \longrightarrow \Pic^{n+d}_{C/k(y)}.
    \]
    Consider the closed subschemes $W^{n+d-p_a(C)}_{n+d}(C/k(y))$ and $W^{n+d-p_a(C) +1}_{n+d}(C/k(y))$ inside $\Pic^{n+d}_{C/k(y)}$.
    By Riemann-Roch $W^{n+d-p_a(C)}_{n+d}(C/k(y))$ coincides with $\Pic_{C/k(y)}^{n+d}$.
    By assumption \hyperlink{assumptions}{(5)}, $p_a(C)+n+d-p_a(C)-n-d\geq 0$: by the Brill-Noether theory fact \ref{factBrillNoether}, we have
    \[
        \Pic^{n+d}_{C/k(y)} \setminus W^{n+d-p_a(C) +1}_{n+d}(C/k(y))
    \]
    is non-empty, open and hence integral (because $\Pic^{n+d}_{C_y/k(y)}$ is integral). 
    
    The $K$-points of $\Pic^{n+d}_{C/k(y)} \setminus W^{n+d-p_a(C) +1}_{n+d}(C/k(y))$ are exactly those classes of line bundles $\mathcal{O}_{C_K}(Z+Q_K) \in \Pic^{n+d}_{C/k(y)}(K)$ such that $H^0(C_K, \mathcal{O}_{C_K}(Z+Q_K))= n+d-p_a(C)+1$ or equivalently $H^1(C_K, \mathcal{O}_{C_K}(Z+Q_K))=0$. 
            
    The preimage of $\Pic^{n+d}_{C/k(y)} \setminus W^{n+d-p_a(C) +1}_{n+d}(C/k(y))$ in $V_y$ is a non-empty open: by lemma \ref{lemmasufficientconditiontostayinVn} $V_n$ is non-empty. \\

    The proof for the open $V'_{n+d}$ is the same: in this case we set $m=n$ and $r=n-p_a(C)$ (still non negative by assumption \hyperlink{assumptions}{(5)}). 
    Again by fact \ref{factBrillNoether}, the open $\Pic^n_{C/k(y')}\setminus W^{n-p_a(C)+1}_n(C/k(y'))$ is non-empty which gives a non-empty open of $V'_{y'}$ via the map
    \[
         V'_{y'} \xhookrightarrow{} \Hilb^{n+d}_{C/k(y')}\longrightarrow \Pic^{n+d}_{C/k(y')} \xrightarrow[]{-\otimes \mathcal{O}_C(-Q_{k(y')})} \Pic^n_{C/k(y')}.
    \] 
    Lemma \ref{lemmasufficientconditiontostayinVn} concludes the proof as before.
\end{proof}
    
\section{Stable Birationality between Hilbert schemes}\label{6}

As checked in the previous section, we have four non-empty integral opens
\[
    U_n,\ U_{n+d},\ V_n \textup{ and } V'_{n+d} 
\] 
such that they fit in the cartesian diagram:
\[
    \begin{tikzcd}
    &
    &
    \mathbb{P}(\mathcal{F}'_{n+d})
    \arrow[r, "{p_{n+d}'}"]
    \arrow[dd, open]
    &
    V'_{n+d}
    \arrow[d, open]
    &
        \\
    &
    &
    &
    \mathbb{P}(\mathcal{E}_{n+d})
    \arrow[r, two heads]
    \arrow[d, open]
    &
    U_{n+d}
    \arrow[d, open]
        \\
    \mathbb{P}(\mathcal{F}_{n})
    \arrow[rr, open]
    \arrow[d, "{p_n}"']
    &
    &
    \Inc^Q_{X,\mathcal{L},n,n'}
    \arrow[d]
    \arrow[r]
    &
    \Inc^Q_{X,\mathcal{L},n'}
    \arrow[r]
    &
    \Hilb^{n+d}_X
        \\
    V_n
    \arrow[r, open]
    &
    \mathbb{P}(\mathcal{E}_n)
    \arrow[r, open]
    \arrow[d, two heads]
    &
    \Inc^Q_{X,\mathcal{L},n}
    \arrow[d]
    &
    &
        \\
    &
    U_n
    \arrow[r, open]
    &
    \Hilb^n_X.
    &
    &
    \end{tikzcd}
\]

\begin{remark}
    The projections $p_n:\mathbb{P}(\mathcal{F}_n)\longrightarrow V_n$ and $p_{n+d}': \mathbb{P}(\mathcal{F}'_{n+d})\longrightarrow V'_{n+d}$ are surjective as well.
    Indeed the rank of $\mathcal{F}_n$ at $y\in V_n$ is 
    \[
        \dim_{k(y)} H^0(\mathcal{C}_y, \mathcal{O}_{\mathcal{C}_y}(\mathcal{Z}_y+Q_y) ) \geq \deg (\mathcal{Z}_y+Q_y) - p_a(C_y) + 1 = n+d -p_a(C)+1\geq 1, 
    \]
    and the rank of $\mathcal{F}'_{n+d}$ at $y\in V'_{n+d}$ is 
    \[
        \dim_{k(y)} H^0(\mathcal{C}_y, \mathcal{O}_{\mathcal{C}_y}(\mathcal{Z}'_y-Q_y) ) \geq \deg (\mathcal{Z}'_y-Q_y) - p_a(C_y) + 1 = n -p_a(C)+1\geq 1, 
    \]
\end{remark}
Now we want to check that the two (integral) opens $\mathbb{P}(\mathcal{F}_n)$ and $\mathbb{P}(\mathcal{F}'_{n+d})$ intersect.

\begin{proposition}
    The two non-empty open subschemes $\mathbb{P}(\mathcal{F}_n)$ and $\mathbb{P}(\mathcal{F}_{n+d}')$ intersect in $\Inc^Q_{X,\mathcal{L},n,n+d}$.
\end{proposition}

\begin{proof}
    Consider a geometric point $y\in V'_{n+d}$ for which the group $H^1(\mathcal{C}_y,\mathcal{O}_{\mathcal{C}_y}(\mathcal{Z}_y'-Q_y))$ is trivial.
    Pick a point $(Z,\mathcal{Z}_y',\mathcal{C}_y)\in \mathbb{P}(\mathcal{F}'_{n+d})(k(y))$ (it exists because $p_{n+d}'$ is surjective).
    By definition, $Z$ is an effective divisor on $\mathcal{C}_y$ equivalent to $\mathcal{Z}_y'-Q_y$ and disjoint from $Q_y$.
    We claim that $(Z,\mathcal{Z}_y,\mathcal{C}_y)\in \mathbb{P}(\mathcal{F}_{n})(k(y))$: equivalently $(Z,\mathcal{C}_y)\in V_n(k(y))$.
    
    Let us take the cohomology sequence associated to the short exact sequence
    \[
        0\longrightarrow
        \mathcal{O}_{\mathcal{C}_y}(\mathcal{Z}_y'-Q_y)
        \longrightarrow
        \mathcal{O}_{\mathcal{C}_y}(\mathcal{Z}_y')
        \longrightarrow
        \mathcal{O}_{\mathcal{C}_y}(\mathcal{Z}_y')|_{Q_y}
        \longrightarrow 0.
    \] 
    We obtain
    \[
        H^0(\mathcal{C}_y, \mathcal{O}_{\mathcal{C}_y}(\mathcal{Z}_y')|_{\mathcal{Q}_y} )
        \longrightarrow 0 \longrightarrow 
        H^1(\mathcal{C}_y, \mathcal{O}_{\mathcal{C}_y}(\mathcal{Z}_y') )
        \longrightarrow  0.
    \]
    Therefore $H^1(\mathcal{C}_y, \mathcal{O}_{\mathcal{C}_y}( Z + Q_y) )=0$. 
    In particular Cohomology and Base change holds true for $\mathcal{O}_{\mathcal{C}_y}(Z+Q_y)$. \\

    \noindent It remains to prove that
    \[
        H^0(X_y, \mathcal{L}_y)\longrightarrow
        H^0(X_y, \mathcal{L}_y|_{Z \sqcup Q_y})
    \]
    is surjective.
    In order to show this, we consider the diagram 
    \[
        \begin{tikzcd}[column sep=0.1cm]
            H^0(X_y,\mathcal{L}_y)
            \arrow[r]
            &
            H^0(\mathcal{C}_y,\mathcal{L}_y|_{\mathcal{C}_y})
            \arrow[r]
            \arrow[d]
            &
            H^0(Z \sqcup Q_y,\mathcal{L}_y|_{Z \sqcup Q_y})
            \\
            &
            H^0(\mathcal{Z}'_y\sqcup Q_y,\mathcal{L}_y|_{\mathcal{Z}'_y\sqcup Q_y}).
            &
        \end{tikzcd}
    \] 
    Since the restriction from $X_y$ to $\mathcal{Z}_y'\sqcup Q_y$ is surjective, then the restriction from $\mathcal{C}_y$ 
    \[
        H^0(\mathcal{C}_y,\mathcal{L}_y|_{\mathcal{C}_y})
        \longrightarrow 
        H^0(\mathcal{Z}_y'\sqcup \mathcal{Q}_y,\mathcal{L}_y|_{\mathcal{Z}_y'\sqcup \mathcal{Q}_y})
    \] 
    is surjective as well. 
    Moreover the map 
    \[
        H^0(X_y,\mathcal{L}_y)
        \longrightarrow
        H^0(\mathcal{C}_y,\mathcal{L}_y|_{\mathcal{C}_y})
    \] 
    is always surjective because it is induced by the short exact sequence
    \[
        0\longrightarrow 
        \mathcal{I}_{\mathcal{C}_y}\otimes \mathcal{L}_y
        \longrightarrow
        \mathcal{L}_y
        \longrightarrow
        \mathcal{L}_y|_{\mathcal{C}_y}
        \longrightarrow 0
    \] 
    and $H^1(X_y, \mathcal{I}_{\mathcal{C}_y}\otimes \mathcal{L}_y)\cong H^1(X_y,\mathcal{O}_{X_y})=0$ by hypothesis on $X$.
    Therefore we just have to prove that
    \[
        H^0(\mathcal{C}_y,\mathcal{L}_y|_{\mathcal{C}_y})
        \longrightarrow
        H^0(Z \sqcup Q_y,\mathcal{L}_y|_{Z \sqcup Q_y})
    \] 
    is surjective.
    Equivalently we have to prove that
    \begin{equation}\label{surjectivityoncurveforpoints}
        H^1(\mathcal{C}_y, \mathcal{L}_y|_{\mathcal{C}_y}(-
        Z-Q_y) )
        \longrightarrow
        H^1(\mathcal{C}_y, \mathcal{L}_y|_{\mathcal{C}_y})
        \longrightarrow 0
    \end{equation} 
    is injective. 
    Since this map is always surjective, injectivity is equivalent to have the following inequality:
    \[
        \dim_{k(y)} H^1(\mathcal{C}_y, \mathcal{L}_y|_{\mathcal{C}_y}) 
        \geq
        \dim_{k(y)}
        H^1(\mathcal{C}_y, \mathcal{L}_y|_{\mathcal{C}_y}(-
        Z-Q_y) ).
    \]
    Since the map $H^0(\mathcal{C}_y,\mathcal{L}_y|_{\mathcal{C}_y}) \longrightarrow H^0(\mathcal{Z}_y'\sqcup Q_y,\mathcal{L}_y|_{\mathcal{Z}_y'\sqcup Q_y})$ is surjective, then the left-hand side is equal to
    \[
        \dim_{k(y)} H^1(\mathcal{C}_y, \mathcal{L}_y|_{\mathcal{C}_y}(-\mathcal{Z}_y'-
        Q_y) ).
    \] 
    Using the short exact sequence
    \[
        0\longrightarrow
        \mathcal{O}_{\mathcal{C}_y}(-Q_y)
        \longrightarrow
        \mathcal{O}_{\mathcal{C}_y}
        \longrightarrow
        \mathcal{O}_{Q_y}
        \longrightarrow 0
    \] 
    tensored by $\mathcal{L}_y|_{\mathcal{C}_y}(-\mathcal{Z}_y')$,
    we get at the cohomology level the following exact sequence
    \begin{align*}
        0\longrightarrow
        H^0(\mathcal{C}_y, \mathcal{L}_y|_{\mathcal{C}_y}(-\mathcal{Z}_y'-Q_y))
        \longrightarrow
        H^0(\mathcal{C}_y, \mathcal{L}_y|_{\mathcal{C}_y}(-\mathcal{Z}_y'))
        \longrightarrow
        H^0(Q_y, \mathcal{O}_{Q_y} )
        \longrightarrow
        \\
        \longrightarrow
        H^1(\mathcal{C}_y, \mathcal{L}_y|_{\mathcal{C}_y}(-\mathcal{Z}_y'-Q_y))
        \longrightarrow
        H^1(\mathcal{C}_y, \mathcal{L}_y|_{\mathcal{C}_y}(-\mathcal{Z}_y'))
        \longrightarrow 0.
    \end{align*} 
    In particular we get that
    \[
        \dim_{k(y)} H^1(\mathcal{C}_y, \mathcal{L}_y|_{\mathcal{C}_y}(-\mathcal{Z}_y'-
        Q_y))
        \geq 
        \dim_{k(y)} H^1(\mathcal{C}_y, \mathcal{L}_y|_{\mathcal{C}_y}(-Z-Q_y)),
    \] 
    which is the wanted inequality. 
    Therefore the point $(Z, C_y)\in V_n(k(y))$ and the two opens intersect.
\end{proof}

\begin{corollary}[Theorem \ref{theorem1}]\label{corollary=thm1}
If a suitable triple $(X,\mathcal{L},Q)$ satisfies assumptions \hyperlink{asumptions}{(1)} to \hyperlink{asumptions}{(5)} with respect to the integer $n$, then $\Hilb^n_{X}$ is stably birational to $\Hilb^{n+d}_X$.
\end{corollary}

\begin{proof}
    Pick a point $y=(Z,Z',C)$ in the intersection of $\mathbb{P}(\mathcal{F}_n)$ and $\mathbb{P}(\mathcal{F}'_{n+d})$.
    Consider the images $p_n(y)=(Z,C)$ and $p_{n+d}'(y)=(Z',C)$ in $V_n$ and $V'_{n+d}$.
    Let $U_Z$ be an integral open neighbourhood of $Z\in U_n$, which trivializes the projective bundle $\mathbb{P}(\mathcal{E}_n)$:
    \[
        \begin{tikzcd}
            U_Z \times \mathbb{P}^{\textup{rk}_Z(\mathcal{E}_n)-1}_k
            \arrow[r,open]
            \arrow[d]
            &
            \mathbb{P}(\mathcal{E}_n)
            \arrow[d]
            \\
            U_Z
            \arrow[r, open]
            &
            U_n.
        \end{tikzcd}
    \]
    Since $\mathbb{P}(\mathcal{E}_n)$ is integral, the open $U_Z \times \mathbb{P}^{\textup{rk}_Z(\mathcal{E}_n)-1}_k$ is integral as well.
    Define $U_{Z'}$ in the same way.
    
    Let $V_{(Z,C)}$ be an integral open neighbourhood of $p_n(y)$ contained in $V_n\cap U_Z \times \mathbb{P}^{\textup{rk}_Z(\mathcal{E}_n)-1}_k$, which trivializes the projective bundle $\mathbb{P}(\mathcal{F}_n)$:
    \[
    \begin{tikzcd}
        V_{(Z,C)} \times \mathbb{P}^{\textup{rk}_{(Z,C)}(\mathcal{F}_n)-1}_k
        \arrow[d, "{p_n|_{V_{(Z,C)}}}"']
        \arrow[rr, open]
        &
        &
        \mathbb{P}(\mathcal{F}_n)
        \arrow[d, "{p_n}"]
        \\
        V_{(Z,C)}
        \arrow[r, open]
        &
        V_n\cap U_Z \times \mathbb{P}^{\textup{rk}_Z(\mathcal{E}_n)-1}_k
        \arrow[r, open]
        &
        V_n 
    \end{tikzcd}
    \]
    Since $V_{(Z,C)}$ and $U_Z \times \mathbb{P}^{\textup{rk}_Z(\mathcal{E}_n)-1}_k$ are both integral neighbourhoods of $p_n(y)$, they are birational.
    Define $V'_{(Z',C)}$ in the same way.
    Therefore, we have that
    \[
        U_Z \times \mathbb{P}^{\textup{rk}_Z(\mathcal{E}_n)+ \textup{rk}_{(Z,C)}(\mathcal{F}_n)-2}_k
        \sim_{\textup{bir}}
        V_{(Z,C)}\times \mathbb{P}^{\textup{rk}_{(Z,C)}(\mathcal{F}_n)-1}_k 
    \]
    and 
    \[
        U_{Z'} \times \mathbb{P}^{\textup{rk}_Z(\mathcal{E}'_{n+d})+\textup{rk}_{(Z',C)}(\mathcal{F}'_{n+d})-2}_k
        \sim_{\textup{bir}}
        V'_{(Z',C)} \times \mathbb{P}^{\textup{rk}_{(Z',C)}(\mathcal{F}'_{n+d})-1}_k.
    \]

    Since $V_{(Z,C)}, V'_{(Z',C')}$ are integral, then $V_{(Z,C)}\times \mathbb{P}^{\textup{rk}_{(Z,C)}(\mathcal{F}_n)-1}_k$ and $V'_{(Z',C)} \times \mathbb{P}^{\textup{rk}_{(Z',C)}(\mathcal{F}'_{n+d})-1}_k$ are integral as well.
    Since they are both neighbourhoods of $y$ they are birational.
    
    Therefore, the two opens $U_Z$ and $U_{Z'}$ are stably-birational, and so are $\Hilb^n_X$ and $\Hilb^{n+d}_X$.
\end{proof}

\section{Definitive behaviour of Stable Birational classes of $\Hilb^n_X$}\label{7}

In this section we are interested in finding sufficient and necessary criteria on $n$ and on $X$ so that we find suitable triples satisfying the \hyperlink{assumptions}{assumptions} \ref{thefiveassumptions} with respect to $n$ (in order to apply corollary \ref{corollary=thm1}).\\

Given a suitable triple $(X,\mathcal{L}, Q)$ as in the introduction, we want to check for which positive integers $(e,n)$ the triple $(X,\mathcal{L}^{\otimes e},Q)$ satisfies the \hyperlink{assumptions}{assumptions} with respect to $n$.

\begin{theorem}[Theorem \ref{theorem2}]\label{theoremgap1}
    Let $(X,\mathcal{L},Q)$ be a suitable triple. 
    Then there exists an integer $e_0=e_0(X,\mathcal{L},Q)$ such that for any $e\geq e_0$ the triple $(X,\mathcal{L}^{\otimes e},Q)$ satisfies the \hyperlink{assumptions}{assumptions} \ref{thefiveassumptions} with respect to $n$ if and only if $n$ is in the interval
    \[
        I_e\coloneqq \left[\frac{e^2 c_1(\mathcal{L})^2 +  e\, c_1(\mathcal{L})\cdot K_X}{2}+1, \frac{ e^2 c_1(\mathcal{L})^2-e\, c_1(\mathcal{L}) \cdot K_X }{2} + \dim_k H^2(X, \mathcal{O}_X) -2\deg(Q) \right] ,
    \]
    in which case the variety $\Hilb^n_X$ is stably birational to $\Hilb^{n+\deg(Q)}_X$.
\end{theorem}

\begin{proof}
Let us consider one assumption at a time.\\

\textbf{Assumption 1.} By Serre vanishing theorem, the group $H^1(X,\mathcal{L}^{\otimes e}\otimes \mathcal{I}_Q)=0$ for any $e>>0$.
Therefore, the restriction map 
\[
    H^0(X,\mathcal{L}^{\otimes e} )\longrightarrow H^0(X,\mathcal{L}^{\otimes e}|_Q)\longrightarrow 0
\]
is surjective for any $e>>0$.\\

\textbf{Assumption 2.} Since $\mathcal{L}$ is ample, again by Serre vanishing theorem $H^i(X,\mathcal{L}^{\otimes e} )=0$ for any $i>0$ and $e>>0$. 
Then, by Riemann Roch, we have that for any $e>>0$
\[
    \dim_k H^0(X, \mathcal{L}^{\otimes e} )=
    \chi (X, \mathcal{L}^{\otimes e} ) = 
    \frac{ e^2 c_1(\mathcal{L})^2-e\,  c_1(\mathcal{L}) \cdot K_X }{2} 
    + 1 +\dim_k H^2(X, \mathcal{O}_X).
\]

\textbf{Assumption 3.} Since $\mathcal{L}$ is ample, the sheaf $\mathcal{L}^{\otimes e} \otimes \mathcal{I}_Q$ is generated by global sections for any $e>>0$.\\

\textbf{Assumption 4.} 
We need to check that for each $\overline{k}$-point $x\in X_{\overline{k}}\setminus Q_{\overline{k}}$ the set 
\[
    \{ s\in H^0(X,\mathcal{L}\otimes \mathcal{I}_Q)_{\overline{k}}: s_x\in m_x\mathcal{L}_x\}
\] 
spans the $\overline{k}$-vector space $m_x\mathcal{L}_x/m^2_x\mathcal{L}_x$.
As seen in the proof of \cite[Tag 0FD5]{deJ}, if we pick $e \geq e_g+e_{va}$ where $\mathcal{L}^{e_g}\otimes \mathcal{I}_Q$ is globally generated and $\mathcal{L}^{e_{va}}$ is very ample then we have 
a locally closed immersion 
\[
X\setminus Q\longrightarrow \mathbb{P}(H^0(X,\mathcal{L}^{\otimes e}\otimes \mathcal{I}_Q)^{\vee} ).
\]
This is enough to ensure that assumption \hyperlink{assumptions}{(4)} is satisfied (see \cite[Prop. II.7.3]{Harty}).\\

\textbf{Assumption 5.} Finally, the arithmetic genus $p_a(C)$ of any curve $C\in |\mathcal{L}^{-\otimes e}|$ is equal to
\[
    \frac{c_1(\mathcal{L}^{\otimes e} )^2 +(c_1(\mathcal{L}^{\otimes e} )\cdot K_X )}{2} +1 = \frac{e^2 c_1(\mathcal{L})^2 +  e \, c_1(\mathcal{L})\cdot K_X}{2}+1.
\]

Therefore for $e>>0$ the only restrictions given by the \hyperlink{assumptions}{assumptions} are the inequalities about $n$.
\end{proof}

Notice that the theorem has a meaning only after one checks that some predicted intervals $I_e$ are non-empty (see subsection \ref{geometricallyrat}).

\subsection{Blow-ups of surfaces}

Let $\widetilde{X}$ be the blow up of $X$ along some geometrically reduced effective $0$-cycle $Q'$.
Then by fact \ref{factonstablebira} we have that $\Hilb^n_X$ is stably birational to $\Hilb^n_{\widetilde{X}}$.

We aim to exploit theorem \ref{theoremgap1} for $\widetilde{X}$ instead of $X$, and gain information on the stable birationality classes of $\Hilb^n_X$.
Let us fix the following general framework.

\begin{framework*}\label{framework}
    Let $(X,\mathcal{L},Q)$ be a suitable triple.
    Let us furthermore fix a geometrically reduced effective $0$-cycle $Q'$ of degree $d'$, disjoint from $Q$.

    Consider the blow-up $\widetilde{X}=Bl_{Q'}X$ of $X$ at $Q'$.
    Call $E$ the exceptional divisor above $Q'$.
    Then we have the cartesian diagram
    \[
         \begin{tikzcd}
             E
             \arrow[r, "{i_E}", hook]
             \arrow[d, "{b_E}"]
             &
             \widetilde{X}
             \arrow[d, "{b}"]
             \\
             Q'
             \arrow[r, hook, "{i'}"]
             &
             X.
         \end{tikzcd}
    \]
    Denote by $\widetilde{\mathcal{L}}_e$ the line bundle $\widetilde{\mathcal{L}}_e=b^*(\mathcal{L}^{\otimes e})\otimes \mathcal{O}_{\widetilde{X}}(-2E)$ on $\widetilde{X}$. 
    Finally, let $\widetilde{Q}$ be the pullback $b^{-1}(Q)$ in $\widetilde{X}$.

    Now we want to study the triples $(\widetilde{X}, \widetilde{\mathcal{L}}_e, \widetilde{Q})$.
    These notations will be used consistently throughout this subsection.
\end{framework*}

Before proceeding further, let us now see that given an effective $0$-cycle $Z$ on any (smooth, projective, geometrically connected) surface $X$ there is always a smooth effective $0$-cycle $Q'$ of degree $d'\geq 2\deg(Q)$ disjoint from $Z$.
Recall first the usual avoidance lemma.

\begin{lemma}\label{LemmaavoidQ}
Given an effective $0$-cycle $Z$ in $\mathbb{P}^m_k$, there exists a non-empty (and hence dense) open $U \subseteq (\mathbb{P}^m_k)^{*}$ of hyperplanes of $\mathbb{P}^m_k$ avoiding $Z$. 

If $k$ is infinite, this means that there exist infinitely many $k$-hyperplanes in $\mathbb{P}^m_k$ avoiding $Z$.
\end{lemma}

\begin{proof}
Avoiding $Z$ is the same as avoiding its reduction $Z_{\textup{red}}$: then we can assume $Z$ to be reduced.
Write $Z$ as the disjoint union of its integral (disjoint) components $\sqcup Z_i$.
If an hyperplane $H$ intersect $Z_i$ then $Z_i\subset H$: hence the locus $U$ avoiding $Z$ is open.

Notice that $U$ is non-empty if and only if $U_{\overline{k}}$ is non-empty. 
Therefore, without loss of generality we can work on the algebraic closure $\overline{k}$.

Let $P_i$ the homogeneous prime ideal associated to $Z_i$.
Assume by contradiction that $U$ is empty: that means that any homogeneous polynomial $F$ of degree $1$ is contained in at least one of the $P_i$. 
Namely $k[x_0,\dots, x_m]_1\subseteq \cup_i P_i$.
Hence 
\[
    k[x_0,\dots, x_m]_1 = \cup_i ( k[x_0,\dots, x_m]_1 \cap P_i ) = \cup_i \textup{Span}_k \langle  k[x_0,\dots, x_m]_1 \cap P_i \rangle .
\]
Since $k[x_0,\dots, x_m]_1$ is an infinite dimensional vector space and $k$ is infinite (because $k$ is algebraically closed), it cannot be a finite union of proper subspaces of itself.

Then there is a $P_i$ such that 
\[
    k[x_0,\dots, x_m]_1 = \textup{Span}_k \langle  k[x_0,\dots, x_m]_1 \cap P_i \rangle \subseteq P_i.
\]
Therefore, $k[x_0,\dots, x_m]_+\subseteq P_i$, which is not possible.
So we have a dense open $U\subseteq (\mathbb{P}^m_k)^*$: if $k$ is infinite, the open $U$ has infinitely many $k$-rational points.
\end{proof}

\begin{proposition}\label{prop:existsmoothcycleavoidQ}
    Let $X$ be a (smooth, projective and geometrically connected) surface and let $Z$ be a $0$-cycle: then there exists a geometrically reduced effective $0$-cycle $Q'$ of degree $d'\geq 2\deg(Z)$ disjoint from $Z$.
\end{proposition}

\begin{proof}    
Consider a very ample divisor $H_X$ and the closed embeddings $X \xhookrightarrow{} \mathbb{P}^{N_m}_k$ induced by its positive multiples $m H_X$.
Call $d_X$ the degree of the embedding defined by $H_X$.

If $k$ is infinite, let $U_Z$ be the open dense subset of hyperplanes avoiding $Z$ given by lemma \ref{LemmaavoidQ}.
By Bertini's theorem \cite[Thm. II.8.18]{Harty} there is an hyperplane $H\subset \mathbb{P}^{N}_k$ defined over $k$ such that $X\cap H$ is a smooth curve.
Let $V_{X\cap H}$ be the dense open set of hyperplanes in $\mathbb{P}^{N_m}_k$ whose intersection with $X\cap H$ is smooth (provided again by Bertini's).
Then the intersection $U_Z\cap V_{X\cap H}$ is still an open dense subset and (since $k$ is infinite) it contains infinitely many $k$-points.
Pick then $H'\in U_Z\cap V_{X\cap H}(k)$.
This gives a geometrically reduced effective $0$-cycle $Q'=X\cap H\cap H'$ of degree $m^2d_X$ disjoint from $Z$.\\

If $k$ is finite, by Poonen's version of Bertini's for smooth quasi-projective varieties over finite fields \cite[Thm. 1.1]{Poonen} we can pick two hypersurfaces $D_m$ and $D'_m$ of degree $d_m$ and $d_m'$ such that $(X\setminus Z)\cap D_m \cap D_m'$ is a smooth effective $0$-cycle $Q'$ of degree $(m^2\cdot d_m \cdot d_m')d_X$.
\end{proof}

Now, back to framework \ref{framework}, let us check first whether a triple $(\widetilde{X}, \widetilde{\mathcal{L}}_e, \widetilde{Q})$ is indeed suitable.

\begin{lemma}
    For $e>>0$ the triple $(\widetilde{X}, \widetilde{\mathcal{L}}_e, \widetilde{Q})$ is suitable.
\end{lemma}

\begin{proof}
The surface $\widetilde{X}$ is a suitable surface (it is geometrically integral since it is the blow-up of a geometrically integral scheme, see \cite[Tag 02ND]{deJ}) 
Indeed its $H^1(\widetilde{X},\mathcal{O}_{\widetilde{X}})$ is still trivial (see \cite[Prop. V.3.4]{Harty}).

Since $Q$ is disjoint from $Q'$, the scheme $\widetilde{Q}$ is isomorphic to $Q$ and in particular it is a geometrically reduced effective $0$-cycle of degree $d$ on $\widetilde{X}$.

Therefore, to conclude that $(\widetilde{X}, \widetilde{\mathcal{L}}_e, \widetilde{Q})$ is suitable it is enough to check that for $e>>0$ the line bundle $\widetilde{\mathcal{L}}_e$ is ample.\\

Let $e$ be big enough so that $\mathcal{I}_{Q'}\otimes \mathcal{L}^{\otimes e}$ is globally generated. 
Denote by $\mathcal{O}_{\widetilde{X}/X}(1)$\footnote{Usually it is denoted by $\mathcal{O}_{\widetilde{X}}(1)$. See \cite[Construction p. 160]{Harty}.} the invertible sheaf on the relative projective scheme $\widetilde{X}$ (over $X$).
Let $r_e+1$ be a set of global sections $\{ s_0,\dots,s_{r_e}\}$ which generate the sheaf.
They define a surjective map
\[
    \mathcal{O}_X^{r_e+1}
    \longrightarrow
    \mathcal{I}_{Q'}\otimes \mathcal{L}^{\otimes e}\longrightarrow 0.
\]
This in turns induces a surjective morphism
\[
    \mathcal{O}_X[s_0,\dots,s_{r_e}] 
    \longrightarrow 
    \Sym^* (  \mathcal{I}_{Q'}\otimes \mathcal{L}^{\otimes e} )
    \longrightarrow 0
\]
which means a closed immersion $\widetilde{X}\xhookrightarrow{} \mathbb{P}^{r_e}_X$.
Let $e_{va}$ be an integer such that $\mathcal{L}^{\otimes e_{va}}$ is very ample.
Then we have the sequence of closed embeddings
\begin{equation*}
    \widetilde{X}\xhookrightarrow{} \mathbb{P}^{r_e}_X = \mathbb{P}^{r_e}_k \times X \xhookrightarrow{ \id \times |\mathcal{L}^{\otimes e_{va}}|} \mathbb{P}^{r_e}_k \times \mathbb{P}( H^0(X, \mathcal{L}^{\otimes e_{va}})^{\vee} ) \xhookrightarrow{} \mathbb{P}^N_k
\end{equation*}
(the last one is the Segre embedding).
The (very) ample line bundle associated to the composite embedding is
\[
    \mathcal{O}_{\widetilde{X}/X}(1)\otimes b^*\mathcal{L}^{\otimes e} \otimes b^*\mathcal{L}^{\otimes e_{va}} \cong \mathcal{O}_{\widetilde{X}}(-E)\otimes b^*\mathcal{L}^{\otimes e + e_{va}}.
\]
Therefore for $e>>0$ the line bundle  $\widetilde{\mathcal{L}}_e\cong b^*\mathcal{L}^{\otimes e} \otimes \mathcal{O}_{\widetilde{X}}(-2E)$ is (very) ample.
\end{proof}

\begin{theorem}\label{theoremgap2}
Within the notation of framework \ref{framework}, for $e>>0$ the suitable triple $(\widetilde{X}, \widetilde{\mathcal{L}}_e,\widetilde{Q})$ satisfies the \hyperlink{assumptions}{assumptions} \ref{thefiveassumptions} with respect to $n$ if and only if $n$ belongs to the following interval
\[
    \widetilde{I}_e\coloneqq \left[ \frac{e^2 c_1(\mathcal{L})^2 +  e\, c_1(\mathcal{L})\cdot K_X}{2} +1 - d', \frac{ e^2 c_1(\mathcal{L})^2-e\, c_1(\mathcal{L}) \cdot K_X }{2} + \dim_k H^2(X, \mathcal{O}_X) - 2d - 3d' \right],
\]
in which case $\Hilb^n_X$ is stably birational to $\Hilb^{n+\deg(Q)}_X$.
\end{theorem}

The proof of theorem \ref{theoremgap2} will be conceptually the same as the proof of theorem \ref{theoremgap1}.
But before checking the \hyperlink{assumptions}{assumptions} one by one, we compute some higher direct images of sheaves and cohomology groups, which will be needed in proof of \ref{theoremgap2}. 

The following are standard computations.

\begin{lemma}\label{lemmacohomologycomputations}
    In the blow-up framework \ref{framework} we have:
    \begin{enumerate}
        \item $R^ib_*(i_{E\, *}\mathcal{O}_E) 
        = \left\{\begin{matrix}
        i'_*\mathcal{O}_{Q'} & \textup{if } i= 0 \\
        0 & \textup{if } i\geq 1.
        \end{matrix}\right.$ ;
        \item $R^ib_*(i_{E\, *}\mathcal{O}_E(1) ) 
        = \left\{\begin{matrix}
        \mathcal{I}_{Q'}/\mathcal{I}^2_{Q'} & \textup{if } i= 0 \\
        0 & \textup{if } i\geq 1
        \end{matrix}\right.$
        \item $R^ib_*\mathcal{O}_{\widetilde{X}}(-E) = \left\{\begin{matrix}
        \mathcal{I}_{Q'} & \textup{if } i= 0 \\
        0 & \textup{if } i\geq 1.
        \end{matrix}\right.$ ;
        \item $R^ib_*\mathcal{O}_{\widetilde{X}}(-2E) = \left\{\begin{matrix}
        \mathcal{I}^2_{Q'} & \textup{if } i= 0 \\
        0 & \textup{if } i\geq 1.
        \end{matrix}\right.$
        \item $R^1b_* (\mathcal{O}_{\widetilde{X}}(-2E)\otimes \mathcal{I}_{\widetilde{Q}})=0$;
        \item $ \dim_k H^0(Q'(2), \mathcal{O}_{Q'(2)}) = 3 \deg(Q')$.
    \end{enumerate}
\end{lemma}

    \begin{proof}
        For any $i\ge 0$ 
        we have
        \[
            i'_* R^i(b_E)_*\mathcal{O}_E
            \cong R^i(b\circ i_E)_*\mathcal{O}_E
            \cong R^ib_* (i_{E\, *}\mathcal{O}_E).
        \]
        This, combined with the following isomorphisms
        \[
            R^i(b_E)_*\mathcal{O}_E
            \cong 
            H^i(E,\mathcal{O}_E)^{\sim} 
            \cong
            H^i(\mathbb{P}^1_{Q'},\mathcal{O}_{\mathbb{P}^1_{Q'}})^{\sim} ,
        \]
        implies that
        \[
            R^ib_*(i_{E\, *}\mathcal{O}_E) 
            = \left\{\begin{matrix}
            i'_*\mathcal{O}_{Q'} & \textup{if } i= 0 \\
            0 & \textup{if } i\geq 1.
            \end{matrix}\right.
        \]
        Similarly for any $i\ge 0$ we have
        \[
            i'_* R^i(b_E)_*\mathcal{O}_E(1)
            \cong R^i(b\circ i_E)_*\mathcal{O}_E(1)
            \cong R^ib_* (i_{E\, *}\mathcal{O}_E(1)),
        \]
        which, combined with
        \[
            R^i(b_E)_*\mathcal{O}_E(1)) 
            \cong 
            H^i(E,\mathcal{O}_{E}(1))^{\sim}
            \cong
            H^i(\mathbb{P}^1_{Q'},\mathcal{O}_{\mathbb{P}^1_{Q'}}(1))^{\sim},
        \]
        gives 
        \[ 
            R^ib_*(i_{E\, *}\mathcal{O}_E(1) ) 
            = \left\{\begin{matrix}
            \mathcal{I}_{Q'}/\mathcal{I}^2_{Q'} & \textup{if } i= 0 \\
            0 & \textup{if } i\geq 1
            \end{matrix}\right.
        \]
        (the result for $i=0$ follows from the direct computation of the stalks at each point of $X$). \\

        Consider the push-forward of the exact sequence defining $E$:
        thanks to what we have just proven about $R^ib_*\mathcal{O}_E$, it looks as follows:
        \[
            0\longrightarrow b_*\mathcal{O}_{\widetilde{X}}(-E)\longrightarrow \mathcal{O}_{X} \longrightarrow i'_*\mathcal{O}_{Q'}
            \longrightarrow R^1b_*\mathcal{O}_{\widetilde{X}}(-E)\longrightarrow 0 \longrightarrow 0 \longrightarrow R^2b_*\mathcal{O}_{\widetilde{X}}(-E) \longrightarrow 0
        \]
        (the vanishing of $R^2b_*\mathcal{O}_{\widetilde{X}}=0$ is proved in \cite[Thm. 1.1]{ChR}).
        Since $\mathcal{O}_X\longrightarrow i'_*\mathcal{O}_{Q'}$ is surjective we have
        \[
            R^ib_*\mathcal{O}_{\widetilde{X}}(-E) = \left\{\begin{matrix}
            \mathcal{I}_{Q'} & \textup{if } i= 0 \\
            0 & \textup{if } i\geq 1.
            \end{matrix}\right.
        \]
        Finally, consider the exact sequence
        \[
            0\longrightarrow 
            \mathcal{O}_{\widetilde{X}}(-2E)
            \longrightarrow 
            \mathcal{O}_{\widetilde{X}}(-E)
            \longrightarrow 
            \mathcal{O}_{\widetilde{X}}(-E)|_E
            \longrightarrow 0.
        \]
        Consider the connected (and hence integral) components $Q_j' = \Spec (k_j')$ of $Q'$: it follows that $E=\sqcup E_j \cong \sqcup \mathbb{P}^1_{k_j'}$.
        Since $E_j^2 = -\deg (k_j')$ and $\Pic(E)\cong \deg (k_j') \mathbb{Z}$ we have that $\mathcal{O}_{\widetilde{X}}(-E)|_E\cong \mathcal{O}_E(1)$.
        Therefore we get that 
        \[
            0\longrightarrow b_*\mathcal{O}_{\widetilde{X}}(-2E)\longrightarrow b_*\mathcal{O}_{\widetilde{X}}(-E) \longrightarrow b_*i_{E\, *}\mathcal{O}_{E}(1)
            \longrightarrow R^1b_*\mathcal{O}_{\widetilde{X}}(-2E)\longrightarrow 0
        \]
        and that $R^2b_*\mathcal{O}_{\widetilde{X}}(-2E)$ is trivial.
        Therefore 
        \[
            R^ib_*\mathcal{O}_{\widetilde{X}}(-2E) = \left\{\begin{matrix}
            \mathcal{I}^2_{Q'} & \textup{if } i= 0 \\
            0 & \textup{if } i\geq 1.
            \end{matrix}\right.
        \]
        Let $E(2)$ be the thickening of $E$ defined by $\mathcal{O}_{\widetilde{X}}(-2E)$ and let $Q'(2)$ be the corresponding thickening of $Q'$ defined by the ideal $\mathcal{I}^2_{Q'}$.
            
        From the discussion above we have that the sheaf $b_*\mathcal{O}_{E(2)}$ coincides $\mathcal{O}_{Q'(2)}$.
        Therefore, since $Q$ is disjoint from $Q'$ we get that $b_*\mathcal{O}_{E(2) \sqcup \widetilde{Q}} \cong \mathcal{O}_{Q'(2) \sqcup Q}$ as well. 
    
        In particular we get the following exact sequences
        \begin{align*}
            0\longrightarrow
            b_*(\mathcal{O}_{\widetilde{X}}(-2E) \otimes \mathcal{I}_{\widetilde{Q}})
            \longrightarrow
            \mathcal{O}_{X} \longrightarrow
            \mathcal{O}_{Q'(2)\sqcup Q}
            \longrightarrow
            R^1b_*(\mathcal{O}_{\widetilde{X}}(-2E) \otimes \mathcal{I}_{\widetilde{Q}})
            \longrightarrow0.
        \end{align*}
        Since the map $\mathcal{O}_{X} \longrightarrow
        \mathcal{O}_{Q'(2)\sqcup Q}$ is surjective, the sheaf $R^1b_* ( \mathcal{O}_{\widetilde{X}}(-2E)\otimes \mathcal{I}_{\widetilde{Q}} )=0$.\\

        Let us now compute the dimension of $H^0(Q'(2), \mathcal{O}_{Q_2'(2)})$: in order to do so we can pass to the algebraic closure $\overline{k}$. Over $\overline{k}$ the scheme $Q'(2)$ is the disjoint union of $\deg(Q')$-copies of $\Spec(\overline{k)}$: call $m_i$ the maximal ideal corresponding to the $i$-th connected component. 
        Then the thickening is given by $m_i^2$.
        So for each of them, the dimension of the global sections is
        \[
            \dim_{\overline{k}} (A_i/m_i^2) = \dim_{\overline{k}} A_i/m_i \oplus m_i/m_i^2 =3,
        \]
        where $\Spec(A_i)$ is an open of $X_{\overline{k}}$ around the $i$-th connected component of $Q'$.
        Therefore
        \[
            \dim_k H^0(Q'(2), \mathcal{O}_{Q'(2)}) = 3 \deg(Q').
        \]
    \end{proof}

    \begin{proof}[Proof of Theorem \ref{theoremgap2}]
        As in the proof of theorem \ref{theoremgap1}, let us check the \hyperlink{assumptions}{assumptions} one at a time.\\

        \textbf{Assumption 1.} It is enough to prove that $H^1(\widetilde{X},
        \widetilde{\mathcal{L}}_e \otimes \mathcal{I}_{\widetilde{Q}} )=0$ for $e>>0$.
        The Leray Spectral Sequence for the sheaf $\widetilde{\mathcal{L}}_e \otimes \mathcal{I}_{\widetilde{Q}}$ with respect to the map $b$ tells us 
            \[
                H^1(X,b_*(\widetilde{\mathcal{L}}_e \otimes \mathcal{I}_{\widetilde{Q}}))
                \longrightarrow 
                H^1(\widetilde{X},\widetilde{\mathcal{L}}_e \otimes \mathcal{I}_{\widetilde{Q}})
                \longrightarrow
                H^0(X,R^1b_*(\widetilde{\mathcal{L}}_e \otimes \mathcal{I}_{\widetilde{Q}})).
            \]
        Using the projection formula
        \begin{align*}
            R^ib_*( \widetilde{\mathcal{L}}_e \otimes \mathcal{I}_{\widetilde{Q}} ) \cong 
            \mathcal{L}^{\otimes e}  \otimes R^ib_* (\mathcal{O}_{\widetilde{X}}(-2E)\otimes \mathcal{I}_{\widetilde{Q}}),
        \end{align*}
        and Serre vanishing theorem, we get that the cohomology group $H^1(X,b_*(\widetilde{\mathcal{L}}_e \otimes \mathcal{I}_{\widetilde{Q}}))$ vanish for $e>>0$.
        Since the sheaf $R^1b_* (\mathcal{O}_{\widetilde{X}}(-2E)\otimes \mathcal{I}_{\widetilde{Q}})$ is already trivial (by lemma \ref{lemmacohomologycomputations}), the cohomology group $H^0(X,R^1b_*(\widetilde{\mathcal{L}}_e \otimes \mathcal{I}_{\widetilde{Q}}))$ is trivial as well.
        Hence we get the vanishing of $H^1(\widetilde{X},
        \widetilde{\mathcal{L}}_e \otimes \mathcal{I}_{\widetilde{Q}} )=0$ for $e>>0$.\\

         \textbf{Assumption 2.} By lemma \ref{lemmacohomologycomputations} the sheaves $R^ib_*(\mathcal{O}_{\widetilde{X}}(-2E) )$ are trivial for $i\geq 1$.
         Hence, the cohomology groups $H^j(X,\mathcal{L}^{\otimes e}\otimes R^ib_* ( \mathcal{O}_{\widetilde{X}}(-2E) ))$ are all zeroes (for $i\geq 1$).
         Therefore, by Leray Spectral sequence
        \[
            H^i(X,b_*\widetilde{\mathcal{L}}_e)\cong H^i(\widetilde{X}, \widetilde{\mathcal{L}}_e),\quad \forall i\geq 0.
        \]
        By projection formula, they are all trivial for $i\geq 1$ and $e>>0$.
        So for $e>>0$ we get $\dim_k H^0(\widetilde{X}, \widetilde{\mathcal{L}}_e) = \chi(\widetilde{X}, \widetilde{\mathcal{L}}_e)$.
        By Riemann-Roch we then have
        \[
            \dim_k H^0(\widetilde{X}, \widetilde{\mathcal{L}}_e) = \frac{c_1(\widetilde{\mathcal{L}}_e)^2-c_1(\widetilde{\mathcal{L}}_e)\cdot K_{\widetilde{X}}}{2} +\chi(\widetilde{X}, \mathcal{O}_{\widetilde{X}}).
        \]
        By \cite[Prop. V.3.4]{Harty} we have $\chi(\widetilde{X}, \mathcal{O}_{\widetilde{X}} = 1+ \dim_k H^2(X,\mathcal{O}_X)$.
        Using the structure of the intersection theory on the blow-up (see \cite[Prop. V.3.2]{Harty}), we get
        \[
            \dim_k H^0(\widetilde{X}, \widetilde{\mathcal{L}}_e) = \frac{e^2c_1(\mathcal{L})^2 -e\, c_1(\mathcal{L})\cdot K_X)}{2} - 3d' + 1 + \dim_k H^2(X,\mathcal{O}_X).
        \]
        So for $e>>0$ assumption \hyperlink{assumptions}{(2)} is true with respect to $n$ if and only if 
        \[
              n \leq \frac{ e^2 c_1(\mathcal{L})^2-e\, c_1(\mathcal{L}) \cdot K_X }{2} + \dim_k H^2(X, \mathcal{O}_X) - 2d - 3d'.
        \] 
        
        \textbf{Assumption 3.} In order to show that the sheaf $\widetilde{\mathcal{L}}_e\otimes \mathcal{I}_{\widetilde{Q}}$ is globally generated for $e>>0$, we write it as
        \[
            \widetilde{\mathcal{L}}_e\otimes \mathcal{I}_{\widetilde{Q}} \cong b^*\mathcal{L}^{\otimes e_2} \otimes \mathcal{I}_{\widetilde{Q}} \otimes b^*\mathcal{L}^{\otimes e_1}\otimes \mathcal{O}_{\widetilde{X}}(-2E),
        \]
        where $e=e_1+e_2$ and $e_1$ is an integer such that $b^*\mathcal{L}^{\otimes e_1}\otimes \mathcal{O}_{\widetilde{X}}(-2E)$ is very ample.
        In this way it is sufficient to check that $b^*\mathcal{L}^{\otimes e_2} \otimes \mathcal{I}_{\widetilde{Q}}$ is globally generated.
        
        The counit map $b^*b_*( b^*\mathcal{L}^{\otimes e_2} \otimes \mathcal{I}_{\widetilde{Q}}  ) \longrightarrow b^*\mathcal{L}^{\otimes e_2} \otimes \mathcal{I}_{\widetilde{Q}}$ gives the following commutative diagram
        \begin{equation}\label{diagramconitmap}
            \begin{tikzcd}
                H^0(\widetilde{X}, b^*\mathcal{L}^{\otimes e_2} \otimes \mathcal{I}_{\widetilde{Q}}) \otimes b^*b_*\mathcal{O}_{\widetilde{X}}
                \arrow[r]
                \arrow[d]
                &
                H^0(\widetilde{X}, b^*\mathcal{L}^{\otimes e_2} \otimes \mathcal{I}_{\widetilde{Q}}) \otimes \mathcal{O}_{\widetilde{X}}
                \arrow[d]
                \\
                b^*b_*( b^*\mathcal{L}^{\otimes e_2} \otimes \mathcal{I}_{\widetilde{Q}}  )
                \arrow[r]
                &
                b^*\mathcal{L}^{\otimes e_2} \otimes \mathcal{I}_{\widetilde{Q}}.
            \end{tikzcd}
        \end{equation}
        Recall that, since $Q$ is disjoint from $Q'$, then
        \[
            b_*\mathcal{I}_{\widetilde{Q}}\cong \mathcal{I}_Q \quad \textup{ and } \quad b^*\mathcal{I}_Q \cong \mathcal{I}_{\widetilde{Q}}.
        \]
        Therefore, by projection formula we get $b^*b_*( b^*\mathcal{L}^{\otimes e_2} \otimes \mathcal{I}_{\widetilde{Q}}  ) \cong b^*\mathcal{L}^{\otimes e_2} \otimes \mathcal{I}_{\widetilde{Q}}$.
        So the bottom arrow in \ref{diagramconitmap} is an isomorphism.

        Now, since $b_*\mathcal{O}_{\widetilde{X}}\cong \mathcal{O}_X$ (see \cite[Prop. V.3.4]{Harty}), the left vertical arrow in \ref{diagramconitmap} coincides with
        \[
            b^*\left(  H^0(X, \mathcal{L}^{e_2}\otimes \mathcal{I}_Q) \otimes \mathcal{O}_X \longrightarrow  \mathcal{L}^{\otimes e_2}\otimes \mathcal{I}_{Q} \right).
        \]
        Since $\mathcal{L}$ is ample, the sheaf $\mathcal{L}^{\otimes e-e_1} \otimes \mathcal{I}_Q$ is globally generated for $e>>0$.
        Therefore the left vertical arrow in \ref{diagramconitmap} is surjective (for $e>>0$), which implies that $b^*\mathcal{L}^{\otimes e-e_1} \otimes \mathcal{I}_{\widetilde{Q}}$ is generated by global sections as well for $e>>0$.\\

        \textbf{Assumption 4.} 
        As done in the proof of theorem \ref{theoremgap1} (when we were checking assumption \hyperlink{assumption}{(4)}), it is enough to pick $e\geq e_{g}+e_{va}$ where $b^*\mathcal{L}^{\otimes e_g} \otimes \mathcal{O}_{\widetilde{X}}(-E)\otimes \mathcal{I}_{\widetilde{Q}}$ is globally generated and $e_{va}$ is such that $b^*\mathcal{L}^{\otimes e_{va}} \otimes \mathcal{O}_{\widetilde{X}}(-E)$ is very ample.\\

        \textbf{Assumption 5.} Using again the intersection theory on $\widetilde{X}$ we can compute 
        \[ 
        \frac{c_1(\widetilde{\mathcal{L}}_e)^2 + c_1(\widetilde{\mathcal{L}}_e)\cdot K_{\widetilde{X}} }{2} +1.
        \]
        It is equal to
        \begin{align*}
            \frac{e^2c_1(\mathcal{L})^2 + e\, c_1(\mathcal{L}_e)\cdot K_{X} }{2} + 1 -d'.
        \end{align*}
        Hence the statement of the theorem.
    \end{proof}

\subsection{Necessary and sufficient conditions for Theorems \ref{theoremgap1} and \ref{theoremgap2}}\label{geometricallyrat}\, \\

Theorems \ref{theoremgap1} and \ref{theoremgap2} define intervals $I_e$ and $\widetilde{I_e}$ for $e>>0$.
However, as mentioned previously, these theorems have a meaning only when we check that those intervals are non-empty.
In this subsection, we study when $I_e$ and $\widetilde{I}_e$ are non-empty for infinitely many values of $e$.

\begin{proposition}\label{propositionc1(L)KXnegative}
    Let $(X,\mathcal{L}, Q)$ be a suitable triple. 
    The interval $I_e$ is non-empty for infinitely many $e$ if and only if $c_1(\mathcal{L})\cdot K_X<0$.
    The same holds true for $\widetilde{I}_e$.

    If this happens, the surface $X$ has Kodaria dimension $\kappa(X)=-\infty$. In particular $H^2(X,\mathcal{O}_X)=0$.
    Since we are assuming $H^1(X,\mathcal{O}_X)=0$, the surface $X$ is geometrically rational.
\end{proposition}

\begin{proof}
    The interval defined in theorem
    \ref{theoremgap1} is non-empty if and only if the difference
    \begin{align*}
        \frac{ e^2 c_1(\mathcal{L})^2-e\, c_1(\mathcal{L}) \cdot K_X }{2} + \dim_k H^2(X, \mathcal{O}_X) -2\deg(Q) - \frac{e^2 c_1(\mathcal{L})^2 +  e \, c_1(\mathcal{L}) \cdot K_X}{2}-1 =
        \\
        = -e \, c_1(\mathcal{L}) \cdot K_X  + \dim_k H^2(X, \mathcal{O}_X) -2\deg(Q) -1
    \end{align*}
    is non-negative.
    If $c_1(\mathcal{L})\cdot K_X > 0$, then for $e>>0$ the difference is negative.
    If $c_1(\mathcal{L})\cdot K_X \leq 0$, then by Nakai-Moishezon Criterion of ampleness, the divisor $K_X$ has no global section (otherwise $K_X$ would be equivalent to an effective divisor $E$ and $c_1(\mathcal{L})\cdot K_X >0$). 
    Therefore by Serre duality $\dim_k H^2(X,\mathcal{O}_X) = \dim_k H^0(X, K_X)=0$ and the difference is $-e\, c_1(\mathcal{L})\cdot K_X -2\deg(Q)-1$, which is non-negative for $e>>0$ if and only if $c_1(\mathcal{L})\cdot K_X<0$.

    By the same argument we get that $H^0(X,mK_X)=0$ for all $m\geq 1$. 
    Therefore $\kappa(X)=\kappa(X_{\overline{k}})=-\infty$.
    Since $H^1(X_{\overline{k}}, \mathcal{O}_{X_{\overline{k}}})$ and $H^0(X_{\overline{k}},
    2K_{X_{\overline{k}}})$ 
    are both trivial, by Castelnuovo's rationality criterion $X_{\overline{k}}$ is rational (see \cite[Thm.\, 1.53]{Pie}). 
\end{proof}

\begin{remark}
For instance, if $X$ is a K3 surface, then $K_X$ is trivial and therefore $c_1(\mathcal{L})\cdot K_X=0$.
So in this case, we cannot apply theorem \ref{theoremgap1} for large $e$.
This was already predicted by Litt's theorem \cite[Thm. 19]{Li}.
\end{remark}

Given proposition \ref{propositionc1(L)KXnegative}, let us restrict our attention to those surfaces with an ample line bundle $\mathcal{L}$ for which the intersection $c_1(\mathcal{L})\cdot K_X$ is negative.

The intervals $I_e = [\min_e, \max_e]$ leave gaps $\textup{gap}_e\coloneqq [\max_e +1 ,  \min_{e+1}-1]$: they have width 
\[
     {\min}_{e+1} -1 - {\max}_e =\frac{(e+1)^2 c_1(\mathcal{L})^2 +  (e+1) c_1(\mathcal{L}) \cdot K_X}{2}  - 
     \frac{ e^2 c_1(\mathcal{L})^2-e\, c_1(\mathcal{L}) \cdot K_X }{2} + 2\deg(Q)
\]
\begin{equation}\label{widthofintervalIe}
    {\min}_{e+1} -1 - {\max}_e =\, (2e+1) \frac{c_1(\mathcal{L})^2 + c_1(\mathcal{L}) \cdot K_X}{2} + 2\deg(Q).
\end{equation}
Therefore the growth of the gaps depend on the sign of $c_1(\mathcal{L})^2 + c_1(\mathcal{L}) \cdot K_X$.

\begin{remark}\label{remarkonthegrowth}
Since $\mathcal{L}$ is ample, by Nakai-Moishezon Criterion, $c_1(\mathcal{L})^2>0$: so  if $c_1(\mathcal{L})^2+c_1(\mathcal{L})\cdot K_X \leq 0$ then the intersection $c_1(\mathcal{L})\cdot K_X$ is negative.
So the intervals $I_e$ and $\widetilde{I}_e$ are growing in width.

Moreover, by equation \ref{widthofintervalIe}, if $c_1(\mathcal{L})^2 + c_1(\mathcal{L}) \cdot K_X<0$ then for $e>>0$ we have that $\min_{e+1}-1\leq \max_e+1$.
Namely, the intervals $I_e$'s contain all the integers $n>>0$.
If $c_1(\mathcal{L})^2 + c_1(\mathcal{L}) \cdot K_X=0$ then the gaps have fixed width equal to $2d$.
If $c_1(\mathcal{L})^2 + c_1(\mathcal{L}) \cdot K_X>0$ then the gaps are growing in width.
\end{remark}

Therefore, for $e>>0$ the gaps happen only in the case $c_1(\mathcal{L})^2 + c_1(\mathcal{L}) \cdot K_X \geq 0$: let us see when the intervals $\widetilde{I}_e$ can help fill these gaps.

\begin{lemma}\label{lemmacovergaps}
    Let $(X,\mathcal{L},Q)$ be a suitable triple where $c_1(\mathcal{L})\cdot K_X<0$  and 
    $c_1(\mathcal{L})^2 + c_1(\mathcal{L}) \cdot K_X \geq 0$.
    Let $Q'$ be as in \ref{framework}.
    Then the following are equivalent:
    \begin{enumerate}
        \item for any $e>>0$ there exists an $f=f(e)$ such that $\textup{gap}_e \subseteq \widetilde{I}_{f}$;
        \item we have $c_1(\mathcal{L})^2 +  c_1(\mathcal{L})\cdot K_X=0$ and $d'\geq 2d$.
    \end{enumerate}
\end{lemma}

\begin{proof}
    The interval $\textup{gap}_e$ is contained in $\widetilde{I}_f = [\widetilde{\min}_f, \widetilde{\max}_f]$ if and only if     
    \[
        \widetilde{\min}_f\leq \max_e +1 \textup{ and } \min_{e+1}-1\leq \widetilde{\max}_f
    \]
    or equivalently
    \begin{align*}
        \frac{f^2c_1(\mathcal{L})^2 + f\, c_1(\mathcal{L})\cdot K_X}{2} + 1 -d' \leq & \, \frac{e^2c_1(\mathcal{L})^2 - e\, c_1(\mathcal{L})\cdot K_X}{2} +1-2d,
        \\
        \frac{(e+1)^2c_1(\mathcal{L})^2 + (e+1) c_1(\mathcal{L})\cdot K_X}{2} \leq & \, \frac{f^2c_1(\mathcal{L})^2 - f\, c_1(\mathcal{L})\cdot K_X}{2} -2d -3d' 
    \end{align*}
    Using $c_1(\mathcal{L})\cdot K_X \geq -c_1(\mathcal{L})^2$: then the above inequalities imply (for $e>>0$)
    \begin{align*}
        \frac{(f^2-f)}{2} c_1(\mathcal{L})^2 \leq & \, \frac{(e^2+e)}{2} c_1(\mathcal{L})^2 +d'-2d \implies f\leq e+1
        \\
        \frac{(e^2+e)}{2}c_1(\mathcal{L})^2 \leq & \, \frac{(f^2+f)}{2} c_1(\mathcal{L})^2 -2d -3d' \implies f\geq e+1.
    \end{align*}
    Therefore for $e>>0$, $\textup{gap}_e \subseteq \widetilde{I}_f$ forces $f$ to be equal to $e+1$.
    Therefore, writing the initial inequalities for $\textup{gap}_e\subseteq \widetilde{I}_{e+1}$, we get 
    \begin{align*}
        (2e+1)\frac{c_1(\mathcal{L})^2 +  c_1(\mathcal{L})\cdot K_X}{2} \leq d' -2d
        \ \textup{ and }\
        (e+1) (-c_1(\mathcal{L})\cdot K_X) \geq  2d + 3d',
    \end{align*}
    which implies $c_1(\mathcal{L})^2 +  c_1(\mathcal{L})\cdot K_X=0$ and $d'\geq 2d$.\\
    
    The same computations show that if $c_1(\mathcal{L})^2 +  c_1(\mathcal{L})\cdot K_X=0$ and $d'\geq 2d$ then for $e>>0$ $\textup{gap}_e\subseteq \widetilde{I}_{e+1}$.
\end{proof}

This comparison between $\mathcal{I}_e$ and $\widetilde{I}_e$ leads to the following corollary of theorems \ref{theoremgap1} and \ref{theoremgap2}.
    
\begin{corollary}\label{corollarydefinitestable}
Let $(X,\mathcal{L},Q)$ be a suitable triple where $\mathcal{L}$ is such that 
\begin{equation}\label{conditiononL}
c_1(\mathcal{L})^2+c_1(\mathcal{L})\cdot K_X \leq 0.
\end{equation}
Then $X$ is geometrically rational.
Furthermore, for $n>>0$ the Hilbert schemes $\Hilb^n_X$ is stably birational to $\Hilb^{n+d}_X$.
\end{corollary}

\begin{proof}
By remark \ref{remarkonthegrowth}, if $c_1(\mathcal{L})^2+c_1(\mathcal{L})\cdot K_X \leq 0$ then $c_1(\mathcal{L})\cdot K_X$ is negative.
By proposition \ref{propositionc1(L)KXnegative} $X$ is geometrically rational.
Meanwhile, theorem \ref{theoremgap1} tells us that $\Hilb^n_X$ is stably birational to $\Hilb^{n+d}_X$ for any $n\in I_e$.
If some gaps $\textup{gap}_e$ are left, then pick a smooth effective $0$-cycle $Q'$ disjoint from $Q$ and of degree $d'\geq 2\deg(Q)$ (we can do it by proposition \ref{prop:existsmoothcycleavoidQ}). 
Then for $e>>0$, those eventual gaps are covered by the intervals $\widetilde{I}_{e+1}$ as seen in lemma \ref{lemmacovergaps}. 
We conclude using theorem \ref{theoremgap2}.       
\end{proof}

\section{Geometrically rational surfaces}\label{8}

As seen in the previous section, theorems \ref{theoremgap1} and \ref{theoremgap2} work best for geometrically rational surfaces.
In particular, corollary \ref{corollarydefinitestable} tells us sufficient conditions on $X$ in order to have $\Hilb^n_X$ stably birational to $\Hilb^{n+d}_X$.
In this section we want to complete the study of the stable birational classes of $\Hilb^n_X$ for all geometrically rational surfaces.

Since the stable birational class of $\Hilb^n_X$ depends only on the stable birational class of $X$ (see fact \ref{factonstablebira}), we can assume without loss of generality that $X$ is minimal (relatively over $k$).

\begin{definition}
    Recall that $X$ is said to be \emph{minimal over} $k$ if any birational morphism $X\longrightarrow X'$ to a smooth, geometrically connected, projective surface $X'$ over $k$ is an isomorphism. 
\end{definition}

Iskovskih in \cite{Isk} classified minimal geometrically rational surfaces over arbitrary fields.

\begin{theorem}[Thm. 1 of \cite{Isk}]\label{classificationofgeomratsurfaces}
    Let $X$ be a geometrically rational, suitable surface over an arbitrary field $k$.

    If $X$ is minimal over $k$ then it is either: 
    \begin{enumerate}
        \item isomorphic to the projective plane $\mathbb{P}^2_k$;
        \item isomorphic to a quadric $Q\subset \mathbb{P}^3_k$ such that  $\Pic(Q)\cong \mathbb{Z}$;
        \item isomorphic to a del Pezzo surface; or
        \item the Picard group is isomorphic to $\Z \oplus \Z$ and there exists a smooth, geometrically connected, projective, genus zero curve $C$ and a surjective map $f:X\longrightarrow C$ whose generic fiber $F_{\eta(C)}$ is a smooth curve of genus zero.
    \end{enumerate}
\end{theorem}

\begin{definition}
    If $X$ is of the fourth type, we say that $X$ is a \emph{rational conic bundle}.
\end{definition}

Notice that theorem \ref{classificationofgeomratsurfaces} is not a classification up to birational maps: for instance, some del Pezzo surfaces are already rational over $k$, a smooth quadric $Q$ in $\mathbb{P}^3_k$ is a del Pezzo (by adjunction formula the canonical $\omega_Q$ is isomorphic to $\mathcal{O}_Q(-2)$ and hence anti-ample) and some rational conic bundles are del Pezzo (see \cite[Thm. 5]{Isk}). 

Therefore, since we are interested in $X$ up to birational equivalence, we have two families to consider.

\begin{corollary}
    Let $X$ be a geometrically rational surface minimal over $k$.
    Then either:
    \begin{enumerate}
        \item[] \textup{Type I.} $X$ is birational to a del Pezzo surface, or
        \item[] \textup{Type II.} $X$ is a rational conic bundle as in theorem \ref{classificationofgeomratsurfaces}.
    \end{enumerate}
\end{corollary}

\subsection{Del Pezzo Surfaces} \, \\

Let us first focus on del Pezzo surfaces $X$ of degree $K_X^2=d_X$. 
The aim of this subsection is to prove the following.

\begin{theorem}\label{theoremdelPezzo}
    Let $X$ be a del Pezzo surface and let $\textup{ind}(X)$ be the index of $X$:  
    then for $n>>0$ the Hilbert scheme $\Hilb^n_X$ is stably birational to $\Hilb^{n+\textup{ind}(X)}_X$.
\end{theorem}

\begin{proof}
As line bundle $\mathcal{L}$, pick the anticanonical $\omega_X^{\vee}$: it is ample ($X$ is del Pezzo) and satisfies condition \ref{conditiononL} of corollary \ref{corollarydefinitestable}
\[
    c_1(\mathcal{L})^2+c_1(\mathcal{L})\cdot K_X = K_X^2-K_X^2=0.
\]
Now recall that for a smooth variety $X$ the index is equal to the greatest common divisor of closed points $x\in X$ such that $k(x)$ is a separable extension of $k$ (see \cite[Thm. 9.2]{Gabber}).
If $k(x)/k$ separable, then $\Spec(k(x))$ is geometrically reduced (see \cite[tag. 030W]{deJ}: then there exist geometrically reduced closed points $Q_1, \dots, Q_r$ on $X$ such that $\gcd \{ \deg(Q_i) \}_{i=1}^r =\textup{ind}(X)$.
Setting $\deg(Q_i)=d_i$ we can assume
\[
    \textup{ind}(X) = \sum_{i=1}^s a_i d_i - \sum_{i=s+1}^r a_i d_i
\]
where the $a_i$'s are positive integers.

Then $(X,\mathcal{L},Q_i)$ satisfies the hypotheses of corollary \ref{corollarydefinitestable}.
Therefore, for $n\geq N_i$ the variety $\Hilb^n_X$ is stably birational to $\Hilb^{n+d_i}_X$.

This means that for any $n\geq \max_i N_i$ we have that $\Hilb^n_X$ is stably birational to $\Hilb^{n + \sum_{i=1}^s a_id_i}_X$ and that $\Hilb^{n+\textup{ind}(X)}_X$ is stably birational to $\Hilb^{n +\textup{ind}(X) + \sum_{i=s+1}^r a_id_i}_X = \Hilb^{n + \sum_{i=1}^s a_id_i}_X$.

Therefore, for $n>>0$ we always have $\Hilb^n_X$ stably birational to $\Hilb^{n+\textup{ind}(X)}_X$.
\end{proof}

\subsection{Minimal rational conic bundles}\, \\

We focus now on Type II surfaces: namely, minimal rational conic bundles.
In \cite{Isk}, Iskovskih gives a detailed description of rational conic bundles over arbitrary fields.

\begin{theorem}[Thm. 3 of \cite{Isk}]\label{descriptionofconicbundles}
Let $f:X\longrightarrow C$ be a rational conic bundle as in theorem \ref{classificationofgeomratsurfaces}. 
The following assertions hold:
\begin{enumerate}
    \item The sheaf $f_*\omega_X^{-1}$ is locally free of rank $3$ and induces a closed embedding
    \[
        \begin{tikzcd}
            X
            \arrow[rd, "f"']
            \arrow[r, hook]
            &
            \mathbb{P}_C(f_*\omega_X^{-1})
            \arrow[d]
            \\
            &
            C
        \end{tikzcd}
    \]
    mapping each fiber $F_t$, $t\in C$, into a (not necessarily smooth) conic of $\mathbb{P}^2_{k(t)}$.
    \item If $f$ is smooth then either $X$ is isomorphic to a product of curves $C\times C'$, where $C'$ is a smooth curve of genus $0$ without $k$-points, or $X$ is isomorphic to $\mathbb{P}_C(\mathcal{E})$ for some rank $2$ vector bundle $\mathcal{E}$ on $C$.
    \item if $f$ is not smooth, then
    \[
        \textup{Pic}(X)=f^*\textup{Pic}(C)+\Z[-K_X].
    \]
    Let $r$ be the number of non-smooth fibers: then $K_X^2=8-r$.
\end{enumerate}
\end{theorem}

In this subsection we want to use the above properties to prove an analogous theorem \ref{theoremdelPezzo} but for minimal rational conic bundles.

\begin{theorem}\label{thmdefinitestablebirofconicbundles}
    Let $X$ be a rational conic bundle as in theorem \ref{classificationofgeomratsurfaces}.
    Let $\textup{ind}(X)$ be the index of $X$.
    Then for $n >>0$ the variety $\Hilb^n_X$ is stably birational to $\Hilb^{n+\textup{ind}(X)}_X$.
\end{theorem}

We split the proof in two parts, depending on whether the map $f$ is smooth.

\begin{proof}[Case $f$ smooth]
    The projective bundle $\mathbb{P}_C(\mathcal{E})$ is birational to $C\times \mathbb{P}^1_k$.
    Since the index is a birational invariant for projective smooth varieties (see \cite[Prop. 1.4]{Merku} or \cite[Prop. 6.8]{Gabber}), without loss of generality we can therefore study the case when $X$ is the product of two smooth curves of genus $0$.
    However, any surface of this type is a del Pezzo surface since the surface $X_{\overline{k}}\cong \mathbb{P}^1_{\overline{k}}\times \mathbb{P}^1_{\overline{k}}$ is.
    In particular, theorem \ref{theoremdelPezzo} applies.
\end{proof}

Therefore, we are left with the case $f$ singular. 

\begin{remark} 
    In this case, we cannot automatically apply the previous strategy since there are rational conic bundles which are not birational to any del Pezzo surface.
    
    For instance, let $k$ be an (infinite perfect) field and $X$ be a minimal rational conic bundle of negative degree $K_X^2<0$.  
    Assume by contradiction that $X$ is birational to a del Pezzo $X'$.
    Then consider a birational morphism $f:X'\longrightarrow X'_{\textup{min}}$ where $X'_{\textup{min}}$ is minimal over $k$.
    Expressing the morphism $f$ as a sequence of blow-ups, by \cite[Cor. 2.8]{Hassett} we have that $X'_{\textup{min}}$ is still del Pezzo.
    By \cite[Prop. 2.4]{Sko} the degree $K_{X'_{\textup{min}}}^2$ is equal to $K_X^2$, so in particular it is negative, which is not possible.
\end{remark}

For del Pezzo's we have used blow-ups of the surface in order to conclude.
Now we can use the information on the structure of $X$ and on its Picard group $\Pic(X)$ given by Iskovskih theorem. \\

We aim to find a family of line bundles $\{ \mathcal{L}_m \}_{m}$ on $X$ (and a $0$-cycle $Q$) such that for any $n>>0$ there exists an $m(n)$ such that $(X, \mathcal{L}_{m(n)}, Q)$ is a suitable triple satisfying the \hyperlink{assumptions}{assumptions}  with respect to that $n$.

\begin{remark}
    As for del Pezzo surfaces, the index of a rational conic bundle $X$ is the greatest common divisor of its geometrically reduced closed points $Q_i$ (see again \cite[Thm. 9.2]{Gabber}).
    In what follows, $Q$ will be any geometrically reduced closed point on $X$, so that the pair $(X,Q)$ is suitable.
\end{remark}
Let us find suitable line bundles, namely ample line bundles on $X$ in the case when $f$ is singular.
\begin{proposition}\label{propalotofamplelinbundleonconicbundles}
    Let $X$ be a rational conic bundle with singular $f$ as in theorem \ref{classificationofgeomratsurfaces}.
    Let $F$ be a smooth fiber over a closed point $t\in C$ of minimal degree.
    There exists a positive integer $m$ such that for any $a>>0$ the divisor $-mK_X+aF$ is ample.
\end{proposition}

Before dealing with the proof, let us prove the following lemma.

\begin{lemma}\label{lemmaDamplethenDF}
Let $X$ be a rational conic bundle as in theorem \ref{classificationofgeomratsurfaces}.
If $D$ is an ample divisor, then $D+F$ is ample as well.
\end{lemma}
        
\begin{proof}
The sum of an ample divisor $D$ with a globally generated divisor is still ample: therefore it is enough to prove that $\mathcal{O}_X(F)$ is generated by global sections.

The degree of $\mathcal{O}_X(F) |_F$ is $F^2$ which is $0$ being a fiber of $f$.
Therefore, since $\Pic(F)\cong \mathbb{Z}$ (because $F$ is a plane conic), the sheaf $\mathcal{O}_X(F) |_F$ is trivial.

From the exact sequence
\[
    0\longrightarrow \mathcal{O}_X
    \longrightarrow \mathcal{O}_X(F)
    \longrightarrow i_{F, *}(\mathcal{O}_X(F)|_F )
    \longrightarrow 0
\]
we have
\[
    H^0(X, \mathcal{O}_X(F))\cong H^0(X,\mathcal{O}_X)\oplus H^0(F, \mathcal{O}_X(F)|_F)\cong H^0(X,\mathcal{O}_X)\oplus H^0(F, \mathcal{O}_F)
\]
(recall that $H^1(X,\mathcal{O}_X)=0$).
Since both $\mathcal{O}_X$ and $\mathcal{O}_F$ are generated by global sections, the sheaf $\mathcal{O}_X(F)$ is globally generated as well.
\end{proof}

    \begin{proof}[Proof of Proposition \ref{propalotofamplelinbundleonconicbundles}]
        Since $X$ is projective, there exists an ample line bundle $\mathcal{L}$ on it. 
        Since any closed point of minimal degree generates $\Pic(C)$, by item $(3)$ of theorem \ref{descriptionofconicbundles} we have that $F$ and $-K_X$ generates $\Pic(X)$.
        Therefore $c_1(\mathcal{L})$ has the form $-mK_X+ aF$ with $m$ and $a$ integers.\\

        By Nakai-Moishezon Criterion (see \cite[Thm.\, V.1.10]{Harty}), a necessary condition for $\mathcal{L}$ to be ample is that $c_1(\mathcal{L}) \cdot F >0$.
        Since $F^2=0$ then this is equivalent to have 
        \[
            (-mK_X+aF) \cdot F = -m K_X \cdot F > 0.
        \]
        By adjunction formula for $F\subset X$ 
        \[
             \deg (K_F) = K_X \cdot F + F^2 = K_X\cdot F .
        \]
        Because $F$ is a conic (over $k(t)$) the degree of $K_F$ (over $k$) is $-2\deg k(t)$.
        Thus we must have 
        \[
            -m K_X \cdot F = 2 m \deg k(t) > 0.
        \]
        Hence the integer $m$ must be positive.\\

        Therefore there exists an $a_0$ such that $-mK_X+a_0F$ is ample and $m>0$.
        By lemma \ref{lemmaDamplethenDF}, we get that for any $a\geq a_0$ the divisor $-mK_X+aF$ is ample as well.
    \end{proof}

    Now that we have many ample line bundles, let us check which of them can lead to a triple which satisfies the \hyperlink{assumptions}{assumptions}.
    
    \begin{corollary}
       Let $\mathcal{L}$ be an ample line bundle $\mathcal{O}_X(-mK_X+aF)$ such that $a > \frac{m(r-8)}{2 \deg k(t)}$.
       Then for $e>>0$ the triple $(X, \mathcal{O}_X(-mK_X+aF)^{\otimes e}, Q)$ satisfies the \hyperlink{assumptions}{assumptions} with respect to any $n\in I_e$ (defined in theorem \ref{theoremgap1}).
    \end{corollary}

    \begin{proof}
        By proposition \ref{propositionc1(L)KXnegative}, it is enough to find an ample line bundle $\mathcal{L}$ on $X$ such that $c_1(\mathcal{L})\cdot K_X<0$.
        By item $(3)$ in theorem \ref{descriptionofconicbundles} we have $K_X^2=(8-r)$.
        Therefore
        \[
        c_1(\mathcal{L})\cdot K_X = (-mK_X+aF)\cdot K_X = -m(8-r)-2a \deg k(t)
        \]
        This is negative if and only if $a  > \frac{m(r-8)}{2 \deg k(t)}$.
    \end{proof}

    \begin{proposition}\label{propinductionona}
        Let $\mathcal{L}$ be an ample line bundle $\mathcal{O}_X(-mK_X+aF)$ such that $a > \frac{m(r-8)}{2 \deg k(t)}$.
        Assume that
        \begin{enumerate}
            \item the triple $(X,\mathcal{L}^{\otimes e},Q)$ satisfies the \hyperlink{assumptions}{assumptions} with respect to $n_0\in I_e$,
            \item $e\geq e_g+e_{va}$ where $\mathcal{L}^{\otimes e_g}\otimes \mathcal{I}_Q$ is globally generated and $\mathcal{L}^{\otimes e_{va}}$ very ample, and
            \item $e>>0$ so that $H^i(X, \mathcal{L}^{\otimes e})=0$  $\forall i>0$ and $-e\, c_1(\mathcal{L})\cdot K_X-1-2\deg(Q) \geq 0$.
        \end{enumerate}
        Then for any $b\in \mathbb{N}$ the triple $(X,\mathcal{L}^{\otimes e}\otimes \mathcal{O}_X(bF),Q)$ satisfies the \hyperlink{assumptions}{assumptions} with respect to any integer $n_b$ inside the non-empty interval
        \[
            I_{(e,b)}\coloneqq \left[\frac{(e\, c_1(\mathcal{L})+bF)^2 + (e\, c_1(\mathcal{L}) +bF )\cdot K_X}{2}+1,  \dim_k H^0(X, \mathcal{L}^{\otimes e} \otimes \mathcal{O}_X(bF) ) - 1 -2\deg(Q) \right] .
        \]
    \end{proposition}

    \begin{proof}
        Let us check the \hyperlink{assumptions}{assumptions} one by one.\\
        
        \textbf{Assumption 1.} Since the composite map 
            \[
                H^0(X, \mathcal{L}^{\otimes e})\xhookrightarrow{} H^0(X,\mathcal{L}^{\otimes e}\otimes \mathcal{O}_X(bF))\longrightarrow H^0(Q, \mathcal{O}_Q)
            \] 
            is surjective, then 
            \[
                H^0(X,\mathcal{L}^{\otimes e}\otimes \mathcal{O}_X(bF))\longrightarrow H^0(Q, \mathcal{O}_Q)
            \]
            is surjective as well.\\

        \textbf{Assumption 2.} \begin{equation}\label{assumpt2forbF}
            n_b \leq  \dim_k H^0(X, \mathcal{L}^{\otimes e} \otimes \mathcal{O}_X(bF) ) - 1 -2\deg(Q).
        \end{equation}
    
        \textbf{Assumption 3.} Since both $\mathcal{L}^{\otimes e}\otimes \mathcal{I}_Q$ and $\mathcal{O}_X(bF)$ are generated by global sections (see proof of lemma \ref{lemmaDamplethenDF}), the product $\mathcal{L}^{\otimes e}\otimes \mathcal{O}_X(bF)\otimes \mathcal{I}_Q$ is globally generated as well.\\
            
        \textbf{Assumption 4.}  Since $\mathcal{L}^{\otimes e_{va}}$ is very ample and $\mathcal{O}_X(bF)$ is globally generated, then $\mathcal{L}^{\otimes e_{va}}\otimes \mathcal{O}_X(bF)$ is very ample as well. 
        Since $\mathcal{L}^{\otimes e_g}\otimes \mathcal{O}_X(bF)\otimes \mathcal{I}_Q$ is globally generated, we can do as done in proof of theorem \ref{theoremgap1} to infer that assumption \hyperlink{assumptions}{(4)} is satisfied as well.\\
            
        \textbf{Assumption 5.} 
        \begin{equation}\label{assump5forbF}
            \frac{(e\, c_1(\mathcal{L})+bF)^2 + (e\, c_1(\mathcal{L}) +bF )\cdot K_X}{2}+1\leq n_b.
        \end{equation}
        
        Therefore it remains to check that the intervals $I_{(e,b)}$ are non-empty. 
        Equivalently, we need to prove that the difference
        \begin{align}\label{differencefornb}
           \dim_k H^0(X, \mathcal{L}^{\otimes e} \otimes \mathcal{O}_X(bF) ) - 1 -2\deg(Q) 
            - \frac{(e\, c_1(\mathcal{L})+bF)^2 + (e\, c_1(\mathcal{L}) +bF )\cdot K_X}{2} -1
        \end{align}
        is non-negative.
        
        In order to do so, we want to express $\dim_k H^0(X, \mathcal{L}^{\otimes e} \otimes \mathcal{O}_X(bF) )$ in terms of the Euler-characteristic $\chi( \mathcal{L}^{\otimes e} \otimes \mathcal{O}_X(bF) )$.

        \begin{claim*}
        Let $b\in \mathbb{N}$.
        If the groups $H^i(X, \mathcal{L}^{\otimes e}\otimes \mathcal{O}_X(bF))$ are trivial for any $i\geq 1$, then $H^i(X, \mathcal{L}^{\otimes e}\otimes \mathcal{O}_X((b+1)F) )$ are trivial as well.
        \end{claim*}
            
        \begin{proof}[Proof of the Claim]
        Thanks to the exact sequence
        \[
            0\longrightarrow \mathcal{L}^{\otimes e} \otimes \mathcal{O}_X(bF)
            \longrightarrow \mathcal{L}^{\otimes e} \otimes \mathcal{O}_X( (b+1)F)
            \longrightarrow \mathcal{L}^{\otimes e}|_F \otimes \mathcal{O}_F((b+1)F)
            \longrightarrow 0
        \]
        we get (for $i\geq 2$)
        \begin{align*}
            0
            \longrightarrow H^1(X, \mathcal{L}^{\otimes e} \otimes \mathcal{O}_X((b+1)F) )
            \longrightarrow & H^1(F, \mathcal{L}^{\otimes e}|_F \otimes \mathcal{O}_F((b+1)F) ) \longrightarrow 0 \dots
            \\
            \cdots \longrightarrow 0
            \longrightarrow & H^i(X, \mathcal{L}^{\otimes e} \otimes \mathcal{O}_X((b+1)F) ) \longrightarrow 0 \longrightarrow
            \cdots .
        \end{align*}
        In particular, $H^i(X, \mathcal{L}^{\otimes e}\otimes \mathcal{O}_X((b+1)F) )$ is trivial for $i\geq 2$.

        The sheaf $\mathcal{L}^{\otimes e}|_F$ is isomorphic to
        \[
            \mathcal{O}_X(-emK_X+eaF) |_F \cong (\omega_X^{-1} |_F)^{\otimes me}\otimes (\mathcal{O}_X(F)|_F)^{\otimes ea} \cong \omega_F^{-\otimes me} \cong \mathcal{O}_F(2me)
        \]
        (because $\mathcal{O}_F(F)$ is trivial: see proof of lemma \ref{lemmaDamplethenDF}).
        Then the cohomology group $H^1(F, \mathcal{L}^{\otimes e}|_F\otimes \mathcal{O}_F((b+1)F))$ is trivial as well.
        \end{proof}

        Thanks to the claim, for any positive $b\in \mathbb{N}$ we have
        \[
            \dim_k H^0(X, \mathcal{L}^{\otimes e} \otimes \mathcal{O}_X(bF) )= \chi(\mathcal{L}^{\otimes e} \otimes \mathcal{O}_X(bF) ) 
            =
            \frac{(e\, c_1(\mathcal{L})+bF)^2-(e\, c_1(\mathcal{L})+bF)\cdot K_X}{2} +1.
        \]
        Hence the difference \ref{differencefornb} becomes
        \[
            \frac{(e\, c_1(\mathcal{L})+bF)^2-(e\, c_1(\mathcal{L})+bF)\cdot K_X}{2} - 2\deg Q  - \frac{(e\, c_1(\mathcal{L})+bF)^2 + (e\, c_1(\mathcal{L})+bF)\cdot K_X}{2} -1 
        \]
        which is equal to
        \[
            -(e\, c_1(\mathcal{L})+bF)\cdot K_X -1- 2\deg Q = -e\, c_1(\mathcal{L})\cdot K_X + 2b \deg k(t) -1 -2\deg(Q).
        \]
        Therefore as soon as $-e\, c_1(\mathcal{L})\cdot K_X-1-2\deg(Q)$ is non-negative, then the difference is always non-negative.
    \end{proof}

    \begin{proof}[Proof of Theorem \ref{thmdefinitestablebirofconicbundles} with $f$ singular]
    Fix $(m,a)$ a pair of integers such that $m$ is positive, $a> \frac{m(r-8)}{2 \deg k(t)}$ and $\mathcal{L}=\mathcal{O}_X(-mK_X+aF)$ is ample.
    Let $e$ be any integer big enough such that the triple $(X,\mathcal{O}_X(-emK_X+eaF),Q)$ satisfies the \hyperlink{assumptions}{assumptions} (with respect to some $n_0$) and such that $e$ satisfies the hypotheses of proposition \ref{propinductionona}.
    
    In order to conclude the proof, it is enough to check that the intervals $I_{(e,b)}=[\min_{(e,b)},\max_{(e,b)}]$'s for $\mathcal{L}^{\otimes e}\otimes \mathcal{O}_X(bF)$
    \[
    \left[\frac{(e\, c_1(\mathcal{L})+bF)^2 + (e\, c_1(\mathcal{L}) +bF )\cdot K_X}{2}+1,  \frac{(e\, c_1(\mathcal{L})+bF)^2 - (e\, c_1(\mathcal{L}) +bF )\cdot K_X}{2} -2\deg(Q) \right],
    \]
    provided by proposition \ref{propinductionona}, contain all the integers bigger than some integer $n_Q$.
    The width of the interval $I_{(e,b)}$
    \[
        -e\, c_1(\mathcal{L})\cdot K_X + 2b \deg k(t) -1 -2\deg(Q)
    \]
    grows when $b$ grows.
    Therefore, it is enough to check that for $b$ big enough we have $\max_{(e,b)}+1 \in I_{(e,b+1)}$. 

    For what concerns the inequality $\max_{(e,b)}+1\leq  \max_{(e,b+1)}$, it is always true. 
    Indeed it reads as follows:
    \[
        \frac{(e\, c_1(\mathcal{L})+bF)^2 - (e\, c_1(\mathcal{L}) +bF )\cdot K_X}{2} +1 \leq \frac{(e\, c_1(\mathcal{L})+(b+1)F)^2 - (e\, c_1(\mathcal{L}) +(b+1) F )\cdot K_X}{2}
    \]
    which is equivalent to
    \[
        e\, c_1(\mathcal{L})\cdot F - e \frac{K_X\cdot F}{2} - 1 = 2em \deg k(t) + e\deg k(t) -1 \geq 0.
    \]
    The other inequality, $\min_{(e,b+1)} \leq \max_{(e,b)}+1$, is equivalent to
    \[
        \frac{(e\, c_1(\mathcal{L})+(b+1)F)^2 + (e\, c_1(\mathcal{L}) + (b+1)F )\cdot K_X}{2} \leq \frac{(e\, c_1(\mathcal{L})+bF)^2 - (e\, c_1(\mathcal{L}) +bF )\cdot K_X}{2} -2\deg(Q) .
    \]
    After carrying on the computations on both sides, we get
    \begin{align*}
        0 \geq e\, (-mK_X +aF ) \cdot  ( K_X + F ) + \frac{2b+1}{2} F\cdot K_X + 2 \deg(Q) =
        \\
        = e( m(r-8) -2a \deg k(t) +2m \deg k(t)  ) - (2b+1)\deg k(t) +2 \deg(Q).
    \end{align*}
    Therefore it is enough to pick $b$ greater or equal than
    \[
        b \geq  \frac{ \deg (Q)}{\deg k(t)} + e \left( \frac{m(r-8)}{2\deg k(t)} -a +m \right) - \frac{1}{2}.
    \]
    So we get that for any $n\geq n_Q$, we have $\Hilb^{n}_X$ stably birational to $\Hilb^{n+\deg(Q)}_X$.

    Let $Q_1,\dots, Q_r$ be geometrically reduced closed points of $X$ such that 
    \[
        \textup{ind}(X)=\sum_{i=1}^s c_i \deg(Q_i) - \sum_{i=s+1}^r c_i\deg(Q_i)
    \]
    and $c_i>0$.
    Picking $n$ greater or equal than $\max_i n_{Q_i}$, we conclude the proof like in theorem \ref{theoremdelPezzo}.
    \end{proof}

    In particular we have treated all the minimal geometrically rational surfaces.
    This can be summarised in the following corollary.
    
    \begin{corollary}[Theorem \ref{theoremrationalsurfintro}]\label{corollarydefinitestablebir}
        Let $X$ be a geometrically rational suitable surface.
        Let $\textup{ind}(X)$ be the index of $X$.
        Then there exists an integer $n_0$ such that for any $n\geq n_0$ the variety $\Hilb^n_X$ is stably birational to $\Hilb^{n+\textup{ind}(X)}_X$.
    \end{corollary}

\section{Rationality of the Motivic Zeta function}\label{9}
     
Let us now focus on $K_0(Var/k)$, the Grothendieck ring of varieties over $k$: it is the commutative ring generated by isomorphism classes $[Y]$ of separated, reduced schemes of finite type over $k$ modulo the \emph{cut-and-paste} relations
\[
[Y] = [Y\setminus Z] + [Z_{\textup{red}}] \quad \textup{and } Z\xhookrightarrow{} Y \textup{ closed immersion}.
\]
The product operation is the one induced by the product of schemes
\[
    [Y] \cdot [W] = [(Y\times W)_{\textup{red}}].
\] 
Let $R$ be a commutative ring. 

\begin{definition}
    An $R$-valued \emph{motivic measure} on $K_0(Var/k)$ is a ring homomorphism
    \[
        \mu: K_0(Var/k)\longrightarrow R.
    \]
    If $R=K_0(Var/k)$, $\mu=id$ is called the \emph{universal motivic measure}.
\end{definition}

Given a motivic measure $\mu:K_0(Var/k)\longrightarrow R$,  Kapranov defined in \cite{Ka} the $\mu$-\emph{motivic zeta function} of $Y$ to be the formal series
\[
    \zeta_{\mu}(Y,t)= 1+ \sum_{n=1}^{\infty} \mu([\Sym^n_Y])t^n \in R[[t]].
\]
If $\mu$ is the universal motivic measure, we get the so-called \emph{motivic zeta function} $\zeta_{\id}(Y,t)$, denoted by $\zeta_{\textup{mot}}(Y,t)$.
It is a generalization of the Hasse-Weil zeta function.

\begin{proposition}[Prop. 2.9, \cite{Mu}]
    Let $k=\mathbb{F}_q$ be a finite field.
    Consider the motivic measure
    \[
        |\cdot |:K_0(Var/k)\longrightarrow \Z, \quad [Y]\, \mapsto\, |Y(\mathbb{F}_{q})|.
    \]
    Then $\zeta_{|\cdot |}(Y,t)$ coincides with the Hasse-Weil zeta function
    \[
        Z(Y,t)=\exp \left ( \sum_{m\geq 1} \frac{|Y(\mathbb{F}_{q^m})|}{m}t^m \right ).
    \]
\end{proposition}

The Hasse-Weil zeta function $Z(Y,t)$ is known to be a rational function (see \cite{Dwo}).
Therefore, it is natural to ask for which motivic measures $\mu$ and variety $Y$, the zeta function $\zeta_{\mu}(Y,t)$ is rational as well.

\begin{definition}
    Let $R$ be a commutative ring: we say that $f \in R[[t]]$ is rational if there exist polynomials $p(t), q(t) \in R[t]$ such that $q(t)$ is invertible in $R[[t]]$ and $f(t) = \frac{p(t)}{q(t)}$.
\end{definition}

For curves we have the following  rationality result. 

\begin{theorem}
    \cite[Thm. 1.1.9]{Ka}
    If $C$ is a smooth, geometrically connected, projective curve, then $\zeta_{\textup{mot}}(C,t)$ is a rational function.
\end{theorem}

In \cite{Ka} it was conjectured that the motivic zeta function was rational for any algebraic variety $Y$. 
However Larsen and Lunts provided a counterexample to this conjecture for $k=\C$ in \cite{LL} and they actually gave a sufficient and necessary condition for rationality in the case of complex surfaces (in \cite{LL1}).

\begin{theorem}\cite[Thm. 1.1]{LL1}
    Let $k=\C$ and let $X$ be a smooth connected projective complex surface.
    Then the motivic zeta function $\zeta_{\textup{mot}}(X,t)$ is rational if and only if $X$ has negative Kodaira dimension.
\end{theorem}

It turns out that having negative Kodaira dimension is a sufficient condition on any algebraically closed fields.

\begin{proposition}
    \cite[Prop. 4.3]{Mu}
    Let $X$ be a smooth connected projective surface over an algebraically closed field $k$. 
    If $X$ has negative Kodaira dimension, then $\zeta_{\textup{mot}}(X, t)$ is a rational function.
\end{proposition}

We now aim to study what happens when we drop the assumption on the field to be algebraically closed.

In order to do so, let us first study the rationality of a closely related zeta function, the $\mu$-\emph{motivic Hilbert zeta function}.

\begin{definition}
    Let $\mu$ be an $R$-valued motivic measure.
    Let $X$ be a geometrically rational surface.
    The $\mu$-\emph{motivic Hilbert zeta function} of $X$ is the power series
    \[
        \zeta_{\mu}^{\Hilb}(X,t)=1+\sum_{n=1}^{\infty} \mu([\Hilb^n_X])t^n\ \in R[[t]].
    \]
\end{definition}

Denote by $\mathbb{L}$ the class of $[\mathbb{A}^1_k]$ in $K_0(Var/k)$. 
Consider the motivic measure defined modding out by the ideal generated by $\mathbb{L}$:
\[
    \mu_{\mathbb{L}}: K_0(Var/k)\longrightarrow K_0(Var/k)/( \mathbb{L} ).
\]

\begin{hypothesis}[Weak Factorization property]\label{hp2}
    Let $f$ be a birational map $Y \overset{f}{\dashrightarrow} Y'$ between smooth projective varieties: then it can be realized as a composition of blow-ups and blow-downs along smooth irreducible centres on smooth projective varieties.
\end{hypothesis}

\begin{remark}
    If char$(k)=0$ then this hypothesis is satisfied \cite[Thm. 0.0.1]{Wl}.
\end{remark}

Here onward, let us assume that the field $k$ satisfies the weak factorization property. 
The usefulness of assuming weak factorization property is that birational smooth projective varieties are identified in the quotient $K_0(Var/k)/(\mathbb{L})$.

\begin{remark}
    Let $f$ be a birational map $Y \overset{f}{\dashrightarrow} Y'$ between smooth projective varieties: then $[Y]=[Y']$ in $K_0(Var/k)/ (\mathbb{L})$.
\end{remark}

\begin{proof}
Decomposing $f$ into a sequence of blow-ups and downs, we can assume that $f$ is a blow-up along a smooth centre.
Since $\mathbb{L}=0$ then
\[
    [\mathbb{A}^n_k] = \mathbb{L}^n=0,\quad \ \textup{and} \ [\mathbb{P}^n_k] = 1 + \mathbb{L} + \dots + \mathbb{L}^n =1.
\]
Therefore if $\mathcal{E}$ is a vector bundle on $Y$ of rank $n$ then $[\mathbb{P}(\mathcal{E})]=[Y]$. 
If we blow-up a smooth connected projective variety $Y$ along a smooth connected closed subvariety $Z$ then 
\[
    [Bl_Z Y] - [Y] = [\mathbb{P}( N_Z Y)] - [Z] =0 \implies [Bl_Z Y] = [Y],
\]
where $N_Z Y$ is the normal bundle of $Z$ in $Y$.
\end{proof}

\begin{proposition}\label{proprationalityformotivicHilb}
    Let $k$ be a field for which weak resolution theorem holds true.
    Let $X$ be a geometrically rational, connected smooth projective surface.
    Then the $\mu_{\mathbb{L}}$-motivic Hilbert zeta function 
    \[
        \zeta_{\mu}^{\Hilb}(X,t)
        =
        1+\sum_{ n=1}^{\infty} \mu_{\mathbb{L}}( [ \Hilb^n_X ] ) t^n \in K_0(Var/k)/(\mathbb{L})[[t]]
    \]
    is a rational function.
\end{proposition}

\begin{proof}
    By corollary \ref{corollarydefinitestablebir}, since $X$ is geometrically rational there is an $n_0$ (which depends only on the birational type of $X$) such that for any $n\geq n_0$ we have $\Hilb^n_X $ stably birational to $\Hilb^{n+\textup{ind}(X)}_X$.
    Therefore the coefficients of $\zeta_{\mathbb{L}}^{\Hilb}(X,t)$ are eventually periodic of period a divisor of $\textup{ind}(X)$, which means that the series can be written as 
    \begin{align*}
        1+\sum_{ n=1}^{\infty} \mu_{\mathbb{Z}} ([ \Hilb^n_X ]) t^n & = 1 + \mu_{\mathbb{L}}([X]) t + \dots + \mu_{\mathbb{L}}( [ \Hilb^{n_0-1}_X ]) t^{n_0-1} + \sum_{j = 0 }^{\infty} \mu_{\mathbb{L}}([ \Hilb^{n_0+j}_X ]) t^{n_0+j}
        \\
        & = P_X(t) + t^{n_0} \sum_{j = 0 }^{\infty} 
        t^{j\cdot \textup{ind}(X)} \left( \mu_{\mathbb{L}}  ( [\Hilb^{n_0}_X] )+ \dots + \mu_{\mathbb{L}} ( [ \Hilb^{n_0+\textup{ind}(X)-1}_X ])  t^{\textup{ind}(X)-1} \right)
        \\
        & = P_X(t)
        + t^{n_0} \frac{ \mu_{\mathbb{L}} ([ \Hilb^{n_0}_X ]) + \dots + \mu_{\mathbb{L}}([ \Hilb^{n_0+\textup{ind}(X)-1}_X ]) t^{\textup{ind}(X)-1}  }{1 - t^{\textup{ind}(X)}}
    \end{align*}
    where $P_X(t)$ is the polynomial $1 + \mu_{\mathbb{L}}([X]) t + \dots + \mu_{\mathbb{L}} ( [\Hilb^{n_0-1}_X] ) t^{n_0-1}$. 
\end{proof}

In order to recover information on the $\mu_{\mathbb{L}}$-motivic zeta function, we now need to compare the classes $\mu_{\mathbb{L}}([\Sym^n_X])$ with $\mu_{\mathbb{L}}([\Hilb^n_X])$ in $K_0(Var/k)/(\mathbb{L})$.\\

The following theorem is due to G\"ottsche.
It was proven under the assumption that $k$ is a field of characteristic zero and algebraically closed, but the latter assumption is not used and can be dropped.

\begin{theorem}\cite[Thm. 1.1]{Go}
Let $k$ be a field of characteristic zero.
Let $\alpha=(m_1,\dots,m_r)$ be a partition of $n$.
Denote by $a_i$ is the number of times for which $m_j=i$.
Let $|\alpha|=\sum_{i=1}^n a_i$. 
Then the class $[\Hilb^n_X]$ in $K_0(Var/k)$ is equal to
\[
    \sum_{\alpha \dashv n} \, [ \Sym^{a_1}_X\times \dots \times \Sym^{a_n}_X \times \mathbb{A}^{n-|\alpha|}_k].
\]
\end{theorem}

Therefore, here onwards, we will assume that $k$ has characteristic zero. In particular, $k$ satisfies the weak factorization theorem.

\begin{corollary}
    Let $k$ be a field of characteristic zero: then in the quotient ring $K_0(Var/k)/(\mathbb{L})$ we have that 
    \[
        \mu_{\mathbb{L}} ([\Sym^n_X] ) = \mu_{\mathbb{L}} ([\Hilb^n_X] ) .
    \]
\end{corollary}

\begin{proof}
    Since $\mathbb{L}$ is sent to zero, then $\mu_{\mathbb{L}}( [\mathbb{A}^{n-|\alpha|}_k] )$ is zero if $|\alpha|$ is not $n$.
    So the only non-zero term of the G\"ottsche's formula in $K_0(Var/k)/(\mathbb{L})$ is the one associated to the partition $\alpha=(1,\dots,1)$ (so $a_1=n$). 
    Hence
    \[
        \mu_{\mathbb{L}} ([\Sym^{n}_X ]) = \mu_{\mathbb{L}} ([\Hilb^n_X] ) .
    \]
\end{proof}

In particular in $K_0(Var/k)/(\mathbb{L})$, the motivic Hilbert zeta function coincides with the motivic zeta function
\[
    1+\sum_{n=1}^{\infty} \mu_{\mathbb{L}}( [ \Sym^n_X ]) t^n.
\]

\begin{corollary}[Theorem \ref{rationalityofstablmotiviczeta}]
    Let $k$ be a field of characteristic zero: if $X$ is a geometrically rational connected smooth projective surface, then the $\mu_{\mathbb{L}}$-motivic zeta function is rational. 
\end{corollary}

\begin{remark}
    By the proof above, notice that in order to conclude the rationality of $\zeta_{\mathbb{L}}(X,t)$ (for a geometrically rational surface $X$), it is sufficient for $k$ to have the weak factorization property and to be such that $\mu_{\mathbb{L}}([\Sym^n_X])=\mu_{\mathbb{L}}([\Hilb^n_X])$ for any $n>0$.
\end{remark}

\nocite{*}
\printbibliography

\end{document}